\documentclass[11pt]{amsart}
\usepackage{amssymb}
\usepackage{xcolor} % A package to add color.
\usepackage{fullpage} % Sets all margins to 1 in.  
\usepackage{amsmath}

\usepackage{amsthm}
\usepackage{graphicx}
\usepackage{placeins}
\usepackage{yhmath}
\usepackage{mathrsfs}
\usepackage{subcaption}

\usepackage[numbers,sort&compress]{natbib}
\usepackage[colorlinks=true,citecolor=blue,linkcolor=blue, urlcolor=blue]{hyperref}

\usepackage{enumitem}

\setlist[enumerate]{leftmargin=1.8em}
\setlist[itemize]{leftmargin=1.8em}
\usepackage{mathtools}
\graphicspath{{figures/}}

\newtheorem{theorem}{Theorem}[section]

\newtheorem{lemma}[theorem]{Lemma}

\newtheorem{proposition}[theorem]{Proposition}

\theoremstyle{definition}

\theoremstyle{remark}
\newtheorem{remark}[theorem]{Remark}

\numberwithin{equation}{section}
\newcommand{\nrm}[1]{\Vert#1\Vert}

\newcommand{\nnrm}[1]{{\vert\kern-0.25ex\vert\kern-0.25ex\vert #1 
		\vert\kern-0.25ex\vert\kern-0.25ex\vert}}

\newcommand{\lmb}{\lambda}

\newcommand{\zt}{\zeta}

\newcommand{\bbR}{\mathbb R}

\begin{document}
	\title[Stability of dipoles]
    {A Variational Family of Traveling Vortex Dipoles: Transition, Uniqueness, and Stability}
    
    %{Variational uniqueness and stability of traveling  vortex dipoles for the 2D Euler Equations}

	\author{Ken Abe}
	\address{Department of Mathematics, Graduate School of Science, Osaka Metropolitan University, 3-3-138 Sugimoto, Sumiyoshi-ku Osaka, 558-8585, Japan.}
	\email{kabe@omu.ac.jp}
	
	\author{In-Jee Jeong}
	\address{School of Mathematics, Korea Institute for Advanced Study, Hoegi-ro 85, Seoul 02445, Republic of Korea.}
	\email{ijeong@kias.re.kr}
	
	\author{Guolin Qin}
	\address{State Key Laboratory of Mathematical Sciences, Academy of Mathematics and Systems Science, Chinese Academy of Sciences, 100190 Beijing, P.R. China}
	\email{qinguolin18@mails.ucas.ac.cn}
	
	\author{Weicheng Zhan}
	\address{School of Mathematical Sciences, Xiamen University, Xiamen, Fujian, 361005, P.R. China }
	\email{zhanweicheng@amss.ac.cn}	

\begin{abstract}
We provide a unified description of a stable variational family of traveling vortex dipoles for the two-dimensional incompressible Euler equations, connecting the classical Lamb dipole to asymptotically radial dipoles in the large-impulse regime. We show that this family undergoes a sharp transition from the Lamb dipole, where the mass constraint is inactive, to non-explicit traveling waves for which the mass constraint becomes active, and determine the critical value exactly. We further prove that the variational solutions are unique up to translation both near the Lamb dipole and in the large-impulse regime. As a consequence, the corresponding individual traveling waves are orbitally stable. Our results provide a unified picture of these two distinct asymptotic regimes within a single variational family of traveling vortex dipoles.
\end{abstract}

	\maketitle

	\tableofcontents

\section{Introduction}

Radial and shear flows constitute two fundamental classes of steady solutions to the two-dimensional incompressible Euler equations. Starting from these canonical structures, a substantial body of recent work has investigated the existence, rigidity, and flexibility of steady Euler flows in various settings; see, for instance, \cite{CS12,CDG21,GPSY21,Ruiz23,DG24,EFR24,EHSX26,EH26}. Another fundamental class of coherent structures, particularly from the physical viewpoint, is given by steadily translating \textit{vortex dipoles}, 
consisting of two vortices of opposite signs moving at a constant speed. They are observed in laboratory and numerical studies \cite{VV98,FV94,VF89,Afan,VAF} and provide a model for the dynamics of interacting vortex pairs \cite{Leweke16}. These vortex dipoles are effectively described by traveling wave solutions of the two-dimensional vorticity equation

	%A fundamental problem in the study of two-dimensional incompressible inviscid flows is the existence and stability of coherent vortex structures. Among the most important examples are steadily translating \textit{vortex dipoles}, consisting of two vortices of opposite signs moving at a constant speed. They are observed in laboratory and numerical studies \cite{VV98,FV94,VF89,Afan,VAF} and provide a model for dynamics of interacting vortex pairs \cite{Leweke16}. These vortex dipoles are effectively described by traveling wave solutions of the two-dimensional vorticity equation

	\begin{equation}\label{eq:2D-Euler}
		\partial_t\omega+u\cdot\nabla\omega=0,
		\qquad u=\nabla^\perp\Delta^{-1}_{\mathbb{R}^{2}}\omega,
	\end{equation}  
	where $\nabla^\perp=(-\partial_{x_2},\partial_{x_1})^\top$.
	
	The Chaplygin--Lamb (or simply Lamb) dipole is the classical example with explicit formula
	\cite{Lamb,Chap1903,MV94}. Other steadily translating pairs were found
	through various approaches including free-boundary methods, numerical continuation, and
	bifurcation; see for instance, \cite{Norbury75,Pierrehumbert80,Tu83a,Tu83b,GarciaHaziot23}. 
	A separate line of work, beginning with ideas of Kelvin, Arnold, and Benjamin
	\cite{Kelvin1875,Arnold66,Benjamin76}, employs the variational approach of finding the vorticity configuration which maximizes the kinetic energy under several constraints including the hydrodynamic impulse. This approach has the significant advantage that once suitable compactness of the set of maximizers is established, nonlinear stability follows together with existence of traveling waves \cite{BNL13}. However, as already remarked in \cite{BNL13}, this gives stability for the whole set of maximizers, rather than for a specific traveling wave. In other words, \textit{uniqueness} of the maximizer for the variational problem is essential for deducing the stability of a traveling wave. In the special case of the Lamb dipole, uniqueness is
	available for several variational formulations
	\cite{Burton96,Burton05b,AC2019,ACJ}. Unfortunately, for traveling waves such uniqueness is largely open beyond this explicit solution, except for the concentrated patch regime
	\cite{CQZZ25}.

    %The purpose of this paper is to prove this rigidity for a family of traveling waves in two limiting regimes, where neither the maximizer nor its vortex-free boundary is explicit.
	 
\subsection{Existence of traveling waves}

    In this paper, we describe a variational family of traveling vortex dipoles connecting the classical Lamb dipole to asymptotically radial vortex dipoles appearing in the large-impulse regime. Our variational problem may be viewed as a coarse-grained version of the classical variational viewpoint of Kelvin, Arnold, and Benjamin: instead of fixing the full vorticity distribution, we retain only a few integral constraints. A central question is whether such coarse information is nevertheless sufficient to select a unique traveling wave, so that stability of the maximizing set can be upgraded to orbital stability of an individual wave. We establish this rigidity in two limiting regimes, where neither the maximizer nor its vortex-free boundary is explicit.

Specifically, we consider the kinetic energy penalized by enstrophy (energy-Casimir functional) with parameter $\lambda>0$,
    \begin{equation}\label{eq:PEF-2}
		E_{2,\lambda}[\omega]
		=E[\omega]-\frac{1}{2\lambda}
		\|\omega\|_{L^2(\mathbb R^2_+)}^2
		=\frac12\int_{\mathbb R^2_+}\omega\mathcal G\omega\,dx
		-\frac{1}{2\lambda}\int_{\mathbb R^2_+}
\omega^{2}\!dx,
	\end{equation}
where 
	\begin{equation*}
		\mathcal{G}\omega(x) = \int_{\mathbb{R}^2_+} G(x,y)\omega(y)\,dy, \qquad G(x,y) = \frac{1}{4\pi} \log\!\left( 1+\frac{4x_2y_2}{|x-y|^2} \right),
	\end{equation*}
	denotes the Dirichlet Laplacian Green operator in the upper half-plane. The impulse is given by 
	\begin{equation}\label{eq:impulse}
		\int_{\mathbb{R}^2_+} x_2\omega(x)\,dx.
\end{equation}
For nonnegative vortices in $\mathbb{R}^{2}_{+}$, we identify their impulse with the $x_2$-weighted $L^{1}$ norm 
	\begin{equation*}
		L^1_*(\mathbb R^2_+)
		=\left\{\omega:\int_{\mathbb R^2_+}x_2|\omega(x)|\,dx<\infty\right\},
		\qquad
		\|\omega\|_{L^1_*(\mathbb R^2_+)}
		=\int_{\mathbb R^2_+}x_2|\omega(x)|\,dx .
\end{equation*}
We consider maximization of the energy-Casimir \eqref{eq:PEF-2} in the class of vorticities odd in $x_{2}$, nonnegative in the upper half-plane $\mathbb{R}^2_+ = \{x=(x_1,x_2)\in\mathbb{R}^2:\ x_2>0\}$ under the constraints on their mass and impulse:
\begin{align}
		I_{\mu,\nu,\lambda}
		&=\sup_{\omega\in  K_{\mu,\nu}}E_{2,\lambda}[\omega],\label{eq: VP-2}\\
        			 K_{\mu,\nu}&=\left\{\omega\in L^2\cap L^1\cap L^1_*(\mathbb R^2_+):
			\omega\geq0,\ \|\omega\|_{L^1(\mathbb R^2_+)}\leq\nu,\
			\|\omega\|_{L^1_*(\mathbb R^2_+)}=\mu\right\}. \label{eq:adm-mass-2}
\end{align}
We denote by
\begin{align}
S_{\mu,\nu,\lambda}
=
\left\{
\omega\in K_{\mu,\nu}:
E_{2,\lambda}[\omega]=I_{\mu,\nu,\lambda}
\right\}  \label{eq:maximizerset}
\end{align}
the set of maximizers of \eqref{eq: VP-2}. Here, we use the norm $\nrm{f}_{X \cap Y} = \max\{\nrm{f}_{X},\nrm{f}_{Y}\}$ for two spaces $X$ and $Y$. Maximizers of the variational problem \eqref{eq: VP-2} provide symmetric traveling wave solutions to \eqref{eq:2D-Euler} in $\mathbb{R}^{2}$ with some velocity $(W,0)$, cf. \cite{Norbury75}. The existence part of the following result can be found in \cite{AC2019}.

\begin{theorem}[Existence]\label{thm:nearly Lamb}
				Let $\lambda,\mu,\nu>0$. Then the variational problem
\eqref{eq: VP-2} admits a maximizer. Moreover, every $\omega\in S_{\mu,\nu,\lambda}$ has compact support and is Steiner symmetric with respect to the
				$x_2$-axis up to horizontal translations and satisfies 
\begin{equation}\label{eq:near-Lamb-fixed-EL}
\omega(x)=\lambda \bigl(\mathcal G\omega(x)
						-W x_2-\gamma \bigr)_+,
						\qquad x\in\mathbb R^2_+,
\end{equation}
for some constants $W>0$ and $\gamma\geq 0$.
\end{theorem}

We refer to the set
\begin{align}
    \Omega=\{x\in\mathbb R^2_+:
    \mathcal G\omega(x)-Wx_2-\gamma>0\}  \label{eq:vortexcore}
\end{align}
as the vortex core of $\omega$. Besides existence, it is proved in \cite{AC2019} that maximizing sequences for \eqref{eq: VP-2} are compact, and the set $S_{\mu,\nu,\lambda}$ is stable in the two-dimensional Euler equations \eqref{eq:2D-Euler}. However, this stability is not equivalent to the stability of an individual traveling wave unless the maximizer is unique. If the impulse is relatively small in the sense that $\kappa=\mu\sqrt{\lambda}/\nu\leq \kappa_0$ for some constant $\kappa_0>0$, the maximizer is the Lamb dipole \cite{AC2019}. Here, the Lamb dipole is a two-parameter family of traveling waves with $\lmb, W>0$, defined by    
	\begin{equation}\label{eq:Lamb} 
		\omega_{\lambda, W}^L(x)= 
		-\frac{2W\sqrt{\lambda}}{J_0(j_{1,1})} J_1(\sqrt{\lambda} r)\mathbf{1}_{[0,  {j_{1,1}}{ \lambda^{-\frac12}}]}(r) \sin\theta,
	\end{equation}
	in polar coordinates $(r,\theta)$, where $J_{m}(r)$ denotes the $m$-th order Bessel function of the first kind and $j_{1,1}=3.8317\cdots$ is the first positive zero of $J_1$.

\subsection{Main results: transition, uniqueness, and stability}           

The variational problem \eqref{eq: VP-2}, after scaling, generates a one-parameter family of traveling vortex dipoles parametrized by \(\kappa\). Our goal is to describe the structure of this family, from the Lamb dipole to the large-impulse regime. Remarkably, this family is selected by a variational principle involving only a few conserved quantities, without prescribing the full vorticity distribution. A central theme of this work is that substantial rigidity and geometric structure survive this coarse-graining. We first identify the exact transition threshold $\kappa_*$ and establish
a sharp dichotomy between the mass-unsaturated and mass-saturated regimes.

\begin{theorem}[Mass transition]\label{thm:threshold}
Let $\lambda,\mu,\nu>0$ and set
\begin{align}
\kappa=\frac{\mu\sqrt{\lambda}}{\nu}.  \label{eq:m}
\end{align}
Then, the following statements hold for the constant 
\begin{align}
\kappa_*
=
-\frac{\pi j_{1,1}^{2}J_{0}(j_{1,1})}
{4\displaystyle\int_{0}^{j_{1,1}} rJ_{1}(r)\,dr}= 1.7622\cdots.  \label{eq:threshold}
\end{align}

\begin{itemize}
\item[(\textrm{A})] (Mass-unsaturated regime) 
If $0<\kappa\leq  \kappa_*$, then $\gamma=0$ and $S_{\mu,\nu,\lambda}=\{\omega^{L}_{\lambda,W_L}(\cdot +(a,0)): a\in \mathbb{R}\}$ for 
\begin{align}
W_L=W_L(\lambda,\mu)=\frac{\mu\lambda}{j_{1,1}^{2}\pi}.  \label{eq:Lambspeed}
\end{align}
Moreover, 
\begin{align}
\|\omega_{\lambda,W_L}^{L}\|_{L^1(\mathbb{R}^{2}_{+})}
\begin{cases}
<\nu, &  0<\kappa<\kappa_*,\\
=\nu, & \kappa=\kappa_*.
\end{cases}
\label{eq:massproperty}
\end{align}
\item[(\textrm{B})] (Mass-saturated regime) If $\kappa_*<\kappa$, then $\gamma>0$ and every $\omega\in S_{\mu,\nu,\lambda}$ satisfies
\begin{equation}\label{eq:supercritical-fixed-mass-properties}
\|\omega\|_{L^1(\mathbb R^2_+)}=\nu,
\qquad
\operatorname{dist}(\operatorname{spt}\omega,
\partial\mathbb R^2_+)>0.
\end{equation}
\end{itemize}
\end{theorem}

\begin{figure}[ht]
		\centering
\includegraphics[width=.50\textwidth]{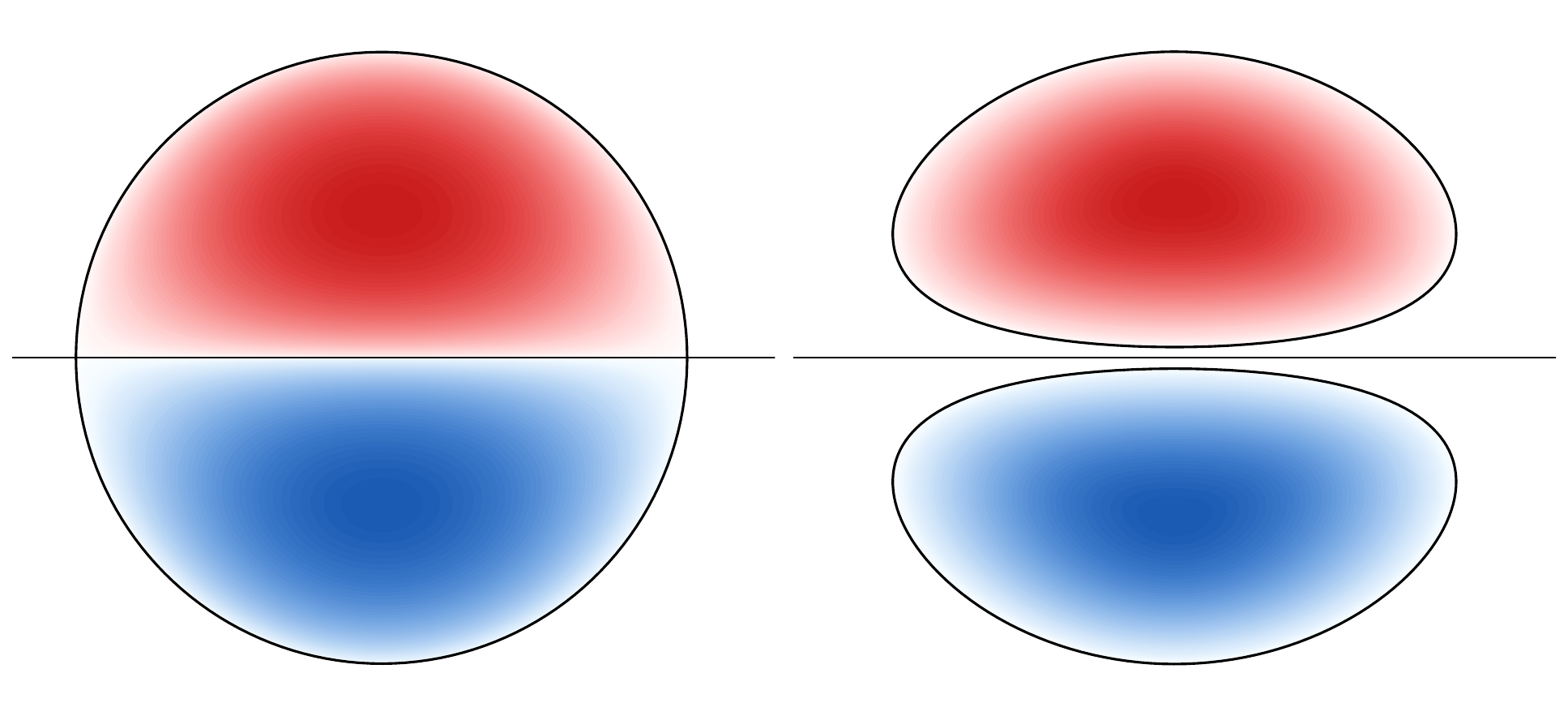}
		\caption{Mass transition from the Lamb dipole (left) to a mass-saturated traveling vortex dipole (right)}
		\label{fig:near-lamb-regime}
\end{figure}

Thus, $\kappa_*$ marks the sharp transition at which the mass constraint
becomes “active” and the explicit Lamb dipole regime gives way to
non-explicit traveling dipoles; see Figure \ref{fig:near-lamb-regime}. In contrast to the mass-unsaturated regime (A), the uniqueness of variational solutions in the mass-saturated regime (B) is considerably more delicate since the resulting vortex profile is no longer radially symmetric. Nevertheless, we show that the variational solution to \eqref{eq: VP-2} is unique in two different regimes: the near-Lamb dipole regime 
\begin{align}
\kappa_*<\kappa\leq \kappa_*+\delta_{*},  \label{eq:near-Lamb-normalized-range}
\end{align}
and the large-impulse regime
\begin{equation}
\kappa_{**}\leq \kappa.  \label{eq:largeimpulse}
\end{equation}
We now formulate the main results of this paper. 

\begin{theorem}[Uniqueness and continuity]\label{thm:unique-near-lamb}
There exist constants $\delta_{*}>0$ and $\kappa_{**}>\kappa_*+\delta_{*}$ such that
\begin{align}
S_{\mu,\nu,\lambda}=\{\omega_{\kappa}(\cdot +(a,0)): a\in \mathbb{R}\}  \label{eq:unique}
\end{align}
for some Steiner symmetric $\omega_\kappa$ and $\kappa\in (0,\kappa_{*}+\delta_{*}]\cup [\kappa_{**},\infty)$. Moreover, the map $\kappa\longmapsto \omega_{\kappa}$ is continuous in $L^{2}\cap L^{1}_{*}\cap L^{\infty}(\mathbb{R}^{2}_{+})$ on each of the intervals $(0,\kappa_{*}+\delta_{*}]$ and $[\kappa_{**},\infty)$.
\end{theorem}

\begin{theorem}[Stability]
\label{thm:stability-near-lamb}
Let $\lambda,\mu,\nu>0$ satisfy \eqref{eq:near-Lamb-normalized-range} or \eqref{eq:largeimpulse}. Then, the traveling wave given by $\omega_{\kappa} \in S_{\mu,\nu,\lambda}$ is stable in the sense that for every $\varepsilon>0$, there exists $\delta>0$ such that, whenever $\zeta_0\in L^{\infty}\cap L^1  \cap L^1_*(\mathbb R^2_+)$ is nonnegative and satisfies $\| \zt_0\|_{L^1}\le\nu$ and 
\begin{equation*}
\inf_{a\in\mathbb R} \nrm{\zeta_0-\omega_{\kappa}(\cdot+(a,0))}_{L^2(\bbR^2_+)} 
+\left|\int_{\mathbb R^2_+}x_2\zeta_0\,dx-\mu\right|
\leq\delta,
\end{equation*}
the unique global-in-time solution $\zeta(t)$ of \eqref{eq:2D-Euler} for initial data $\zeta_0$ satisfies
\begin{equation*}
\inf_{a\in\mathbb R} \nrm{\zeta(t)-\omega_{\kappa}(\cdot+(a,0))}_{L^2\cap L^1_*(\bbR^2_+)} 
\leq\varepsilon,
\qquad t \in \bbR. 
\end{equation*}
\end{theorem} 

\begin{remark}
		The $L^{\infty}$ assumption for $\zt_0$ is used only to guarantee uniqueness of the solution by Yudovich's theorem. Even without this $L^{\infty}$ assumption, one can still get the existence of a global-in-time solution satisfying the same estimate. 
\end{remark}

\begin{figure}[!hb]
\centering
\includegraphics[width=0.93\textwidth]{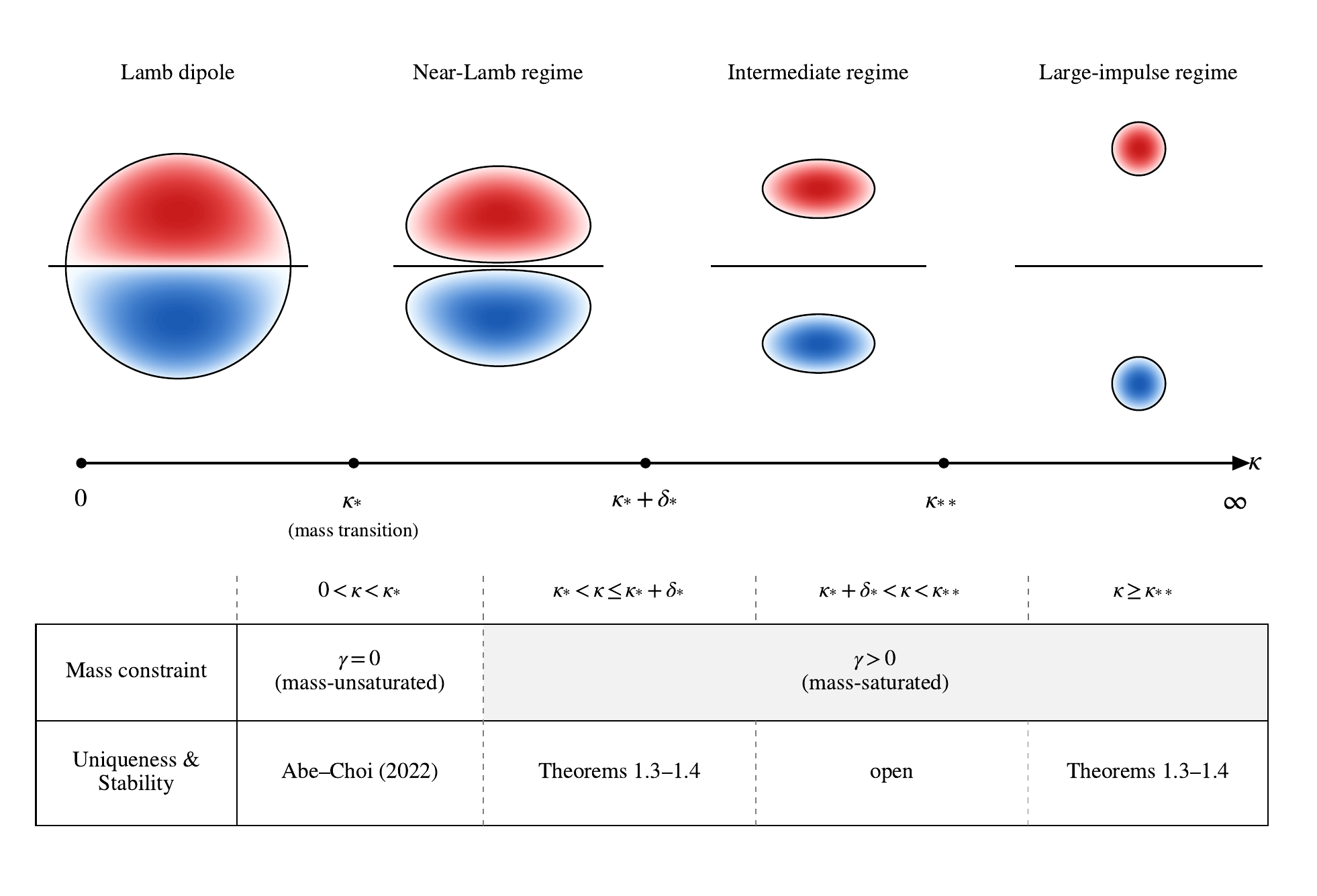}
%{global_branch.pdf}
\caption{
Global picture of the variational family as the dimensionless impulse \(\kappa\) increases. At \(\kappa=\kappa_*\), the mass constraint becomes active and the explicit Lamb-dipole regime gives way to non-explicit traveling dipoles (Theorem \ref{thm:threshold}). Uniqueness and orbital stability hold in the near-Lamb and large-impulse regimes (Theorems \ref{thm:unique-near-lamb} and   \ref{thm:stability-near-lamb}), while uniqueness remains open in the intermediate regime.}
\label{fig:global-branch}
\end{figure}

Figure \ref{fig:global-branch} summarizes the global picture of the variational family and the regimes covered by Theorems \ref{thm:threshold}--\ref{thm:stability-near-lamb}. The two uniqueness regimes lie near the opposite ends of the mass-saturated family, where two explicit limiting solutions to \eqref{eq:near-Lamb-fixed-EL} arise: the Lamb dipole with $\gamma=0$ in $\mathbb R^2_+$ and the radially symmetric solution with $W=0$ in $\mathbb R^2$.

The radially symmetric solutions in $\mathbb{R}^{2}$ form a one-parameter family parametrized by $\lambda>0$ and are given by 
\begin{align}
\omega^{R}_{\lambda,\gamma_R}(x)= \frac{\lambda}{2\pi j_{0,1} J_1(j_{0,1})} J_{0}(\sqrt{\lambda}r)\mathbf{1}_{[0,  {j_{0,1}}{ \lambda^{-\frac12}}]}(r),   \label{eq:radiallysymmetric}
\end{align}
where $j_{0,1}=2.4048\cdots$ is the first positive zero of $J_0$. The radial solution has unit mass   $||\omega_{\lambda,\gamma_R}^{R}||_{L^{1}}=1$ and $\psi=(-\Delta_{\mathbb{R}^{2}})^{-1}\omega^R_{\lambda,\gamma_R}$ satisfies \eqref{eq:near-Lamb-fixed-EL} in $\mathbb{R}^{2}$ for $W=0$ and 
\begin{align}
\gamma_R=\gamma_R(\lambda)=-\frac{1}{2\pi}\log\left(\frac{j_{0,1}}{\sqrt{\lambda}}\right).  \label{eq:radialflux}
\end{align}
See Figure \ref{fig:well-separated-regime}.

We show that the profiles $\omega_{\kappa}$ in Theorems \ref{thm:unique-near-lamb} and \ref{thm:stability-near-lamb} converge to those explicit solutions \eqref{eq:Lamb} and \eqref{eq:radiallysymmetric} as $\kappa\to \kappa_*^{+}$ and $\infty$.

\begin{theorem}[Limiting profiles]\label{thm:well separated dipole}
Let $\lambda,\mu,\nu>0$ satisfy \eqref{eq:near-Lamb-normalized-range} or \eqref{eq:largeimpulse}. Set $W_L=W_L(1,\kappa_{*})>0$ and $\gamma_R=\gamma_R(1)$ by \eqref{eq:Lambspeed} and \eqref{eq:radialflux}. Let $\omega_\kappa$ be the Steiner symmetric maximizer given in
Theorem~\ref{thm:unique-near-lamb}, and let $W>0$ and $\gamma>0$
be the corresponding constants in \eqref{eq:near-Lamb-fixed-EL}.

\noindent
(i) If \eqref{eq:near-Lamb-normalized-range} holds, there exist constants $C, R>0$ such that 
\begin{align}
&C^{-1}\frac{\sqrt{\lambda}\gamma}{\nu}\leq
\operatorname{dist}(\operatorname{spt}\omega_{\kappa},
\partial\mathbb R^2_+)
\leq C\frac{\sqrt{\lambda}\gamma}{\nu}, \label{eq:near-Lamb-fixed-gap}\\
&\operatorname{spt} \omega_{\kappa}\left(\frac{\cdot}{\sqrt{\lambda}}\right)\subset B_0(R)_{+}=B_0(R)\cap \mathbb{R}^{2}_{+}. \label{eq:supportboundhalfplane}
\end{align}
Moreover, as $\kappa\to \kappa_*^+$,
\begin{align}\label{eq:near-Lamb-fixed-convergence}
\frac{1}{\lambda \nu}\omega_{\kappa}\left(\frac{\cdot}{\sqrt{\lambda}}\right)-\omega^L_{1,W_L}&\to 0\quad \textrm{uniformly in}\ \mathbb{R}^{2}_{+},\\
\frac{W}{\nu\sqrt{\lambda}} -W_L&\to 0,\quad 
\frac{\gamma}{\nu}\to 0.
\end{align}

\noindent
(ii) If \eqref{eq:largeimpulse} holds, there exists $R>0$ such that 
\begin{equation}
\operatorname{spt} \omega_{\kappa}\left(\frac{\cdot+\left(0,\kappa\right)}{\sqrt{\lambda}}\right) \subset B_{0}(R).  \label{eq:uniformvortexcorebound}
\end{equation}
Moreover, as $\kappa\to\infty$,
\begin{align}
&\frac{1}{\lambda\nu}\omega_{\kappa}\!\left(\frac{\cdot+\left(0,\kappa\right)}{\sqrt{\lambda}}\right)-\omega^{R}_{1,\gamma_R}
\to 0
\quad \textrm{uniformly in}\ \mathbb{R}^{2},  \label{eq:uniformconvergenceR2}\\
&\frac{W\mu}{\nu^{2}}-\frac1{4\pi}\to0,
\quad
\frac{\gamma}{\nu} - \frac{1}{2\pi}\log\left(2\kappa\right)
+\frac1{4\pi}-\gamma_R\to 0.\label{eq:convergenceWg} 
\end{align}
\end{theorem}

Theorem~\ref{thm:well separated dipole} also describes the behavior of the
traveling speed and the position of the vortex core along the variational
family. As $\kappa\to\kappa_*^+$, the flux constant $\gamma$ vanishes
and the normalized traveling speed $W/(\nu\sqrt{\lambda})$ converges to the Lamb dipole speed $W_L$.
For $\kappa\leq\kappa_*$, the family consists exactly of Lamb dipoles,
whose speed is proportional to the impulse by \eqref{eq:Lambspeed}.
At the opposite end, as $\kappa\to\infty$, the vortex core recedes from
the boundary while $\gamma$ diverges logarithmically and the traveling
speed tends to zero. Thus the variational family interpolates between
the Lamb-dipole regime and a pair of increasingly separated vortex cores, each approaching a stationary radial profile while the translation speed tends to zero.

Despite the weak constraints in \eqref{eq: VP-2},
Theorems~\ref{thm:threshold}--\ref{thm:well separated dipole} reveal a
remarkably rigid structure of the resulting one-parameter family, including
its mass transition, uniqueness, orbital stability, and limiting profiles. They also provide detailed geometric information on the vortex cores, including their localization, separation from the boundary, and asymptotic shape. Thus, substantial rigidity and geometric structure survive the coarse-graining of the classical variational principle.

\subsection{Related works}
    
	\subsubsection{Previous works on the Lamb dipole}
	The variational study of the Lamb dipole is the closest precedent for the
present paper. Burton proved uniqueness in two rearrangement classes
\cite{Burton96,Burton05b}; the work \cite{AC2019} established a variational
characterization and orbital stability. More recently, quantitative orbital
stability was studied in \cite{LSZ26}, while the linearized dynamics around
the Lamb dipole were investigated in \cite{Pro24,PV}. Related variational formulations were studied in \cite{ACJ,Wang.2024}. The Lamb dipole
	profile has since played a role in multi-dipole stability
problems \cite{AJY,CJY}, examples of vorticity-gradient growth
\cite{CJ-Lamb,JYZ}, and convex integration constructions of weak solutions
to the two-dimensional Euler equations \cite{BCK26}. These developments
exploit, in different ways, the explicit Bessel formula for the Lamb dipole. In contrast, the present work concerns a variational family beyond the Lamb dipole, where neither the maximizer nor the geometry of its vortex core is explicit. Thus, the main issue is not the stability of the Lamb dipole itself, but whether its variational rigidity persists after the mass constraint becomes active.

%The Lamb dipole profile has since been used in multi-dipole stability problems \cite{AJY,CJY} and in examples of vorticity-gradient growth \cite{CJ-Lamb,JYZ}, and, more recently, as a building block in convex integration constructions of weak solutions to the two-dimensional Euler equations \cite{BCK26}. These arguments benefit from the explicit Bessel formula for the Lamb dipole.

	\subsubsection{Existence results for vortex dipoles}
	Existence theory for vortex dipoles is much broader. Various types of traveling vortex dipoles were
	constructed by free-boundary, variational, bifurcation, and
	desingularization methods
	\cite{Norbury75,Pierrehumbert80,Tu83a,Tu83b,Burton88,BNL13,CLZ21,SV10,CLW14,HaHm21,GarciaHaziot23,DDMP24}. Traveling vortex pairs have also been used as building blocks in the
construction of genuinely time-dependent global solutions to the
two-dimensional Euler equations \cite{DDMP26}. More recently, touching traveling vortex pairs, including the Sadovskii vortex patch, have been constructed by variational and free-boundary methods \cite{HT25, CJS,CSW,ACJSW}. The nondegeneracy of the Sadovskii vortex patch was recently established in \cite{DMW26}. Most of these results concern existence or stability of a maximizing set, rather than uniqueness along a variational family.
    
    %These results
	%establish existence, compactness, and in several cases stability of the entire maximizing set.
    
    % The issue addressed here is different: we identify that set with a single translation orbit in limits where the free boundary either degenerates at the symmetry axis or moves with the solution.
	
	For vortex patches, uniqueness and orbital stability of a
	well-separated pair were obtained in \cite{CQZZ25}. While \cite{CQZZ25} concerns the concentration regime of vortex patches, the large-impulse regime considered here has a different asymptotic structure: after recentering, each vortex core remains of order one and converges to a nontrivial radially symmetric profile, while the two cores separate. Accordingly, our uniqueness argument does not rely on a point-vortex reduction or local Poho\v{z}aev analysis, but instead on uniform localization of the vortex core, logarithmic free-energy maximization, and coercivity of the linearized operator around the limiting radial profile; see \S \ref{ss:1.4}.
    
    The near-Lamb regime considered here has no counterpart in the concentrated-patch setting and involves a different limiting mechanism associated with the activation of the mass constraint. 
	
	\subsubsection{Analogy with uniformly rotating solutions}
	
	The study of uniformly rotating vorticity solutions for \eqref{eq:2D-Euler}, so-called $V$-states, is another broad subject. As in the case of traveling vortex dipoles, the variational principle of maximizing the energy gives the existence as well as stability of uniformly rotating states. The simplest among such solutions is arguably the Rankine vortex, which is simply the stationary vortex patch supported on a disc. It can be shown that it is the unique energy maximizer under mass constraint in the patch class. Upon fixing the radius of a Rankine vortex, by imposing an \textit{additional} constraint on the radial impulse, the maximizer becomes the Kirchhoff ellipses, at least locally near the Rankine vortex \cite{Tang,WP}. In this analogy, the Lamb dipole (which is indeed a radial object in $\bbR^4$) corresponds to the Rankine vortex, with slightly deformed Lamb dipoles to the Kirchhoff ellipses with small eccentricity. In the present problem, however, the deformation from the Lamb dipole is not generated by imposing an additional geometric constraint, but emerges from the activation of the mass constraint in the variational problem.
	 
	\subsubsection{Fraenkel--Norbury family of vortex rings}

The family of two-dimensional vortex dipoles considered in this paper may be viewed as an analogue of the Fraenkel--Norbury family of vortex rings for the three-dimensional axisymmetric Euler equations without swirl \cite{Fra70,No73}; see also \cite[7.2.3]{WMZ15}. The Fraenkel--Norbury solutions can also be constructed through a variational principle for the vorticity, similar in spirit to \eqref{eq: VP-2} \cite{FT81}. The variational problem there, however, has the special form of maximizing the kinetic energy under an $L^\infty$ constraint on the relative vorticity, and its maximizers are vortex patches. Uniqueness is known for Hill's spherical vortex and a family of Norbury rings \cite{AF86,AF88}. Corresponding orbital stability results are also available \cite{Choi24,CQZZ23}, and both uniqueness and stability have been established in the thin-cored regime \cite{CQYZZ22, CLQZZ26}. However, a unified theory of uniqueness and stability along the Fraenkel--Norbury family is still lacking. The results of the present paper suggest that a picture analogous to Figure \ref{fig:global-branch} may also hold for the Fraenkel--Norbury family.

	\subsubsection{Related equations}
	
	A closely related variational problem appears in the Euler--Poisson and Vlasov--Poisson setting: Jang and Seok used uniqueness of McCann's constrained minimizers to prove orbital stability of uniformly rotating binary stars
	and galaxies \cite{JangSeok22}.
	
	We finally mention that a different question, the long-time viscous
	evolution of dipoles, has recently been studied by Dolce and Gallay
	\cite{DG26}.
	
\subsection{Main ideas and difficulties}\label{ss:1.4}

The key step in the proof is the uniqueness result in
Theorem \ref{thm:unique-near-lamb}. Its proof relies on the limiting profiles in Theorem \ref{thm:well separated dipole} and on coercive estimates for the corresponding linearized operators. The two limiting regimes require substantially different compactness and rigidity mechanisms, which we outline below.

\subsubsection{Large-impulse regime}

A distinctive feature of the large-impulse regime is that the vortex cores do not concentrate: their size remains of order one while their separation diverges. This leads, after recentering and subtracting the leading interaction energy, to a logarithmic free-energy maximization problem on $\mathbb {R}^2$. More precisely, we consider the free energy
\begin{align}
	{\mathcal{E}}_{2}[\omega]
	=\mathcal{E}[\omega]-\frac{1}{2}
		\|\omega\|_{L^2(\mathbb R^2)}^2=
	\frac{1}{4\pi}
	\iint_{\mathbb R^2\times\mathbb R^2}
	\log\frac{1}{|x-y|}	\omega(x)\omega(y)
	\,dx\,dy
	-
	\frac{1}{2}
	\int_{\mathbb R^2}\omega^2\,dx
	\label{eq:FE}
\end{align}
and the maximization problem
\begin{align}
	\mathcal I
	&=
	\sup_{\omega\in\mathcal K}
	\mathcal E_2[\omega],
	\label{eq:MFE}\\
	\mathcal K
	&=
	\left\{
	\omega\in L^2\cap L^1(\mathbb R^2):
	\ \omega\geq 0,\ 
	|x|\omega\in L^1(\mathbb R^2),\
	\|\omega\|_{L^1(\mathbb R^2)}=1,\
	\int_{\mathbb R^2}x\,\omega(x)\,dx=0
	\right\}.
	\label{eq:AFE}
\end{align}
This variational problem arises in aggregation--diffusion models, and in particular in the two-dimensional Keller--Segel equation with quadratic diffusion. Carrillo, Castorina, and Volzone~\cite{Car} proved that the maximizer of \eqref{eq:MFE} is uniquely given by the radially symmetric profile \eqref{eq:radiallysymmetric} for $\lambda=1$.
For a more detailed classification of stationary states and the long-time behavior of solutions to aggregation--diffusion equations, we refer to Carrillo, Hittmeir, Volzone, and Yao~\cite{CHVY19}; see also Bailo, Carrillo, and G\'{o}mez-Castro \cite{BCG26} for the recent survey.

\begin{figure}[ht]
\centering
\includegraphics[width=.50\textwidth]{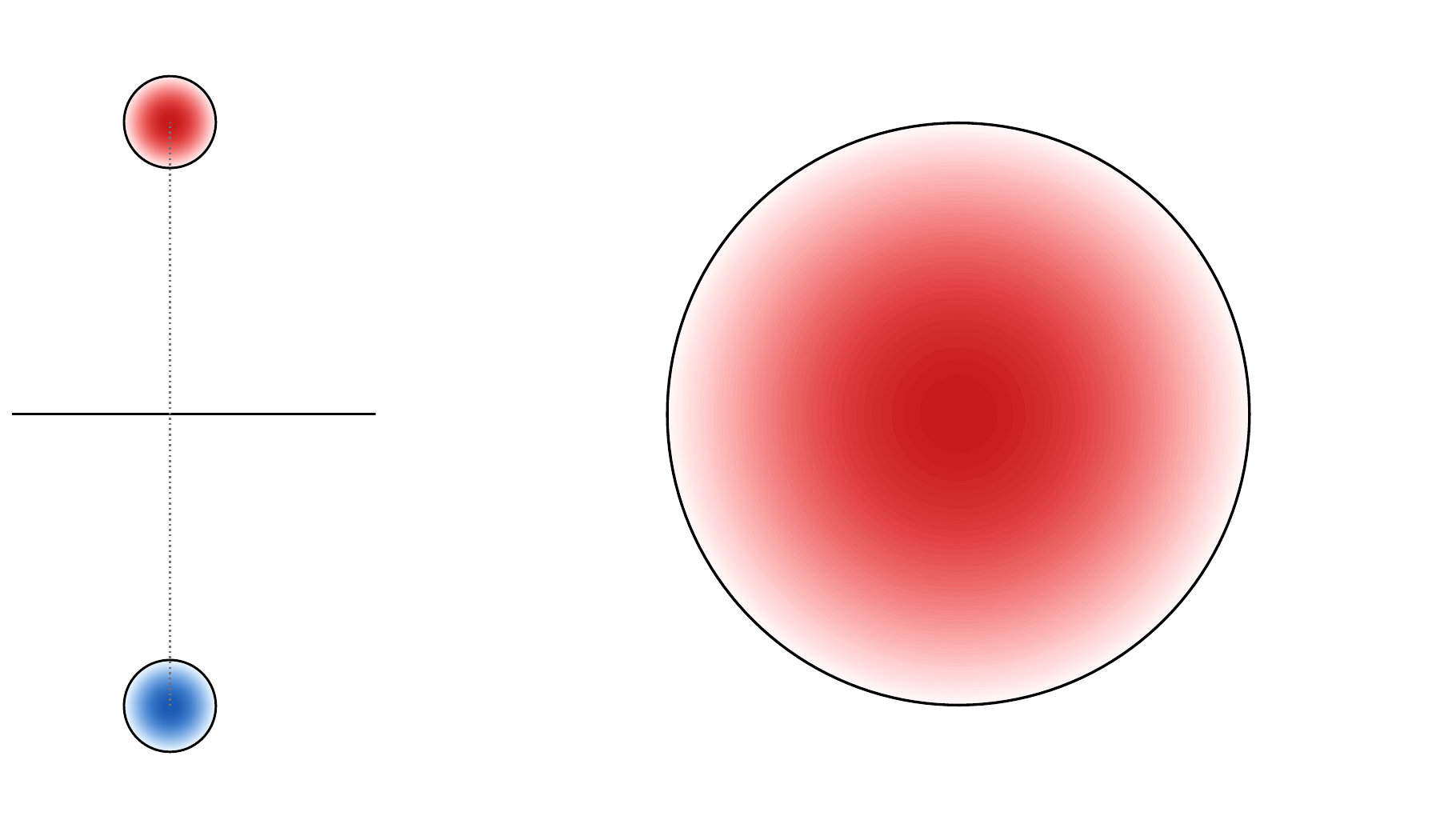}
\caption{A large-impulse vortex dipole in Theorem \ref{thm:well separated dipole} (ii) and the radially symmetric solution \eqref{eq:radiallysymmetric}.}
\label{fig:well-separated-regime}
\end{figure}

The maximization problem \eqref{eq: VP-2} in the mass-saturated regime $\kappa>\kappa_*$ can be viewed as a half-plane analogue of the problem \eqref{eq:MFE} on $\mathbb R^2$.
Indeed, we will see the link between the two problems by showing the asymptotics of the maximum
\begin{align*}
	I_{\mu,1,1}
	-
	\frac{1}{4\pi}\log(2\mu)
	=
	\mathcal I
	+
	O(\mu^{-2})\quad \textrm{as}\ \mu\to\infty.
\end{align*}
Namely, after subtracting the leading term $(4\pi)^{-1}\log(2\mu)$, the maximal value converges to the maximum of the logarithmic energy maximization problem on $\mathbb R^2$. More precisely, for a Steiner symmetric maximizer $\omega$ of $I_{\mu,1,1}$, we show that its recentered vortex 
\begin{align*}
	\widehat{\omega}(x)
	=
	\omega\bigl(x+(0,\mu)\bigr),
\end{align*}
forms a maximizing sequence for the logarithmic energy maximization problem \eqref{eq:MFE}.

Without radial symmetry, maximizing sequences for \eqref{eq:MFE} need not in general be compact; see \cite{Car}. In the present problem, however, we can exploit the additional compact-support structure of the vortex core.
More precisely, once the uniform bound on the vortex core in \eqref{eq:uniformvortexcorebound} is established, we obtain the convergence toward the radially symmetric solution \eqref{eq:uniformconvergenceR2} along some sequence $\{\mu_n\}$ satisfying $\mu_n\to\infty$ and a corresponding sequence of Steiner symmetric maximizers $\{\omega_n\}$. This convergence is the starting point for proving the uniqueness of the maximizer in the large-impulse regime.

We now outline the uniqueness argument in the large-impulse regime $\mu\geq \kappa_{**}$ (with $\lambda=\nu=1$). We will argue by contradiction. Suppose that there are two Steiner symmetric maximizers $\omega^{1}_{n}$ and $\omega_n^{2}$ of \eqref{eq: VP-2} along a sequence $\mu_n\to\infty$. We will show the uniform bound 
\begin{align}
\operatorname{spt} \hat{\omega}_n^{i}\subset B=B_0(R) \label{eq:compactsptbound}
\end{align}
with some $R>0$ for the recentered $\hat{\omega}^{i}_{n}$, extended by zero to $\mathbb{R}^{2}$. We deduce from this uniform vortex core bound that the recentered functions $\hat{\omega}_n^i$ converge to the radially symmetric solution $\omega^{R}_{1,\gamma_R}$. We then normalize the difference of the two maximizers. Using the Euler--Lagrange equation \eqref{eq:near-Lamb-fixed-EL} and the Poho\v{z}aev identities, we show that, along a subsequence, the normalized difference converges strongly to a nontrivial limit \(\zeta\) satisfying
\begin{align*}
L\zeta=0,\quad L=I-\mathbf{1}_{B_0(j_{0,1})}(-\Delta_{\mathbb{R}^{2}})^{-1}\mathbf{1}_{B_0(j_{0,1})}.
\end{align*}
The zero-moment condition, together with the coercivity of \(L\), forces \(\zeta=0\), contradicting the normalization.

The main difficulty in the large-impulse regime is to obtain a uniform bound on the recentered vortex core without any a priori localization. We first derive a uniform local lower bound on the vorticity mass and combine it with Steiner symmetry to control the horizontal extent of the core. The vertical extent is then controlled through the logarithmic free energy: if the vertical diameter were large, two portions carrying a definite amount of mass would be widely separated, forcing the free energy to tend to \(-\infty\), in contradiction with a uniform lower bound. This yields the uniform localization \eqref{eq:compactsptbound}, which in turn gives compactness and convergence to the radial limiting profile. We will outline the idea in more detail in \S \ref{ss:5.1}.

\subsubsection{Near-Lamb dipole regime}

The uniqueness in the near-Lamb dipole regime is proved by a similar
contradiction argument, but the rigidity mechanism is different from that
in the large-impulse regime. Suppose that $\mu_n\to\kappa_*^+$ and that
there exist two distinct Steiner symmetric maximizers
$\omega_n^1,\omega_n^2\in S_{\mu_n}$. Since these are maximizing
sequences for $I_{\kappa_*}$, compactness and the uniqueness at
$\kappa_*$ imply that both sequences converge to the Lamb dipole.
Together with the Euler--Lagrange equation, this convergence yields a uniform bound on their vortex cores.

We then normalize the difference $\omega_n^1-\omega_n^2$ and pass to a
nontrivial limit $\zeta$. Unlike in the large-impulse regime, the limiting
equation is inhomogeneous because the Lagrange multipliers also vary.
Moreover, the linearized operator around the Lamb dipole is indefinite:
it possesses a negative eigenvalue. Nevertheless, the Steiner symmetry
and the mass and impulse constraints allow us to establish coercivity of
the quadratic form on the relevant constrained subspace. Applying this
coercive estimate to $\zeta$ forces $\zeta=0$, contradicting the
normalization.

The two uniqueness arguments are compared in Table~\ref{tab:1}.
Although they have the same contradiction structure, the rigidity
mechanisms at the two endpoints are fundamentally different.

\begin{table}[h]
\begin{tabular}{|c|c|c|}
\hline
& Near-Lamb dipole                                                                           & Large-impulse                                                                    \\ \hline
Impulse    & $\kappa_*<\kappa\leq \kappa_*+\delta_{*}$                                               & $\kappa_{**}\leq \kappa$                                                                                 \\ \hline
\begin{tabular}[c]{@{}c@{}}Endpoint\\ variational \\ principle\end{tabular} & \begin{tabular}[c]{@{}c@{}}The energy-Casimir \\ maximization \eqref{eq: VP-2} in $\mathbb{R}^{2}_{+}$\end{tabular}    & \begin{tabular}[c]{@{}c@{}}The free-energy \\ maximization \eqref{eq:MFE} in $\mathbb{R}^{2}$\end{tabular}                          \\ \hline
\begin{tabular}[c]{@{}c@{}}Explicit \\ solutions\end{tabular}    & The Lamb dipole \eqref{eq:Lamb}                                                                          & \begin{tabular}[c]{@{}c@{}}The radially symmetric \\ solution \eqref{eq:radiallysymmetric}\end{tabular}                                \\ \hline
Compactness                                                      & \begin{tabular}[c]{@{}c@{}}Compactness of \\ maximizing sequence\end{tabular}                                              & \begin{tabular}[c]{@{}c@{}}Uniform localization of \\ vortex cores\end{tabular}                                                         \\ \hline
\begin{tabular}[c]{@{}c@{}}Linearized \\ operator\end{tabular}   & \begin{tabular}[c]{@{}c@{}}$L_+=I-\mathbf{1}_{B_0(j_{1,1})_+}(-\Delta_D)^{-1}\mathbf{1}_{B_0(j_{1,1})_+}$\\ (Indefinite)\end{tabular} & \begin{tabular}[c]{@{}c@{}}$L=I-\mathbf{1}_{B_0(j_{0,1})}(-\Delta_{\mathbb{R}^{2}})^{-1}\mathbf{1}_{B_0(j_{0,1})}$\\ (Nonnegative)\end{tabular} \\ \hline
\begin{tabular}[c]{@{}c@{}}Limit \\ equations\end{tabular}       & $L_+\zeta=-2\left(\frac{x_2}{\kappa_*}-1\right)(\zeta,\omega^L_{1,W_L})_{L^2}$                                         & $L\zeta=0$                                                                                                                     \\ \hline
\end{tabular}
\vspace{10pt}
\caption{Comparison of the two uniqueness arguments in Theorem \ref{thm:unique-near-lamb}.}
\label{tab:1}
\end{table}

Recently, Li, Song, and Zhou \cite{LSZ26} established a coercive estimate for the operator arising from the linearization around the Lamb dipole, under suitable orthogonality conditions, in their study of its quantitative orbital stability. In contrast, in the present paper, we establish the corresponding coercivity under the symmetry and zero-impulse conditions naturally arising from our variational problem, and use it to prove the uniqueness of the maximizer.

\subsubsection{Intermediate regime}

Our uniqueness arguments rely essentially on detailed information about the two explicit solutions arising in the limiting regimes. Consequently, the uniqueness of maximizers for the variational problem \eqref{eq: VP-2} remains open in the intermediate regime
\begin{align*}
\kappa_*+\delta_{*}<\kappa<\kappa_{**}.
\end{align*}
A closely related question concerns the spectral properties of the operator obtained by linearizing the Euler--Lagrange equation \eqref{eq:near-Lamb-fixed-EL} around a general maximizer $\omega$:
\begin{align*}
L_+=I-\mathbf{1}_{\Omega}(-\Delta_{D})^{-1}\mathbf{1}_{\Omega},
\end{align*}
where $\Omega$ is the vortex core of $\omega$. At the radially symmetric limiting profile \eqref{eq:radiallysymmetric}, $L$ is nonnegative with a finite-dimensional kernel, whereas at the opposite limiting profile, the Lamb dipole \eqref{eq:Lamb}, $L_+$ possesses a negative direction. Related spectral instability questions
for steady Euler flows have recently been studied in
\cite{Pro24,ZPS24,CCDV26}. In both limiting cases, the explicit geometry of the vortex core allows us to obtain the spectral information required for the uniqueness argument through Fourier decomposition in the angular variable. For a general maximizer in the intermediate regime, however, the geometry of the vortex core $\Omega$ is not explicitly known, and such a separation of variables is no longer available. Whether the variational family remains unique, and hence whether individual orbital stability persists, throughout the intermediate regime is an interesting open problem.

\subsection{Organization of the paper} 
This paper is organized as follows. In \S\ref{s:2}, we establish the Steiner symmetry and the mass transition of the variational solutions, proving Theorems \ref{thm:nearly Lamb} and \ref{thm:threshold}. In \S\ref{s:3}, we establish coercive estimates for the linearized operators. In \S\ref{s:4} and \S\ref{s:5}, we study the uniqueness and limiting profiles of the variational solutions in the near-Lamb dipole and large-impulse regimes, respectively, and prove Theorems \ref{thm:unique-near-lamb} and \ref{thm:well separated dipole}. Finally, in \S\ref{s:6}, we deduce the orbital stability from the uniqueness results, thereby proving Theorem \ref{thm:stability-near-lamb}.

\section{Properties of variational solutions}\label{s:2}

In this section, we prove Theorems \ref{thm:nearly Lamb} and \ref{thm:threshold}. We first collect estimates for the stream function and kinetic energy derived from pointwise bounds on the Green function. We then introduce Steiner symmetrization and prove Theorem \ref{thm:nearly Lamb}. Using known uniqueness results for the variational problem, we next establish the mass transition in Theorem \ref{thm:threshold}. In the latter part of the section, we derive the Pohožaev identities and uniform estimates for the maximizers that will be used in the analysis of uniqueness and limiting profiles in the subsequent sections.\\

Let $\mu,\lambda,\nu>0$. Observe that the rescaled function  
	\begin{align*}
		\tilde{\omega}(x)=\frac{1}{\lambda \nu}\omega\left(\frac{x}{\sqrt{\lambda}}\right),  %\label{eq:rescaling}
	\end{align*}
satisfies 
	\begin{align*}
		E_{2,1}[\tilde{\omega}]=\frac{1}{\nu^{2}}E_{2,\lambda}[{\omega}],\quad \nrm{\tilde{\omega}}_{L^{1}(\mathbb{R}^{2}_{+})}=\frac{1}{\nu}\nrm{\omega}_{L^{1}(\mathbb{R}^{2}_{+})}\quad \nrm{\tilde{\omega}}_{L^{1}_*(\mathbb{R}^{2}_{+})}=\frac{\sqrt{\lambda}}{\nu}\nrm{\omega}_{L^{1}_{*}(\mathbb{R}^{2}_{+})}. 
	\end{align*}
This induces the scaling property of the maximum \eqref{eq: VP-2}  
	\begin{align*}
		I_{\kappa,1,1}=\frac{1}{\nu^{2}}I_{\mu,\nu,\lambda},\quad\mbox{where}\quad  \kappa=\frac{\sqrt{\lambda}\mu}{\nu}.  %\label{eq:scalingmaximum}
\end{align*}
Moreover, if $\omega$ satisfies the Euler--Lagrange equation \eqref{eq:near-Lamb-fixed-EL} for constants $W>0$ and $\gamma\geq 0$, $\tilde{\omega}$ satisfies \eqref{eq:near-Lamb-fixed-EL} for 
\begin{align*}
\tilde{W}=\frac{W}{\nu\sqrt{\lambda}},\quad \tilde{\gamma}=\frac{\gamma}{\nu}.
\end{align*}
We thus reduce the variational problem \eqref{eq: VP-2} to the case $\mu>0$ and $\lambda=\nu=1$ and denote the energy, the maximum, and the set of maximizers by $E_{2} = E_{2,1}$, $I_{\mu}=I_{\mu,1,1}$, and $S_{\mu}=S_{\mu,1,1}$. In the sequel, it suffices to prove Theorems \ref{thm:threshold}--\ref{thm:well separated dipole} in the case $\mu>0$ and $\lambda=\nu=1$.

\subsection{Stream function estimates}

We use the Green function expressions 
\begin{align}
G(x,y)=N(x-y)-N(x-\bar{y})=\frac{1}{4\pi}\log\left(1+\frac{4x_2y_2}{|x-y|^{2}}\right),\quad N(x)=-\frac{1}{2\pi}\log|x|, \label{eq:Green}
\end{align}
where $\bar y=(y_1,-y_2)$ is the reflection point of $y=(y_1,y_2)$ across the $x_1$-axis. By the pointwise estimate $\log(1+t)\lesssim t^{\alpha}$ for $t>0$ and $\alpha\in (0,1]$, the Green function satisfies 
\begin{align}
G(x,y)
\lesssim  \frac{x_2^{\alpha}y_2^{\alpha}}{|x-y|^{2\alpha}},\quad x,y\in \mathbb{R}^{2}_{+}.  \label{eq:GFE}
\end{align}
The following energy inequality \eqref{eq:KEE} was partially obtained in \cite{AC2019} for $\alpha=1/4$ and in \cite{ACJ} for $\alpha=1/2$. We also collect basic estimates for stream functions.

\begin{proposition}
Let $0<\alpha\leq 1/2$ and $0<\theta<1$, and $2<r<\infty$. The following estimates hold whenever the norms appearing on the
right-hand sides are finite:
\begin{align}
||\mathcal{G}\omega||_{L^{\infty}(\mathbb{R}^{2}_{+})}&\lesssim  ||\omega||_{L^{1}_{*}(\mathbb{R}^{2}_{+})}^{\frac{1}{2}}||\omega||_{L^{2}(\mathbb{R}^{2}_{+})}^{\frac{1}{2}}, \label{eq:SFE} \\
||\nabla \mathcal{G}\omega||_{L^{\infty}(\mathbb{R}^{2}_{+})}&\lesssim ||\omega||_{L^{1}_{*}(\mathbb{R}^{2}_{+})}^{\frac{\alpha}{2}}||\omega||_{L^{1}(\mathbb{R}^{2}_{+})}^{\frac{1}{2}-\alpha}||\omega||_{L^{2}(\mathbb{R}^{2}_{+})}^{\frac{\alpha}{2}}||\omega||_{L^{\infty}(\mathbb{R}^{2}_{+})}^{\frac{1}{2}},  \label{eq:sinfty}\\
\mathcal{G}\omega(x) &\lesssim \frac{1}{x_2}\big(\log x_2 + 1\big)\|\omega\|_{L^1_*(\mathbb{R}^{2}_{+})} + \frac{1}{x_2^{\theta}} \|\omega\|_{L^1_*(\mathbb{R}^{2}_{+})}^\theta \|\omega\|_{L^2(\mathbb{R}^{2}_{+})}^{1-\theta}, \qquad x_2\ge 1,\label{eq:inequality sf 2}\\
E[\omega]
&\lesssim \nrm{\omega}_{L^1_*(\mathbb R^2_+)}^{2\alpha}\nrm{\omega}_{L^{1}(\mathbb{R}^{2}_{+})}^{2(1-2\alpha)}\nrm{\omega}_{L^{2}(\mathbb{R}^{2}_{+})}^{2\alpha}.   \label{eq:KEE}
\end{align}
If $\omega$ is Steiner symmetric with respect to the $x_2$-axis, then 
\begin{equation}\label{eq:inequality sf 4}
\mathcal{G}\omega(x) \lesssim \big( \|\omega\|_{L^1_*(\mathbb{R}^{2}_{+})} + \|\omega\|_{L^1(\mathbb{R}^{2}_{+})} + \|\omega\|_{L^r(\mathbb{R}^{2}_{+})}\big) \, x_2 \, \min\left\{1,\,\frac{1}{|x_1|^{\frac{1}{2r}}}\right\}.
\end{equation}
\end{proposition}

\begin{proof}
The estimates \eqref{eq:inequality sf 2} and \eqref{eq:inequality sf 4} are due to \cite[Lemmas 1 and 5]{Burton21}. The estimate \eqref{eq:SFE} follows from the Green function estimate \eqref{eq:GFE} \cite[Lemma 2.3]{ACJSW}. The estimate \eqref{eq:sinfty} follows from \eqref{eq:KEE} for $E[\omega]=\frac{1}{2}||\nabla \mathcal{G}\omega||_{L^{2}}^{2}$ and the interpolation inequality
\begin{align*}
||\nabla \mathcal{G}\omega||_{L^{\infty}(\mathbb{R}^{2}_{+})}\lesssim ||\nabla \mathcal{G}\omega||_{L^{2}(\mathbb{R}^{2}_{+})}^{\frac{1}{2}}
||\omega||_{L^{\infty}(\mathbb{R}^{2}_{+})}^{\frac{1}{2}}.
\end{align*}
We show the estimate \eqref{eq:KEE}. We may assume that $\omega$ is nonnegative. We take $0<\alpha\leq 1/2$ and apply the Green function estimate \eqref{eq:GFE} to get 
\begin{align*}
E[\omega]\lesssim \iint_{\mathbb{R}^{2}_{+}\times \mathbb{R}^{2}_{+}}\frac{x_2^{\alpha}\omega(x)y_2^{\alpha}\omega(y)}{|x-y|^{2\alpha}}dxdy.
\end{align*}
We consider the zero extension of $\omega$ to $\mathbb{R}^{2}$ and set $\eta=x^{\alpha}_2\omega$. By the Hardy--Littlewood--Sobolev inequality, the map $\eta\longmapsto |x|^{-2+\beta}*\eta$ is a bounded operator from $L^{q}(\mathbb{R}^{2})$ to $L^{p}(\mathbb{R}^{2})$ for $1/p=1/q-\beta/2$. We set $\beta=2-2\alpha$ and choose $p$ so that $1/p+1/q=1$. Namely, $1/p=\alpha/2$ and $1/q=1-\alpha/2$. Then, $E[\omega]\lesssim ||\eta||_{L^{q}}^{2}$. Applying H\"older's inequality, we find that 
\begin{align*}
\int_{\mathbb{R}^{2}_{+}}\eta^{q}dx
=\int_{\mathbb{R}^{2}_{+}}x_2^{2 q-2}\omega^{q}dx
=\int_{\mathbb{R}^{2}_{+}}(x_2\omega)^{2 q-2}\omega^{2-q}dx
\leq ||x_2\omega||_{L^{1}}^{2q-2}||\omega||_{L^{(2-q)l}}^{2-q},
\end{align*}
for $l$ satisfying $2q-2+1/l=1$. By H\"older's inequality,
\begin{align*}
||\omega||_{L^{(2-q)l}}\leq ||\omega||_{L^{1}}^{\frac{4-3q}{2-q}}||\omega||_{L^{2}}^{\frac{2q-2}{2-q}}.
\end{align*}
Using $\alpha/2=1-1/q$, we obtain the desired estimate.
\end{proof}

\subsection{Steiner symmetry}

The existence and basic properties of maximizers stated in
Theorem~\ref{thm:nearly Lamb} were established in \cite{AC2019}.
In this subsection, we complete the proof of Theorem~\ref{thm:nearly Lamb}
by establishing the Steiner symmetry of maximizers up to horizontal
translations.

\begin{proposition}[Steiner symmetry]\label{prop:Steiner}
(i) For $\omega\geq 0$ satisfying $\omega\in L^{2}\cap L^{1} \cap L^1_*(\mathbb{R}^{2}_{+})$, there exists $\omega^{*}\geq 0$ such that 
		\begin{equation}
			\begin{aligned}
				&\omega^{*}(x_1,x_2)=\omega^{*}(-x_1,x_2), \\
				&\omega^{*}(x_1,x_2)\ \textrm{is non-increasing for}\ x_1>0. 
			\end{aligned}
			\label{eq:S1}
		\end{equation}
		Moreover, 
		\begin{equation}
			\begin{aligned}
				&\|\omega^{*}\|_{L^q(\mathbb R^2_+)}
				=\|\omega\|_{L^q(\mathbb R^2_+)}
				\quad 1\leq q\leq 2,    \\
				&\nrm{\omega^*}_{L^1_*(\mathbb R^2_+)}
				=\nrm{\omega}_{L^1_*(\mathbb R^2_+)},   \ \\
				&E(\omega^{*})\geq E(\omega).  
			\end{aligned}
			\label{eq:S2}
\end{equation}
(ii) If a nonzero $\omega$ satisfies $E[\omega^*]=E[\omega]$, then there exists $a\in\mathbb R$ such that $\omega(x_1,x_2)=\omega^*(x_1+a,x_2)$ almost everywhere.
\end{proposition}

\begin{proof}
The assertion (i) can be found in \cite[Appendix I]{FB74}, \cite[p.1053]{Tu83b}, \cite[Proposition 3.1]{AC2019}. We prove (ii). We express $\omega$ by the layer-cake representation  
\begin{align*}
\omega(x_1,x_2)=\int_{0}^{\infty}\mathbf{1}_{A_{t,x_2}}(x_1)dt,\quad A_{t,x_2}=\{z_1\in \mathbb{R}:\ \omega(z_1,x_2)>t\},
\end{align*}
and the kinetic energy by the Riesz functional
\begin{align*}
2E[\omega]&=\iint_{(0,\infty)^{4}}I(A_{t,x_2},A_{s,y_2})dtdsdx_2dy_2,\\
I(A_{t,x_2},A_{s,y_2})&=\iint_{\mathbb{R}^{2}}K_{x_2,y_2}(x_1-y_1)\mathbf{1}_{A_{t,x_2}}(x_1)\mathbf{1}_{A_{s,y_2}}(y_1)dx_1dy_1,
\end{align*}
where 
\[
K_{x_2,y_2}(t)=\frac1{4\pi}\log\left(1+\frac{4x_2y_2}{t^2+(x_2-y_2)^2}\right)
\]
is even and strictly decreasing in $|t|$. By Riesz's rearrangement inequality, $I(A_{t,x_2},A_{s,y_2})$ is bounded by $I(A_{t,x_2}^{*},A_{s,y_2}^{*})$ for the superlevel set $A^{*}_{t,x_2}$ of the Steiner symmetrization $\omega^{*}$. Using $E[\omega]=E[\omega^{*}]$, we find that
\begin{align}
I(A_{t,x_2},A_{s,y_2})=I(A_{t,x_2}^{*},A_{s,y_2}^{*}),\quad \textrm{a.e.}\ (t,x_2), (s,y_2)\in (0,\infty)^{2}.  \label{eq:I=I^{*}}
\end{align}
We show existence of a constant $a\in \mathbb{R}$ such that 
\begin{align}
\mathbf{1}_{A_{t,x_2}}(z_1)=\mathbf{1}_{A_{t,x_2}^{*}}(z_1+a),\quad \textrm{a.e.}\ (z_1,t,x_2)\in \mathbb{R}\times (0,\infty)^{2}. \label{eq:A=A^{*}+a}
\end{align}
This implies 
\begin{align*}
\omega(x_1,x_2)=\int_{0}^{\infty}\mathbf{1}_{A_{t,x_2}}(x_1)dt=\int_{0}^{\infty}\mathbf{1}_{A_{t,x_2}^{*}}(x_1+a)dt=\omega^{*}(x_1+a,x_2),
\end{align*}
and the desired result follows.

Let
\[
X=\{p=(t,x_2)\in (0,\infty)^2:0<|A_p|<\infty\},
\]
and set  
\begin{align*}
N&=\{(p,q)\in X^{2}: I(A_p,A_q)<I(A_p^{*},A_q^{*})\},\\
N_q&=\{p\in X: (p,q)\in N\}.
\end{align*}
By \eqref{eq:I=I^{*}}, we have $|N|=0$. By Fubini's theorem,
\begin{align*}
0=|N|=\int_{X}dq\int_{N_q}dp=\int_{X}|N_q|dq.
\end{align*}
Thus $|N_q|=0$ for a.e. $q\in X$. We choose $q_0\in X$ such that $|N_{q_0}|=0$. Then
\begin{align*}
I(A_p,A_{q_0})=I(A_p^{*},A_{q_0}^{*}),\quad \textrm{a.e.}\ p\in X.
\end{align*}
By the strict rearrangement inequality \cite[Theorem 3.9]{LL01}, for a.e. $p\in X$, there exists a constant $c_p\in\mathbb{R}$ such that
\begin{align*}
\mathbf{1}_{A_p}(z_1)
&=\mathbf{1}_{A_p^{*}}(z_1+c_p),\\
\mathbf{1}_{A_{q_0}}(z_1)
&=\mathbf{1}_{A_{q_0}^{*}}(z_1+c_p),
\end{align*}
for a.e. $z_1\in\mathbb{R}$.
Since $A_{q_0}^{*}$ is an interval of positive finite length,
the translation parameter in the second identity is unique.
Therefore $c_p$ is independent of $p$ for a.e. $p\in X$.
Denoting this common value by $a$, we obtain
\[
\mathbf{1}_{A_p}(z_1)
=
\mathbf{1}_{A_p^{*}}(z_1+a)
\]
for a.e. $(z_1,p)\in \mathbb{R}\times X$.

For $p$ such that $|A_p|=0$, both sides vanish for a.e. $z_1$.
Since $|A_p|<\infty$ for a.e. $p\in(0,\infty)^2$, 
\eqref{eq:A=A^{*}+a} follows.
\end{proof}

\begin{proof}[Proof of Theorem \ref{thm:nearly Lamb}]
It suffices to show that the maximizers are Steiner symmetric. If $\omega$ is a maximizer of \eqref{eq: VP-2}, its Steiner symmetry $\omega^{*}$ is also a maximizer by Proposition \ref{prop:Steiner} (i). Thus $I_{\mu}=E_{2}[\omega]=E_{2}[\omega^{*}]$. By Proposition \ref{prop:Steiner} (ii), $\omega$ is Steiner symmetric up to $x_1$-translation.
\end{proof}

\subsection{The mass transition}

We show the mass transition (Theorem \ref{thm:threshold}) by using two variational characterizations of the Lamb dipole: uniqueness of maximizers of \eqref{eq: VP-2} with mass condition for small impulse \cite[Theorems 1.5 and 6.1]{AC2019} and uniqueness of maximizers for $E_2$ without mass condition \cite[Theorem 2.5]{ACJ}. The conservation of the Lamb dipole follows from the explicit form \eqref{eq:Lamb}, e.g., \cite[(1.6)]{ACJ}.

\begin{lemma}[Uniqueness with mass]\label{l:uniquemass}
Let $\omega\in S_{\mu}$ be a maximizer satisfying \eqref{eq:near-Lamb-fixed-EL} for some $W>0$ and $\gamma\geq 0$. If $\gamma=0$, then, up to translation in $x_1$-direction,  $\omega=\omega^{L}_{1,W_L}$ for $W_L$ given by \eqref{eq:Lambspeed}.  
\end{lemma}

\begin{lemma}[Uniqueness without mass]\label{p:uniquenesswithoutmass}
For $\mu>0$, set 
\begin{align}
\tilde{I}_{\mu}&=\sup_{\omega\in \tilde{K}_{\mu}}E_{2}[\omega],\\
\tilde{K}_{\mu}&=\{\omega\in L^{2}\cap L^{1}_{*}(\mathbb{R}^{2}_{+})\ :\ \omega\geq 0,\ ||\omega||_{L^{1}_{*}(\mathbb{R}^{2}_{+})}=\mu \}.
\end{align}
Then, the Lamb dipole $\omega^{L}_{1,W_L}$ for $W_{L}$ given by \eqref{eq:Lambspeed} is the unique maximizer of $\tilde{I}_{\mu}$ up to $x_1$-translations.
\end{lemma}

\begin{proposition}[Conservation]\label{p:Lambmass}
The following hold for the Lamb dipole $\omega^{L}_{\lambda,W}$ with $\lambda,W>0$:   
\begin{align}
E_{2,\lambda}[\omega^{L}_{\lambda,W}]=\frac{j_{1,1}^{2}\pi W^{2}}{2\lambda},\quad
||\omega^{L}_{\lambda,W}||_{L^{1}_{*}(\mathbb{R}^{2}_{+})}=\frac{j_{1,1}^{2}\pi W}{\lambda},
\quad
||\omega^{L}_{\lambda,W}||_{L^{1}(\mathbb{R}^{2}_{+})}=\frac{j_{1,1}^{2}\pi W }{\kappa_*\sqrt{\lambda}}.  \label{eq:massLamb}
\end{align}
If $||\omega^{L}_{\lambda,W}||_{L^{1}_{*}}=\mu$, then
\begin{align}
||\omega^{L}_{\lambda,W}||_{L^{1}(\mathbb{R}^{2}_{+})}=\frac{\mu\sqrt{\lambda}}{\kappa_*},
\end{align}
for $W=W_L$ given by \eqref{eq:Lambspeed}.
\end{proposition}

\begin{proof}[Proof of Theorem \ref{thm:threshold}]
We give a proof for the case $\mu>0$ and $\lambda=\nu=1$. We first observe that if some $\omega\in S_{\mu}$ satisfies  
\begin{align}
\int_{\mathbb{R}^{2}_+}\omega dx<1, \label{eq:masssmall}
\end{align}
then $\mu<\kappa_*$. We take an arbitrary $\omega \in S_{\mu}$ satisfying \eqref{eq:masssmall} and denote its Lagrange multipliers in Theorem \ref{thm:nearly Lamb} by $W>0$ and $\gamma\geq 0$. If $\gamma>0$, we have $\int_{\mathbb{R}^{2}_{+}}\omega dx=1$ by \cite[Remark 2.6 (ii)]{AC2019}. Thus $\gamma=0$. By the uniqueness in Lemma \ref{l:uniquemass}, $\omega$ is a horizontal translation of the Lamb dipole $\omega^{L}_{1,W_L}$. By Proposition \ref{p:Lambmass} and the assumption \eqref{eq:masssmall}, the condition $\mu<\kappa_*$ follows.

We now prove (B). Let $\mu>\kappa_*$. We take an arbitrary $\omega\in S_{\mu}$. If $\omega$ satisfies \eqref{eq:masssmall}, we have $\mu<\kappa_*$. This is a contradiction. Thus $\int_{\mathbb{R}^{2}_{+}}\omega dx=1$. If $\gamma=0$, $\omega$ is a horizontal translation of $\omega^{L}_{1,W_L}$ by Lemma \ref{l:uniquemass}. Applying Proposition \ref{p:Lambmass} yields $||\omega^{L}_{1,W_L}||_{L^{1}}=\mu/\kappa_*>1$. This is a contradiction. Thus $\gamma>0$. The property $\textrm{dist}(\operatorname{spt} \omega, \partial \mathbb{R}^{2}_{+})>0$ follows from $\gamma>0$.

It remains to show (A). Let $\mu\leq \kappa_*$. We set the Lamb dipole $\omega^{L}_{1,W_L}$ with the speed $W_L$ in \eqref{eq:Lambspeed}. Then $|| \omega^{L}_{1,W_L}||_{L^{1}_{*}}=\mu$ and $|| \omega^{L}_{1,W_L}||_{L^{1}}\leq 1$ by Proposition \ref{p:Lambmass}. Since $\omega^{L}_{1,W_L}\in K_{\mu}\subset \tilde{K}_{\mu}$, it follows that 
\begin{align*}
E_{2}[\omega^{L}_{1,W_L}]\leq I_{\mu}\leq \tilde{I}_{\mu}=E_{2}[\omega^{L}_{1,W_L}].
\end{align*}
Hence $I_{\mu}=\tilde{I}_{\mu}$. If  $\omega\in K_{\mu}$ is a maximizer of $I_{\mu}$, it is also a maximizer of $\tilde{I}_{\mu}$. By the uniqueness of the maximizer of $\tilde{I}_{\mu}$ in Lemma \ref{p:uniquenesswithoutmass}, $\omega$ is a horizontal translation of $\omega^{L}_{1,W_L}$. Thus $\gamma=0$. By Proposition \ref{p:Lambmass}, the mass property \eqref{eq:massproperty} follows.
\end{proof}

\begin{remark}\label{r:properties}  
The maximum satisfies the following properties: 
\begin{align} 
&I_{0} = 0< I_{\mu}< I_{\alpha}<\infty,\quad 0<\mu<\alpha,  \label{eq:P1}\\
&I_\mu=\frac{1}{2j_{1,1}^{2}\pi}\mu^{2},\quad 0\leq \mu\leq \kappa_*, \label{eq:P2}\\
&I_\mu\lesssim \mu^{\varepsilon},\quad \kappa_*\leq \mu,\quad 0<\varepsilon\leq 1, \label{eq:P3}
\end{align}
The property \eqref{eq:P1} is shown in \cite[Lemma 2.3]{AC2019}. The property \eqref{eq:P2} follows from Theorem \ref{thm:threshold}. The growth bound \eqref{eq:P3} follows from the estimate \eqref{eq:KEE}. We will show in Lemma \ref{l:Lip} and Proposition \ref{p:convergencemax} that $I_{\mu}$ is globally Lipschitz continuous on $[0,\infty)$ and the leading term of $I_{\mu}$ is $(4\pi)^{-1}\log{(2 \mu)}$ as $\mu\to\infty$.  
\end{remark}

\begin{remark}
The stream function of the Lamb dipole \eqref{eq:Lamb} is given by
\begin{align}
\mathcal{G}\omega^{L}_{\lambda,W}-Wx_2
=\begin{cases}
\ &C_L J_1(\sqrt{\lambda}r)\sin\theta,\quad r\leq a,\\
\ &-W \displaystyle\left(r-\frac{a^{2}}{r}\right),\quad r>a,
\end{cases}
\label{eq:Lambstream}
\end{align}
with the constants 
\begin{align}
C_L=-\frac{2W}{\sqrt{\lambda}J_0(j_{1,1})},\quad a=\frac{j_{1,1}}{\sqrt{\lambda}}.  \label{eq:Lambstreamconst}
\end{align}
\end{remark}

\subsection{Uniform bound for maximizers}

We derive a uniform estimate for maximizers in the mass-saturated regime $\mu>\kappa_*$. The following formulas can be found in \cite[p.1062]{Tu83b}, \cite[Remark 2.6 (i)]{AC2019}, \cite{ACJSW}.
    
\begin{lemma}[Integral identities]\label{lemma:pohozaev}
Let $\mu,\lambda,\nu>0$. The following statements hold for $\omega\in S_{\mu,\nu,\lambda}$ and constants $W>0$ and $\gamma\ge0$ satisfying \eqref{eq:near-Lamb-fixed-EL}: 
\begin{align}
I_{\mu,\nu,\lambda}&= \frac12W\mu + \frac{\gamma}{2}\|\omega\|_{L^1(\mathbb R^2_+)}, \label{eq:pohozaev-penalized}\\
\|\omega\|_{L^2(\mathbb R^2_+)}^2 &= \lambda W\mu, \label{eq:pohozaev-L2} \\
E[\omega] &= W\mu + \frac{\gamma}{2}\|\omega\|_{L^1(\mathbb R^2_+)}, \label{eq:pohozaev-energy} \\
W &= \frac{1}{2\pi\|\omega\|_{L^1(\mathbb R^2_+)}} \iint_{\mathbb R^2_+\times\mathbb R^2_+} \frac{x_2+y_2}{|x-\bar y|^2} \omega(x)\omega(y)\,dx\,dy.\label{eq:pohozaev-W}
\end{align}
\end{lemma}
	
\begin{proof}
We include the proof for clarity. Multiplying \eqref{eq:near-Lamb-fixed-EL} by $\omega$ and integrating over $\mathbb R^2_+$ yields \eqref{eq:pohozaev-penalized}. We set $\psi=\mathcal G\omega$. Since $-\Delta\psi=\omega$ in $\mathbb R^2_+$ and $\psi=0$ on $\partial\mathbb R^2_+$, the standard Poho\v{z}aev computation gives
\begin{equation*}
\int_{\mathbb R^2_+} \omega\, x\cdot\nabla\psi\,dx = 0.
\end{equation*}
We apply this identity to compute 
\begin{align*}
\int_{\mathbb R^2_+} x\cdot \nabla (\psi-Wx_2-\gamma)_+^2 \,dx 
            &= 2 \int_{\mathbb R^2_+} (\psi-Wx_2-\gamma)_+ \,x\cdot\nabla(\psi-Wx_2-\gamma)\,dx \\
			&= \frac{2}{\lambda} \int_{\mathbb R^2_+} \omega\, x\cdot\nabla(\psi-Wx_2-\gamma)\,dx
			= -\frac{2W\mu}{\lambda}.
\end{align*}
The left-hand side equals 
\begin{equation*}
 -2\int_{\mathbb R^2_+}(\psi-Wx_2-\gamma)_+^2\,dx,
\end{equation*}
by the divergence theorem and we obtain \eqref{eq:pohozaev-L2}. The identity \eqref{eq:pohozaev-energy} follows from \eqref{eq:pohozaev-penalized} and \eqref{eq:pohozaev-L2}. Differentiating \eqref{eq:near-Lamb-fixed-EL} with respect to $x_2$, we find that
\begin{equation*}
\partial_{x_2}\omega
=\lambda \mathbf{1}_{(0,\infty)}(\psi-Wx_2-\gamma)(\partial_{x_2}\psi-W).
\end{equation*}
Multiplying by $\omega$ and integrating over $\mathbb R^2_+$, we obtain
\begin{equation*}
\int_{\mathbb R^2_+} \omega\,\partial_{x_2}\omega\,dx
=\lambda \int_{\mathbb R^2_+} \omega\,\partial_{x_2}\psi\,dx - \lambda W\|\omega\|_{L^1(\mathbb R^2_+)}.
\end{equation*}		
The left-hand side vanishes. Using
		\begin{equation*}
			\partial_{x_2}G(x,y) = \frac1{2\pi} \left( \frac{x_2+y_2}{|x-\bar y|^2} - \frac{x_2-y_2}{|x-y|^2} \right),
		\end{equation*}
the second term is expressed as 
\begin{align*}
			\int_{\mathbb R^2_+} \omega\partial_{x_2}\psi\,dx = \frac1{2\pi} \iint_{\mathbb R^2_+\times\mathbb R^2_+} \left( \frac{x_2+y_2}{|x-\bar y|^2} - \frac{x_2-y_2}{|x-y|^2} \right) \omega(x)\omega(y)\,dx\,dy.
\end{align*}
Since the kernel $(x_2-y_2)/|x-y|^2$ is antisymmetric with respect to $x$ and $y$, its integral against the symmetric product $\omega(x)\omega(y)$ vanishes identically. We obtain \eqref{eq:pohozaev-W}.
\end{proof}

\begin{lemma}
We have
\begin{align}
\sup_{\omega\in S_{\mu}}\left(||\omega||_{L^{p}(\mathbb{R}^{2}_{+})}+\left\|\frac{\mathcal{G}\omega}{x_2}\right\|_{L^{\infty}(\mathbb{R}^{2}_{+})}\right)\leq C, \quad 1\leq p\leq \infty,\quad \kappa_*\leq \mu.\label{eq:uniformvortex}
\end{align}
\end{lemma}

\begin{proof}
We first show the uniform estimate of the vorticity for $1\leq p\leq 2$. For $\omega\in S_{\mu}$, we set
$$
\omega_{s}(x_1,x_2)=s\omega(sx_1,x_2),\quad s>0.
$$
Observe that $\omega_{s}\in K_{\mu}$. Using 
\begin{align*}
E[\omega_s]&=\frac{1}{2}\iint_{\mathbb R^2_+\times\mathbb R^2_+}G\left(\frac{x_1}{s},\frac{y_1}{s},x_2,y_2\right)\omega(x)\omega(y)dxdy, \\
\int_{\mathbb{R}^{2}_{+}}\omega_s^{2}dx&=s\int_{\mathbb{R}^{2}_{+}}\omega^{2}dx,
\end{align*}
the maximality at $s=1$ yields  
\begin{align*}
0=\frac d{ds}E_2[\omega_{s}]\Bigg|_{s=1}
=\frac{1}{2}\iint_{\mathbb R^2_+\times\mathbb R^2_+}\partial_s G\left(\frac{x_1}{s},\frac{y_1}{s},x_2,y_2\right)\Bigg|_{s=1}\omega(x)\omega(y)dxdy
-\frac{1}{2}\int_{\mathbb{R}^{2}_{+}}\omega^{2}dx.
\end{align*}
Differentiating the Green function, we have
\begin{align*}
\partial_s G\left(\frac{x_1}{s},\frac{y_1}{s},x_2,y_2\right)\Bigg|_{s=1}
=\frac{2x_2y_2 |x_1-y_1|^{2}}{\pi |x-y|^{2}|x-\bar{y}|^{2}}\leq \frac{1}{2\pi}.
\end{align*}
Since $||\omega||_{L^{1}}=1$ for $\mu\geq \kappa_*$ by Theorem \ref{thm:threshold}, we obtain $||\omega||_{L^2}^2\leq1/(2\pi)$ and hence the uniform vorticity bound for $1\leq p\leq2$.

We infer from \eqref{eq:SFE} that $\|\mathcal{G}\omega\|_{L^\infty}\leq C\sqrt{\mu}$.
By the Hardy-Littlewood-Sobolev inequality, we then get $\|\nabla\mathcal{G}\omega\|_{L^4(\mathbb{R}^{2}_{+})}\leq C\|\omega\|_{L^{4/3}(\mathbb{R}^{2}_{+})}\leq C$. Moreover, we obtain $W=\|\omega\|_{L^2}^2/\mu\leq C$ by \eqref{eq:pohozaev-L2}   For  the zero extension of $\omega$ to $\mathbb{R}^{2}$, the chain rule and the Sobolev inequality on unit balls yield
\begin{align*}
\|\omega\|_{L^{\infty}(B_x(1))}&\leq C\left(\|\omega\|_{L^2(B_x(1))}+\|\nabla\omega\|_{L^4(B_x(1))}\right)\\
&\leq C\left(\|\omega\|_{L^2(\mathbb{R}^{2}_{+})}+\|\nabla\mathcal{G}\omega\|_{L^4(\mathbb{R}^{2}_{+})}+W\right)\leq C,
\end{align*}
uniformly in $x\in\mathbb{R}^{2}$ and $\mu\geq\kappa_*$. This proves the remaining $L^p$ bounds of $\omega$ by interpolation.

Finally, splitting the gradient kernel at $|x-y|=1$ gives $\|\nabla\mathcal{G}\omega\|_{L^{\infty}}\leq C(\|\omega\|_{L^1}+\|\omega\|_{L^{\infty}})\leq C$. Since $\mathcal{G}\omega(x_1,0)=0$, it follows that $\|\mathcal{G}\omega/x_2\|_{L^{\infty}}\leq C$. This proves \eqref{eq:uniformvortex}.
%By the Hardy-Littlewood--Sobolev inequality, $\mathcal{G}\omega$ is uniformly bounded in $W^{1,q}$ for $1/q=1/p-1/2$ for $1<p<2$. By the Sobolev inequality and \eqref{eq:near-Lamb-fixed-EL}, $\psi$ and $\omega$ are uniformly bounded in $L^{\infty}$. Thus the uniform bound \eqref{eq:uniformvortex} for $2\leq p\leq \infty$ holds. The stream function estimate in \eqref{eq:uniformvortex} follows from \eqref{eq:sinfty}.
\end{proof}

\subsection{Lipschitz continuity of the maximum}

We will use the continuity of the function $I_{\mu}$ for the compactness of maximizers in \S \ref{s:4}.

\begin{lemma}\label{l:Lip}
The function $I_{\mu}$ is globally Lipschitz continuous on
$[0,\infty)$.    
\end{lemma}

\begin{proof}
Since $I_{\mu}$ is quadratic in $[0,\kappa_*]$ by \eqref{eq:P2}, it suffices to show that $I_{\mu}$ is globally Lipschitz continuous on $[\kappa_*,\infty)$. For $\kappa_*\leq \mu<\alpha$ and $\omega\in S_{\alpha}$, $\omega_{\tau}(x)=\tau^{2}\omega(\tau x)$ satisfies 
\begin{align*}
||\omega_{\tau}||_{L^{2}}^{2}=\tau^{2}||\omega||_{L^{2}}^{2},\quad 
||\omega_{\tau}||_{L^{1}_{*}}=\frac{1}{\tau}||\omega||_{L^{1}_{*}},\quad 
||\omega_{\tau}||_{L^{1}}=||\omega||_{L^{1}},\quad E[\omega_{\tau}]=E[\omega].
\end{align*}
We choose $\tau=\alpha/\mu>1$ so that $\omega_{\tau}\in K_{\mu}$. It follows from \eqref{eq:uniformvortex} that 
\begin{align*}
I_{\alpha}=E_2[\omega]
=E[\omega]-\frac{1}{2}||\omega||_{L^{2}}^{2} 
=E[\omega_{\tau}]-\frac{1}{2}||\omega_{\tau}||_{L^{2}}^{2}+\frac{1}{2}(\tau^{2}-1)||\omega||_{L^{2}}^{2} 
\leq I_{\mu}+ C\frac{(\alpha-\mu)(\alpha+\mu)}{2\mu^{2}}.
\end{align*}
By the monotonicity \eqref{eq:P1}, if $\mu<\alpha\leq 2\mu$, then
\begin{align*}
0\leq
\frac{I_{\alpha}-I_{\mu}}{\alpha-\mu}
\lesssim
\frac{\alpha+\mu}{2\mu^{2}}
\leq \frac{3}{2\kappa_*}.
\end{align*}
On the other hand, if $\alpha>2\mu$, then $\alpha-\mu>\alpha/2$, and the growth estimate \eqref{eq:P3} yields, %for any fixed $0<\varepsilon<1$,
\begin{align*}
0\leq
\frac{I_{\alpha}-I_{\mu}}{\alpha-\mu}
\leq \frac{2I_{\alpha}}{\alpha}\leq C.
%\lesssim \frac{1}{\alpha^{1-\varepsilon}}\leq \frac{1}{\kappa_*^{1-\varepsilon}}.
\end{align*}
Thus the difference quotients of \(I_\mu\) are uniformly bounded on $[\kappa_*,\infty)$, and hence $I_{\mu}$ is globally Lipschitz continuous there. Together with its quadratic behavior on $[0,\kappa_*]$, this proves the assertion.
\end{proof}

%S3
\section{Coercive estimates for the linearized operators}\label{s:3}

We establish coercive estimates for the quadratic forms associated with the operators \(L\) on \(L^2(B)\), obtained by linearizing the Euler--Lagrange equation \eqref{eq:near-Lamb-fixed-EL} around the two explicit solutions \eqref{eq:radiallysymmetric} and \eqref{eq:Lamb}. For the radially symmetric solution, \(L\) is nonnegative but has a nontrivial kernel arising from symmetries. We obtain coercivity by eliminating these neutral directions through the zero-moment condition. In contrast, for the Lamb dipole, \(L\) has a negative eigenvalue. We nevertheless establish coercivity on the subspace determined by the \(x_1\)-symmetry and the zero-impulse condition.

\subsection{The Fredholm operator}

We consider the operator of the form  
\begin{align}
L=I-K, \quad 
K=\mathbf{1}_{B}(-\Delta_{\mathbb{R}^{2}})^{-1}\mathbf{1}_{B}, \label{eq:opL}
\end{align}
for a ball $B=B_0(R)$ with radius $R$ centered at the origin. Using the kernel $N=-(2\pi)^{-1}\log|x|$, $K$ is expressed as 
\begin{align}
Kf(x)=\mathbf{1}_{B}(x)\int_{B}N(x-y)f(y)dy. \label{eq:opKrep}
\end{align}

\begin{proposition}\label{p:spectrumL}
The following statements hold:
\begin{itemize}
\item[(i)] The operator $K: L^{2}(B)\to L^{2}(B)$ is a compact and symmetric operator. The nonzero spectrum consists of real eigenvalues $\{\mu_n\}$, with \(0\) as the only possible accumulation point.

\item[(ii)] $\sigma(L)=\{1-\mu:\mu\in \sigma(K)\}$.
\end{itemize}
\end{proposition}

\begin{proof}
By standard interior elliptic regularity, the operator $K:L^{2}(B)\to H^{2}(B)$ is bounded. The compactness follows from the Rellich theorem $H^{2}(B)\Subset L^{2}(B)$. The symmetry follows from the representation
\begin{align*}
(Kf,g)_{L^{2}(B)}=\iint_{B\times B}N(x-y)f(y)g(x)dydx=(f,Kg)_{L^{2}(B)}.
\end{align*}
Thus (i) holds, e.g., \cite[Appendix D, Theorem 6]{Evans2010}. If $\mu$ belongs to the resolvent set of $K$, for $f\in L^{2}$ there exists a unique $\psi\in L^{2}$ satisfying 
\begin{align*}
f=(\mu-K)\psi=(\mu-1+L)\psi.
\end{align*}
Thus \(1-\mu\) belongs to the resolvent set of \(L\). The converse also holds. Hence $\mu\in \sigma(K)$ iff $1-\mu\in \sigma(L)$. We proved (ii).
\end{proof}

\begin{proposition}
The following statements hold:
\begin{itemize}
\item[(i)] All eigenvalues associated with eigenfunctions of $L$ satisfying $\int_{\mathbb{T}}fd\theta=0$ for a.e. $0<r<R$ are given by 
\begin{align}
\lambda_{m,k}=1-\frac{R^{2}}{j_{m-1,k}^{2}},\quad m,k=1,2,\ldots,  \label{eq:eigenvalues}
\end{align}
where $j_{m-1,k}$ is the $k$-th positive zero of $J_{m-1}$.
\item[(ii)] If $\lambda<1$ is an eigenvalue associated with some radial eigenfunction of $L$, then
\begin{align}
J_0(\kappa R)+\kappa R\log{R}J_1(\kappa R)=0,\quad \kappa=\frac{1}{\sqrt{1-\lambda}}.\label{eq:eigenvalues2}
\end{align}
\end{itemize}
\end{proposition}

\begin{proof}
We take an eigenfunction $f\in L^{2}(B)$ associated with the eigenvalue $\lambda\in \mathbb{R}$ and set 
\begin{align*}
\psi(x)=(-\Delta_{\mathbb{R}^{2}})^{-1}(\mathbf{1}_Bf)=\int_{B}N(x-y)f(y)dy.
\end{align*}
By Fourier expansion,
\begin{align*}
f=\sum_{m\in \mathbb{Z}}\hat{f}_m(r)e^{im\theta},\quad \psi=\sum_{m\in \mathbb{Z}}\hat{\psi}_m(r)e^{im\theta}.
\end{align*}
We show (i). For $f$ satisfying $\hat{f}_0=(2\pi)^{-1/2}\int_{\mathbb{T}}fd\theta=0$, we may assume that $m\neq 0$. By $Lf=\lambda f$, $\psi$ satisfies 
\begin{align*}
\left(\Delta +\frac{1}{1-\lambda}\mathbf{1}_B\right)\psi=0.
\end{align*}
Then, $\hat{\psi}_m$ satisfies 
\begin{align*}
\left(\partial_r^{2}+\frac{1}{r}\partial_r-\frac{m^{2}}{r^{2}}+\frac{1}{1-\lambda}\mathbf{1}_{(0,R)}\right)\hat{\psi}_m=0.
\end{align*}
We may assume that $m\geq 1$. For $r>R$, $\hat{\psi}_m$ is decaying as $r\to\infty$. By solving the above equation, we obtain  $\hat{\psi}_m=C_1r^{-m}$. This implies $\nabla \psi_m\in L^{2}(\mathbb{R}^{2})$ for $\psi_m=\hat{\psi}_me^{im\theta}$. Using $-\Delta \psi_m=(1-\lambda)^{-1}\mathbf{1}_B\psi_m$, we find that $\lambda<1$. By solving Bessel's differential equation for $0<r<R$, we obtain $\hat{\psi}_{m}=C_2J_m(\kappa r)$ for $\kappa=(1-\lambda)^{-1/2}$. The \(C^1\)-matching condition at \(r=R\), together with the recurrence relation $zJ'_m(z)+mJ_m(z)=zJ_{m-1}(z)$, yields \(J_{m-1}(\kappa R)=0\). Thus $\kappa R=j_{m-1,k}$ for $k=1,2,\ldots$ and we obtain the eigenvalues \eqref{eq:eigenvalues}.

We show (ii). If $f=f(r)$ is an eigenfunction of $L$, $\psi=C_1\log r$ for $r>R$. Since $\lambda<1$, we have $\psi=C_2J_0(\kappa r)$. The $C^{1}$ continuity at $r=R$ and $J_0'=-J_1$ yield the condition \eqref{eq:eigenvalues2}.
\end{proof}

\begin{proposition}\label{p:L2Rbasis}
For $m\geq 1$, the functions 
\begin{align}
\left\{J_m\left(\frac{j_{m-1,k}}{R}r\right):k\geq 1\right\}  \label{eq:basis}
\end{align}
form a complete orthogonal basis of $L^{2}(0,R;rdr)$ solving Bessel's differential equation
\begin{align}
\left(\partial_r^{2}+\frac{1}{r}\partial_r-\frac{m^{2}}{r^{2}}+\frac{j_{m-1,k}^{2}}{R^{2}}\right)\varphi&=0,\quad 0<r<R,  \label{eq:Bessel}\\
R\varphi'(R)+m\varphi(R)&=0.  \label{eq:BC}
\end{align}
For $m=0$, $\{1\}$ and \eqref{eq:basis} form a complete orthogonal basis of $L^{2}(0,R;rdr)$ with $j_{-1,k}=j_{1,k}$.
\end{proposition}

\begin{proof}
We consider the operator 
\begin{align*}
\mathscr{L}\varphi=\left(\partial_r^{2}+\frac{1}{r}\partial_r-\frac{m^{2}}{r^{2}}\right)\varphi
=\frac{1}{r}\partial_r(r\partial_r \varphi)-\frac{m^{2}}{r^{2}}\varphi,
\end{align*}
and the associated bilinear form 
\begin{align*}
B(\varphi,\phi)=\int_{0}^{R}\left(\varphi'\phi'+\frac{m^{2}}{r^{2}}\varphi\phi \right)rdr+m\varphi(R)\phi(R).
\end{align*}
We define the weighted Hilbert space
\begin{align*}
V=\left\{\varphi\in L^{2}(0,R; rdr): \varphi',\frac{\varphi}{r}\in L^{2}(0,R; rdr) \right\},
\end{align*}
endowed with the inner product
\begin{align*}
(\varphi,\phi)_{V}=\int_{0}^{R}\left(\varphi'\phi'+\frac{1}{r^{2}}\varphi\phi \right)rdr.
\end{align*}
The bilinear form $B$ is coercive on $V$ and we have the compact embedding 
\begin{align*}
V\Subset L^{2}(0,R; rdr),
\end{align*}
as follows from the Rellich theorem away from \(r=0\) together with the estimate
\begin{align*}
\int_0^{\delta}|\varphi|^{2}rdr\leq
\delta^{2}\int_{0}^{\delta}\left|\frac{\varphi}{r}\right|^{2}rdr,\quad \delta>0.
\end{align*}
Since the bilinear form $B$ is coercive, for $f\in L^{2}(0,R;rdr)$, there exists a unique $\varphi\in V$ such that $B(\varphi,\phi)=(f,\phi)_{L^{2}(0,R;rdr)}$ by the Lax--Milgram theorem. The solution operator 
\begin{align*}
\mathscr{K}:L^{2}(0,R;rdr)\ni f\longmapsto \varphi\in V\subset L^{2}(0,R;rdr)
\end{align*}
is compact and self-adjoint. Hence, by the spectral theorem for compact self-adjoint operators, there exists a complete orthogonal basis $\{\varphi_k\}_{k=1}^{\infty}$ of $L^{2}(0,R;rdr)$ satisfying 
\begin{align*}
-\mathscr{L}\varphi_k=\kappa_{k}^{2}\varphi_k,
\end{align*}
and \eqref{eq:BC}. This equation is Bessel’s differential equation. Since the singular Bessel solution \(Y_m(\kappa_k r)\) does not belong to $V$, we have 
\begin{align*}
\varphi_k(r)=CJ_m\left(\kappa_{k}r\right).
\end{align*}
By \eqref{eq:BC} and the recurrence relation $zJ'_m(z)+mJ_m(z)=zJ_{m-1}(z)$, we obtain $\kappa_{k}=j_{m-1,k}/R$.

For $m=0$, we consider the space of average zero functions $L^{2}_{\textrm{av}}(0,R;rdr)$ and define 
\begin{align*}
V_{0}&=\left\{\varphi\in L^{2}(0,R; rdr): \varphi'\in L^{2}(0,R; rdr) \right\}, \\
V_{0,\textrm{av}}&=V_{0}\cap L^{2}_{\textrm{av}}(0,R;rdr).
\end{align*}
Then, the bilinear form $B$ for $m=0$ is coercive on $V_{0,\textrm{av}}$ and the same compactness and spectral argument applies on \(V_{0,\mathrm{av}}\). Thus the functions \eqref{eq:basis} form a complete orthogonal basis of $L^{2}_{\textrm{av}}(0,R;rdr)$. Since $J_{-1}=-J_1$, we have $j_{-1,k}=j_{1,k}$.
\end{proof}

\begin{lemma}\label{l:Fourier}
Set 
\begin{align}
L^{2}(B)=\bigoplus_{m\in \mathbb{Z} }H_m,\quad 
H_m&=\left\{f(r)e^{im\theta}:\ f\in L^{2}(0,R;rdr)\right\}. \label{eq:fourier}
\end{align}
There exists a complete orthonormal basis $\{f_{m,k}^{+},f_{m,k}^{-}:m,k\geq 1\}$ of $H_0^{\perp}=\oplus_{m\in \mathbb{Z}\backslash \{0\} }H_m$ consisting of eigenfunctions of the operator $L$ with the eigenvalues \eqref{eq:eigenvalues}.  
\end{lemma}

\begin{proof}
We set 
\begin{align*}
\psi_{m,k}^{+}=J_m\left(\frac{j_{m-1,k}}{R}r\right)e^{im\theta},\quad \psi_{m,k}^{-}=J_m\left(\frac{j_{m-1,k}}{R}r\right)e^{-im\theta},
\end{align*}
so that $\{\psi_{m,k}^{+}:k\geq 1\}$ and $\{\psi_{m,k}^{-}:k\geq 1\}$ are bases of $H_m$ and $H_{-m}$, respectively. The function $\psi_{m,k}^{+}$ satisfies 
\begin{align*}
-\Delta \psi_{m,k}^{+}=\frac{1}{1-\lambda_{m,k}}\psi_{m,k}^{+},
\end{align*}
in $B$, and is extendable to $\mathbb{R}^{2}$ as a $C^{1}$ function at $r=R$ and a harmonic function in $B^{c}$. We use the same symbol for the extended $\psi_{m,k}^{+}$ and set the function $f_{m,k}^{+}$ by 
\begin{align*}
f_{m,k}^{+}=-\Delta \psi_{m,k}^{+}=\frac{1}{1-\lambda_{m,k}}\mathbf{1}_{B}\psi_{m,k}^{+},
\end{align*}
so that $f_{m,k}^{+}$ is an eigenfunction of $L$. Similarly, we set $f_{m,k}^{-}$. Since $\{\psi_{m,k}^{+},\psi_{m,k}^{-}:m,k\geq 1\}$ is an orthogonal basis of $H_0^{\perp}$, we obtain the desired orthonormal basis by normalization.
\end{proof}

\begin{lemma}\label{l:singeigenvalues}
Let $\lambda_{m,k}$ be as in \eqref{eq:eigenvalues}. The following statements hold:
\begin{itemize}
\item[(i)] For $R=j_{0,1}$, $\lambda_{1,1}=0$ and $\lambda_{m,k}>0$ for $(m,k)\neq (1,1)$.
\item[(ii)] For $R=j_{1,1}$, $\lambda_{1,1}<0$, $\lambda_{2,1}=0$, $\lambda_{m,k}>0$ for $(m,k)\neq (1,1)$, $(2,1)$.
\end{itemize}
\end{lemma}

\begin{proof}
Using the standard ordering of the positive zeros of Bessel functions, the results follow from the three smallest numbers of $\{j_{m-1,k}:m,k\geq 1\}$: \begin{equation}
    \begin{split} j_{0,1}=2.4048\cdots<j_{1,1}=3.8317\cdots<j_{2,1}=5.1356\cdots<\cdots. \qedhere  
    \end{split}
\end{equation} 
\end{proof}

\subsection{The case \texorpdfstring{$R=j_{0,1}$}{R=j01}}

We first establish the coercive estimate of the operator $L$ for $R=j_{0,1}$ on the real-valued space $L^{2}(B)$. In this case, the kernel of $L$ is two-dimensional, and the zero-moment condition allows us to control the kernel component and obtain the desired coercive estimate. 

\begin{theorem}[Coercive estimate for $R=j_{0,1}$]\label{t:wholespacecoercive}
Let $L$ be the operator in \eqref{eq:opL}. Assume that $R=j_{0,1}$. Then, there exists $\delta>0$ such that 
\begin{align}
(Lf,f)_{L^{2}(B)}\geq \delta ||f||_{L^{2}(B)}^{2}  \label{eq:coercivej01}
\end{align}
for all $f\in L^{2}(B)$ satisfying 
\begin{align}
\int_{B}xfdx=0.  \label{eq:momentzero}
\end{align}
\end{theorem}

We consider the direct sum decomposition $L^{2}(B)=H_0\oplus H_0^{\perp}$ by the space of all radial functions $H_0=L^{2}(0,R;rdr)$. We show the coercive estimate \eqref{eq:coercivej01} on \(H_0\) and on \(H_0^\perp\cap(\operatorname{Ker}L)^\perp\), respectively. 

\begin{proposition}\label{p:coercive01}
The estimate \eqref{eq:coercivej01} holds on $H_0$.
\end{proposition}

\begin{proof}
We first observe that $\inf\sigma(L|_{H_0})>0$. Indeed, if there exists a non-positive eigenvalue $\lambda\leq 0$, the condition \eqref{eq:eigenvalues2} holds. Since $\kappa R=\kappa j_{0,1}\leq j_{0,1}<j_{1,1}$, $J_0(\kappa j_{0,1})\geq 0$ and $J_1(\kappa j_{0,1})>0$. The left-hand side of \eqref{eq:eigenvalues2} is positive by  $\log{j_{0,1}}>0$. This is a contradiction.

We take an orthonormal basis $\{e_n\}_{n=1}^{\infty}$ of $H_0$ consisting of eigenfunctions of $K$ with eigenvalues $\mu_n$. If $f\in \operatorname{Ker} K$, we find that $0=-\Delta Kf=f$. Thus $\operatorname{Ker} K=\{0\}$ and $\mu_n\neq 0$. By applying Proposition \ref{p:spectrumL} (ii) for $L|_{H_0}$, we find that 
\begin{align*}
0<\inf\sigma(L|_{H_0})=\inf\{1-\mu: \mu\in \sigma(K|_{H_0})\}=\inf\left\{1,\inf_{n\geq 1}\{1-\mu_n\}\right\}.
\end{align*}
This implies 
\begin{align*}
(Lf,f)_{L^{2}(B)}=||f||_{L^{2}(B)}^{2}-(Kf,f)_{L^{2}(B)}
=\sum_{n\geq 1}(1-\mu_n)|(f,e_n)_{L^{2}(B)}|^{2}\geq \inf_{n\geq 1}\{1-\mu_n\}||f||_{L^{2}(B)}^{2},
\end{align*}
and hence \eqref{eq:coercivej01} holds on $H_0$.
\end{proof}

\begin{proposition}\label{p:coercive02}

The functions 
\begin{align}
J_{m}\left(\frac{j_{m-1,k}}{R}r\right)\cos m\theta,\ J_{m}\left(\frac{j_{m-1,k}}{R}r\right)\sin m\theta,\quad m,k\geq 1,  \label{eq:basisL2}
\end{align}
form an orthogonal basis of $H_0^{\perp}$. In particular, 
\begin{align}
\textrm{Ker}\ L=\textrm{span}\left\{J_1\left(\frac{j_{0,1}}{R}r\right)\cos\theta,J_1\left(\frac{j_{0,1}}{R}r\right)\sin\theta\right\}.  \label{eq:KerL}
\end{align}
Moreover, the estimate \eqref{eq:coercivej01} holds on $H_0^{\perp}\cap (\textrm{Ker}\ L)^{\perp}$.
\end{proposition}

\begin{proof}
The functions \eqref{eq:basisL2} form a basis of $L^{2}(B)$ since they form a basis of the complex-valued $L^{2}$ space by Lemma \ref{l:Fourier}. Since $\lambda_{1,1}=0$ by Lemma \ref{l:singeigenvalues} (i), the representation \eqref{eq:KerL} holds.

For $f\in H_0^{\perp}\cap (\textrm{Ker}\ L)^{\perp}$, we consider the expansion 
\begin{align*}
f&=\sum_{(m,k)\neq (1,1)}J_{m}\left(\frac{j_{m-1,k}}{R}r\right)(a_{m,k}\cos m\theta+b_{m,k}\sin m\theta),\\
Lf&=\sum_{(m,k)\neq (1,1)}\lambda_{m,k}J_{m}\left(\frac{j_{m-1,k}}{R}r\right)(a_{m,k}\cos m\theta+b_{m,k}\sin m\theta).
\end{align*}
Since $\inf\{\lambda_{m,k}:(m,k)\neq (1,1)\}>0$ by Lemma \ref{l:singeigenvalues} (i), the estimate \eqref{eq:coercivej01} follows.
\end{proof}

\begin{proof}[Proof of Theorem \ref{t:wholespacecoercive}]
For $f\in L^{2}(B)$ satisfying the condition \eqref{eq:momentzero}, we consider the direct sum decomposition 
\begin{align*}
f=f_0+f_1+f_2\in H_0\oplus \textrm{Ker}L\oplus (H_0^{\perp}\cap (\textrm{Ker}L)^{\perp}).
\end{align*}
Since $Lf=Lf_0+Lf_2\in H_0\oplus (H_0^{\perp}\cap (\textrm{Ker}L)^{\perp})$, we apply Propositions \ref{p:coercive01} and \ref{p:coercive02} to estimate 
\begin{align*}
(Lf,f)_{L^{2}}=(Lf_0,f_0)_{L^{2}}+(Lf_2,f_2)_{L^{2}}\geq \delta_0\left(||f_0||_{L^{2}}^{2}+||f_2||_{L^{2}}^{2}\right),
\end{align*}
with some constant $\delta_0>0$. Using the basis of Ker $L$, we set 
\begin{align*}
f_1=ag(r)\cos\theta+bg(r)\sin\theta.
\end{align*}
Observe that 
\begin{align*}
||f_1||_{L^{2}}^{2}=C(|a|^{2}+|b|^{2}).
\end{align*}
Moreover, multiplying $f=f_0+f_1+f_2$ by $x_i$ and integrating it on $B$, we have 
\begin{align*}
0&=a(r,g(r)\cos^{2}\theta)_{L^{2}}+(x_1,f_2)_{L^{2}},\\
0&=b(r,g(r)\sin^{2}\theta)_{L^{2}}+(x_2,f_2)_{L^{2}}.
\end{align*}
Using H\"older's inequality, we obtain 
\begin{align*}
||f_1||_{L^{2}}\leq C||f_2||_{L^{2}},
\end{align*}
with some constant $C$ independent of $f$. Thus the coercive estimate \eqref{eq:coercivej01} holds.
\end{proof}

\subsection{The case \texorpdfstring{$R=j_{1,1}$}{R=j11}}

We establish a coercive estimate for the operator $L$ when $R=j_{1,1}$. In this case, the operator $L$ has a negative direction and a two-dimensional kernel on $L^{2}(B)$. We restrict $f$ to the symmetric class
below. These symmetries eliminate the kernel, while the zero-$x_2$-moment
condition controls the remaining negative direction.

\begin{theorem}[Coercive estimate for $R=j_{1,1}$]\label{t:halfspacecoercive}
Let $L$ be the operator in \eqref{eq:opL}. Assume that $R=j_{1,1}$. Then, there exists $\delta>0$ such that 
\begin{align}
(Lf,f)_{L^{2}(B)}\geq \delta ||f||_{L^{2}(B)}^{2}  \label{eq:coercivej11}
\end{align}
for all $f\in L^{2}(B)$ satisfying 
\begin{align}
f(x_1,x_2)&=f(-x_1,x_2),  \label{eq:x1even}\\
f(x_1,x_2)&=-f(x_1,-x_2),\label{eq:x2odd}\\
\int_{B}x_2fdx&=0.  \label{eq:momentzero2}
\end{align}
\end{theorem}

\begin{proposition}
Let $L^{2}_{\textrm{sy}}(B)$ denote the space of all functions in $L^{2}(B)$ satisfying \eqref{eq:x1even}-\eqref{eq:x2odd}. The functions 
\begin{align}
 J_{2l+1}\left(\frac{j_{2l,k}}{R}r\right)\sin((2l+1)\theta),\quad l\geq 0,k\geq 1,  \label{eq:basisL2sy}
\end{align}
form an orthogonal basis of $L^{2}_{\textrm{sy}}(B)$. Set 
\begin{align}
e_1=cJ_{1}\left(\frac{j_{0,1}}{R}r\right)\sin\theta,  \label{eq:e1}
\end{align}
with the constant $c$ satisfying  $||e_1||_{L^{2}}=1$. Then,
\begin{align}
L^{2}_{\textrm{sy}}(B)=Y\oplus Z,\quad
Y=\mathbb{R}e_{1},\quad Z=Y^{\perp}.
\end{align}
\end{proposition}

\begin{proof}
Since the functions \eqref{eq:basisL2} form an orthogonal basis of $H_0^{\perp}$, the conditions \eqref{eq:x1even}-\eqref{eq:x2odd} for $f\in L^{2}(B)$ imply that $f$ can be expanded in a sine series. 
\end{proof}

\begin{proposition}
The estimate \eqref{eq:coercivej11} holds on $Z$ with $\delta=\inf\{\lambda_{2l+1,k}: (l,k)\neq (0,1)\}$. Moreover, the inequality 
\begin{align}
|(Lg,h)_{L^{2}(B)}|\leq \sqrt{(Lg,g)_{L^{2}(B)}}\sqrt{(Lh,h)_{L^{2}(B)}} \label{eq:CS}
\end{align}
holds for $g,h\in Z$.
\end{proposition}

\begin{proof}
By Lemma \ref{l:singeigenvalues} (ii), the operator $L$ on $L^{2}_{\textrm{sy}}(B)$ has one negative eigenvalue $\lambda_{1,1}<0$ and trivial kernel. Thus the estimate \eqref{eq:coercivej11} follows from the eigenfunction expansions of $L$ on $Z$. By \eqref{eq:coercivej11}, the bilinear form $(Lg,h)_{L^{2}(B)}$ defines an inner product on $Z$, and \eqref{eq:CS} follows from the Cauchy--Schwarz inequality.
\end{proof}

\begin{proposition}
Set 
\begin{align}
\omega_-=\mathbf{1}_{B}\frac{2}{J_0(j_{1,1})}J_1(r)\sin\theta=(\omega_-,e_{1})_{L^{2}(B)}e_1+\omega^{\perp}_{-}\in Y\oplus Z. \label{eq:Lambminus}
\end{align}
Then, the following statements hold:
\begin{align}
L\omega_{-}&=\mathbf{1}_{B}x_2, \label{eq:Lambminus1}\\
(L\omega_{-},\omega_{-})_{L^{2}(B)}&=-2j_{1,1}^{2}\pi, \label{eq:Lambminus2}\\
(L\omega^{\perp}_{-},\omega^{\perp}_{-})_{L^{2}(B)}&<a|(\omega_-,e_1)_{L^{2}(B)}|^{2},\quad a=-\lambda_{1,1}. \label{eq:Lambminus3}
\end{align}
\end{proposition}

\begin{proof}
The function $\omega_{-}$ is the negative of the Lamb dipole $\omega_{1,1}^{L}$ \eqref{eq:Lamb} for $\lambda=W=1$. Namely, $\omega_{-}=-\omega_{1,1}^{L}$. Since $\omega_{1,1}^{L}=(\mathcal{G}\omega_{1,1}^{L}-x_2)_{+}$ and $\omega^{L}_{1,1}$ is $x_2$-odd symmetric,
\begin{align*}
\omega^{L}_{1,1}=(-\Delta_{\mathbb{R}^{2}})^{-1}\mathbf{1}_{B}\omega^{L}_{1,1}-x_2,\quad x\in B.
\end{align*}
Thus $L\omega^{L}_{1,1}=-\mathbf{1}_{B}x_2$ and \eqref{eq:Lambminus1} holds. By the impulse of the Lamb dipole \eqref{eq:massLamb} and \eqref{eq:Lambminus1}, we have 
\begin{align*}
2j_{1,1}^{2}\pi=2\int_{\mathbb{R}^{2}_{+}}x_2\omega_{1,1}^{L}dx=(x_2\mathbf{1}_{B},\omega^{L}_{1,1})_{L^{2}(B)}=-(L\omega_{-},\omega_{-})_{L^{2}(B)}.
\end{align*}
Thus \eqref{eq:Lambminus2} holds. Using $L\omega_-=-a(\omega_-,e_1)_{L^{2}}e_1+L\omega^{\perp}_{-}$ and \eqref{eq:Lambminus2}, it follows that 
\begin{align*}
0>(L\omega_-,\omega_{-})_{L^{2}(B)}=-a|(\omega_-,e_1)_{L^{2}(B)}|^{2}+(L\omega^{\perp}_{-},\omega^{\perp}_{-})_{L^{2}(B)}.
\end{align*}
Thus \eqref{eq:Lambminus3} holds.
\end{proof}

\begin{proof}[Proof of Theorem \ref{t:halfspacecoercive}]
We set $f=(f,e_1)_{L^{2}}e_1+g\in Y\oplus Z$. Since $L$ is symmetric, the condition \eqref{eq:momentzero2} implies
\begin{align*}
0=(x_2\mathbf{1}_B,f)_{L^{2}(B)}
=(L\omega_{-},f)_{L^{2}(B)}
=(\omega_{-},Lf)_{L^{2}(B)}.
\end{align*}
Using 
\begin{align*}
(\omega_{-},Lf)_{L^{2}(B)}
&=\left((\omega_{-},e_1)_{L^{2}(B)}e_1+\omega_{-}^{\perp},-a(f,e_1)_{L^{2}(B)}e_1+Lg\right)_{L^{2}(B)}\\
&=-a(\omega_{-},e_1)_{L^{2}(B)}(f,e_1)_{L^{2}(B)}+(\omega^{\perp}_-,Lg)_{L^{2}(B)},
\end{align*}
we obtain 
\begin{align*}
a(\omega_{-},e_1)_{L^{2}(B)}(f,e_1)_{L^{2}(B)}=(\omega^{\perp}_-,Lg)_{L^{2}(B)}.
\end{align*}
Applying the inequality \eqref{eq:CS} for $\omega^{\perp}_-, g\in Z$, we estimate 
\begin{align*}
|a(\omega_{-},e_1)_{L^{2}(B)}(f,e_1)_{L^{2}(B)}|^{2}=|(Lg,\omega^{\perp}_-)_{L^{2}(B)}|^{2}\leq (Lg,g)_{L^{2}(B)}(L\omega^{\perp}_-,\omega^{\perp}_-)_{L^{2}(B)}.
\end{align*}
By \eqref{eq:Lambminus3}, we set 
\begin{align*}
\sigma=\frac{(L\omega^{\perp}_{-},\omega^{\perp}_{-})_{L^{2}(B)}}{a|(\omega_-,e_1)_{L^{2}(B)}|^{2}}<1,
\end{align*}
and divide both sides by $a|(\omega_-,e_1)_{L^{2}}|^{2}$ to obtain 
\begin{align*}
a|(f,e_1)_{L^{2}(B)}|^{2}\leq \sigma (Lg,g)_{L^{2}(B)}.
\end{align*}
Using this inequality, we find that 
\begin{align*}
(Lf,f)_{L^{2}(B)}&=\left(-a(f,e_1)_{L^{2}(B)}e_1+Lg,(f,e_1)_{L^{2}(B)}e_1+g\right)_{L^{2}(B)} \\
&=-a|(f,e_1)_{L^{2}(B)}|^{2}+(Lg,g)_{L^{2}(B)} \\
&\geq (1-\sigma)(Lg,g)_{L^{2}(B)}.
\end{align*}
Since \eqref{eq:coercivej11} holds on $Z$, we estimate 
\begin{align*}
||f||_{L^{2}(B)}^{2}=|(f,e_1)_{L^{2}(B)}|^{2}+||g||_{L^{2}(B)}^{2}
\leq \left(\frac{\sigma}{a}+\frac{1}{\delta}\right)(Lg,g)_{L^{2}(B)}
\leq \frac{1}{1-\sigma}\left(\frac{\sigma}{a}+\frac{1}{\delta}\right)(Lf,f)_{L^{2}(B)}.
\end{align*}
We obtained \eqref{eq:coercivej11}.
\end{proof}

We state Theorem \ref{t:halfspacecoercive} in the half-plane setting for the application to the proof of Theorem \ref{thm:unique-near-lamb}.

\begin{theorem}[Coercive estimate in a half-disk]\label{t:halfspacecoercive2}
Define the operator $L_+:L^{2}(B_+)\to L^{2}(B_+)$ by
\begin{align}
L_{+}=I-\mathbf{1}_{B_+}(-\Delta_D)^{-1}\mathbf{1}_{B_+},\quad
B_+=B_0(j_{1,1})\cap \mathbb{R}^{2}_{+}.  \label{eq:L+}
\end{align}
Then, the estimate
\begin{align}
(L_+f,f)_{L^{2}(B_+)}\geq \delta ||f||_{L^{2}(B_+)}^{2}  \label{eq:coercive+}
\end{align}
holds for all $f\in L^{2}(B_+)$ satisfying \eqref{eq:x1even} and 
\begin{align}
\int_{B_+}x_2fdx=0.   \label{eq:impulsezero}
\end{align}
The constant $\delta$ is the same constant as in Theorem \ref{t:halfspacecoercive}.
\end{theorem}

\begin{proof}
We consider the odd extension of $f\in L^{2}(B_{+})$ to $\mathbb{R}^{2}$ by setting 
\begin{align*}
f^{\textrm{o}}(x_1,x_2)=
\begin{cases}
\ f(x_1,x_2)\quad &x_2\geq 0,\\
\ -f(x_1,-x_2)\quad & x_2<0.
\end{cases}
\end{align*}
The odd extension $f^{\textrm{o}}\in L^{2}(B)$ satisfies the conditions \eqref{eq:x1even}-\eqref{eq:momentzero2}. Using 
\begin{align*}
(-\Delta_{\mathbb{R}^{2}})^{-1}\mathbf{1}_{B}f^{\textrm{o}}|_{\mathbb{R}^{2}_{+}}
=E*(\mathbf{1}_{B}f^{\textrm{o}})|_{\mathbb{R}^{2}_{+}}
=\mathcal{G}(\mathbf{1}_{B_{+}}f)
=(-\Delta_{D})^{-1}(\mathbf{1}_{B_{+}}f),
\end{align*}
we find that 
\begin{align*}
Lf^{\textrm{o}}|_{\mathbb{R}^{2}_{+}}
=(f^{\textrm{o}}-\mathbf{1}_{B}(-\Delta_{\mathbb{R}^{2}})^{-1}\mathbf{1}_{B}f^{\textrm{o}})|_{\mathbb{R}^{2}_{+}}
=f-\mathbf{1}_{B_+}(-\Delta_{D})^{-1}\mathbf{1}_{B_+}f
=L_+f.
\end{align*}
Thus applying \eqref{eq:coercivej11} yields 
\begin{align*}
2(L_+f,f)_{L^{2}(B_+)}=(Lf^{\textrm{o}},f^{\textrm{o}})_{L^{2}(B)}\geq \delta ||f^{\textrm{o}}||_{L^{2}(B)}^{2}=2\delta||f||_{L^{2}(B_+)}^{2}.
\end{align*}
We obtain the desired result.
\end{proof}

\begin{remark}
In the study of quantitative stability of the Lamb dipole, the work \cite[Proposition 3.2]{LSZ26} shows the coercive estimate \eqref{eq:coercive+} for $f\in L^{2}(B_{+})$ satisfying the conditions \eqref{eq:momentzero2} and 
\begin{align}
\int_{B_+}\partial_{x_1}\omega^{L}_{1,1}fdx=0 \label{eq:kc},
\end{align}
with some constant $\delta_0>0$. Without the $x_1$ symmetry \eqref{eq:x1even}, the operator $L_+$ has one negative eigenvalue $\lambda_{1,1}<0$ and one-dimensional kernel corresponding to $\lambda_{2,1}=0$ with eigenfunctions $e_1$ and $\partial_{x_1}\omega^{L}_{1,1}$, respectively. Namely, 
\begin{align*}
L^{2}(B_+)=\mathbb{R}e_1\oplus \mathbb{R}\partial_{x_1}\omega^{L}_{1,1}\oplus Z_{+}.
\end{align*}
The proof of Theorem \ref{t:halfspacecoercive} also shows that \eqref{eq:coercive+} holds for $f\in \mathbb{R}e_1\oplus Z_+=(\mathbb{R}\partial_{x_1}\omega^{L}_{1,1})^{\perp}$ satisfying \eqref{eq:momentzero2}. Thus, the $x_1$-symmetry assumption \eqref{eq:x1even} in Theorem~\ref{t:halfspacecoercive2} can be replaced by
the orthogonality condition \eqref{eq:kc}.
\end{remark}

\section{Asymptotics near the Lamb dipole}\label{s:4}

We prove the uniqueness of the maximizer in the near-Lamb dipole regime in Theorem \ref{thm:unique-near-lamb} and the limiting profile in Theorem \ref{thm:well separated dipole} (i). We first use the compactness of maximizing sequences to show that, as $\mu_n\to\kappa_*^+$, the corresponding maximizers converge to the Lamb dipole. We then use the resulting uniform bounds to establish a uniform bound on the vortex cores, which in turn upgrades the convergence to uniform convergence. Combining this convergence with the coercivity of the linearized operator established in \S \ref{s:3}, we prove the uniqueness of the maximizer. Once uniqueness is established, the limiting profile in Theorem \ref{thm:well separated dipole} (i) follows without passing to subsequences.

\subsection{Convergence toward the Lamb dipole}\label{sec:ab}

We will use the following compactness result for maximizing sequences \cite[Lemma 3.5, Theorem 1.3]{AC2019} to prove the convergence of the maximum toward the Lamb dipole as $\mu\to\kappa_*^{+}$.

\begin{proposition}[Compactness of maximizing sequences]\label{p:compactness-maximizing}
Let $\mu>0$. Let $\{\omega_n\}\subset  L^{2}\cap L^{1}_{*} (\mathbb R^2_+)$ be a Steiner symmetric sequence satisfying $\omega_n\geq 0$, $||\omega_n||_{L^{1}}\leq 1$,   $\nrm{\omega_n}_{L^1_*}\to\mu$ and $E_2[\omega_n]\to I_\mu$. Then, this sequence is relatively compact in $L^2 \cap L^1_*(\mathbb R^2_+)$. 
\end{proposition}

\begin{lemma}\label{lem:convergenceLamb}
Let $\{\mu_n\}$ be a sequence such that $\mu_n\to \kappa_*^{+}$. Let $\omega_n\in S_{\mu_n}$ be a Steiner symmetric maximizer in Theorem \ref{thm:nearly Lamb} with the constants $W_n>0$ and $\gamma_n>0$. Then, passing to a subsequence,
\begin{align}
\omega_n&\to \omega^{L}_{1,W_L}\quad \textrm{in}\ L^{2}\cap L^{1}_{*}(\mathbb{R}^{2}_{+}),  \label{eq:convergenceLamb}\\
W_{n}&\to W_{L}=\frac{\kappa_*}{j_{1,1}^{2}\pi},\label{eq:convergenceLambW}\\
\gamma_{n}&\to 0. \label{eq:convergenceLambg}
\end{align}
Moreover, 
\begin{align}
&\mathcal{G}\omega_n \to \mathcal{G}\omega^{L}_{1,W_{L}} \quad \textrm{uniformly in}\ \overline{\mathbb{R}^{2}_{+}}, \label{eq:psiconvergence} \\
&\mathcal{G}\omega_n-W_nx_2-\gamma_n\to
\mathcal{G}\omega^{L}_{1,W_{L}}-W_{L}x_2
\quad \textrm{locally uniformly in}\ \overline{\mathbb{R}^{2}_{+}}.  \label{eq:Psiconvergence}
\end{align}
In particular,  
\begin{align}
\mathbf{1}_{(0,\infty)}(\mathcal{G}\omega_n-W_nx_2-\gamma_n)\to \mathbf{1}_{B_0(j_{1,1})_{+}}(x),\quad \textrm{a.e.}\ x\in \mathbb{R}^{2}_{+}. \label{eq:convergeLamb2}
\end{align}
\end{lemma}

\begin{proof}
By continuity of $I_{\mu}$ in Lemma \ref{l:Lip}, $E_2[\omega_n]=I_{\mu_n}\to I_{\kappa_*}$ as $\mu_n\to \kappa_*^{+}$. Thus, $\{\omega_n\}$ is a maximizing sequence for $I_{\kappa_*}$. By the compactness result in Proposition \ref{p:compactness-maximizing}, $\{\omega_n\}$ converges along a subsequence to a maximizer of $I_{\kappa_*}$ in $L^{2}\cap L^{1}_{*}$. By Theorem \ref{thm:threshold}, the maximizer of $I_{\kappa_*}$ is the Lamb dipole $\omega^{L}_{1,W_{L}}$. Thus the convergence \eqref{eq:convergenceLamb} follows. By the identities \eqref{eq:pohozaev-L2} and \eqref{eq:pohozaev-penalized}, we obtain
\begin{align*}
W_n&=\frac{||\omega_n||_{L^{2}(\mathbb{R}^{2}_{+})}^{2}}{\mu_n}
\to \frac{||\omega^{L}_{1,W_{L}}||_{L^{2}(\mathbb{R}^{2}_{+})}^{2}}{\kappa_*}=W_{L},\\
\gamma_n&=2I_{\mu_n}-W_n\mu_n\to 2I_{\kappa_*}-W_{L}\kappa_*=0. 
\end{align*}
The convergences \eqref{eq:psiconvergence} and \eqref{eq:Psiconvergence} follow from the stream function estimate \eqref{eq:SFE} and \eqref{eq:convergenceLamb}-\eqref{eq:convergenceLambg}. By the explicit form \eqref{eq:Lamb},  $\omega^{L}_{1,W_{L}}$ is supported on $B_0(j_{1,1})_{+}$. For a.e. \(x\in\mathbb R^2_+\), the limiting function \(\mathcal{G}\omega^L_{1,W_L}-W_Lx_2\) is nonzero; hence \eqref{eq:Psiconvergence} implies \eqref{eq:convergeLamb2}.
\end{proof}

\subsection{Boundedness of the vortex core}

We use upper and lower $L^2$ bounds for the maximizers to obtain a uniform bound on their vortex cores.

\begin{lemma}\label{lem:bddsupport}
Let $\{\omega_n\}$ be a sequence of maximizers in Lemma \ref{lem:convergenceLamb}. Then, there exists $R>0$ such that for all $n\geq 1$,
\begin{align}
\operatorname{spt} \omega_n\subset B_0(R)_{+}=B_0(R)\cap \mathbb{R}^{2}_{+}.   \label{eq:ubddspt}
\end{align}
\end{lemma}

\begin{proof}
Since the sequence $\{\omega_n\}$ converges in $L^{2}\cap L^{1}_{*}$, we have  
\begin{align*}
N\leq ||\omega_n||_{L^{2}}\leq M, 
\end{align*}
for some constants $M, N>0$. By the equation \eqref{eq:near-Lamb-fixed-EL}, we have $\omega_n\leq \psi_n$ for $\psi_n=\mathcal G\omega_n$. We apply \eqref{eq:SFE} to estimate 
\begin{align*}
\|\omega_n\|_{L^\infty} \le ||\psi_n||_{L^{\infty}} \leq C\mu_n^{1/2}M^{1/2}.
\end{align*}
By the Poho\v{z}aev identity \eqref{eq:pohozaev-L2},  
\begin{align*}
W_n=\frac{1}{\mu_n}||\omega_n||_{L^{2}}^{2}\geq 
\frac{N^{2}}{\mu_n}.
\end{align*}
Since $\psi_n-W_nx_2-\gamma_n>0$ on  $\operatorname{spt} \omega_n $, we find that 
\begin{equation*}
\frac{N^{2}}{\mu_n} x_2
\leq W_nx_2+\gamma_n
\leq \psi_n (x)
\leq C\mu_n^{1/2}M^{1/2}.
\end{equation*}
Since $\mu_n\to \kappa_*^{+}>0$, \(\operatorname{spt}\omega_n\) is uniformly bounded in the \(x_2\)-direction.

We take an arbitrary point $x=(x_1,x_2)\in \mathbb{R}^{2}_{+}$ satisfying $|x_1|>1$ and define a vertical strip of width $2\sqrt{|x_1|}$ centered at $x_1$ by
\begin{align*}
A_{x_1} = \left\{ y=(y_1,y_2) \in \mathbb{R}^2_+ : |y_1 - x_1| < \sqrt{|x_1|} \right\}.
\end{align*}
For $\omega=\omega_n$, we set 
\begin{align*}
\omega(y)
=\omega \mathbf{1}_{A_{x_1}}(y)+\omega(1- \mathbf{1}_{A_{x_1}})(y)=\omega_1(y)+\omega_2(y).
\end{align*}
Observe that non-increasing functions $f(t) \ge 0$ on $[0,\infty)$ satisfy  
\begin{align*}
\int_{t-a}^{t+a} f(s) \, ds \le 2a f(t-a)\leq \frac{2a}{t-a}||f||_{L^{1}(0,\infty)},\quad t>a>0. \label{eq:nonincreasing} 
\end{align*}
Applying this inequality for $f=\omega$, $a=\sqrt{|x_1|}$, and $t=x_1$ yields   
\begin{align*}
\int_{x_1-\sqrt{|x_1|}}^{x_1+\sqrt{|x_1|}} \omega(y_1,y_2) \, dy_1 \le \frac{1}{\sqrt{|x_1|}-1}\, ||\omega_n||_{L^{1}(\mathbb{R})}(y_2).
\end{align*}
Integrating this inequality for $ y_2>0$ gives 
\begin{align*}
||\omega_1||_{L^{1}_*(\mathbb{R}_+^{2})}
 \le \frac{1}{\sqrt{|x_1|}-1}\,||\omega||_{L^{1}_*(\mathbb{R}_+^{2})}.
\end{align*}
Similarly, we estimate 
\begin{align*}
||\omega_1||_{L^{2}(\mathbb{R}_+^{2})}
 \le \left(\frac{1}{\sqrt{|x_1|}-1}\right)^{\frac{1}{2}} ||\omega||_{L^{2}(\mathbb{R}_+^{2})}.
\end{align*}
Applying the stream function estimate \eqref{eq:sinfty} for $\psi_1=\mathcal{G}\omega_1$ yields 
\[
||\nabla \psi_{1}||_{L^{\infty}(\mathbb{R}^{2}_{+})}
\le C
\|\omega_{1}\|_{L^1_*(\mathbb{R}^{2}_{+})}^{\frac{1}{4}}
\|\omega_{1}\|_{L^2(\mathbb{R}^{2}_{+})}^{\frac{1}{4}} \|\omega_{1}\|_{L^\infty(\mathbb{R}^{2}_{+})}^{\frac{1}{2}}.
\]
We thus obtain
\[
\frac{\psi_{1}(x)}{x_2} \le C
 \left(\frac{1 }{\sqrt{|x_1|-1}} \right)^{\frac{3}{8}}\mu_n^{1/2}M^{1/2}.
\]
For $y \notin A_{x_1}$, we have $|x-y| \ge |x_1 - y_1| \ge \sqrt{|x_1|}$. Using \eqref{eq:GFE} for $\alpha=1$, we estimate $\psi_2=\mathcal{G}\omega_2$ by 
\[
\psi_{2}(x) \le \frac{1}{4\pi} \int_{\mathbb{R}^2_+ \setminus A_{x_1}} \frac{4 x_2 y_2}{|x-y|^2} \omega (y) \, dy \le \frac{1}{\pi} \frac{x_2}{|x_1|} \int_{\mathbb{R}^2_+} y_2 \omega (y) \, dy = \frac{\mu_n x_2}{\pi|x_1|}.
\]
For $x = (x_1, x_2) \in \operatorname{spt} \omega_n$ satisfying $|x_1|>1$, we find that 
\begin{align*}
\frac{N^{2}}{\mu_n}\leq \frac{\psi_n(x)}{x_2}
\leq C 
 \left(\frac{1 }{\sqrt{|x_1|-1}} \right)^{\frac{3}{8}}\mu_n^{1/2}M^{1/2}+\frac{\mu_n}{\pi|x_1|},
\end{align*}
and \(\operatorname{spt}\omega_n\) is uniformly bounded in the \(x_1\)-direction.
\end{proof}

\begin{proposition}\label{p:strongerconvergence0}
The sequence $\{\omega_n\}$ in Lemma \ref{lem:convergenceLamb} satisfies
\begin{align}
\omega_n&\to \omega^{L}_{1,W_L}\quad  \textrm{uniformly in}\  \overline{\mathbb{R}^{2}_{+}},  \label{eq:omegauniform}\\
\nabla \mathcal{G}\omega_n&\to \nabla \mathcal{G}\omega^{L}_{1,W_{L}}\quad  \textrm{uniformly in}\  \overline{\mathbb{R}^{2}_{+}}.\label{eq:nablapsiuniform}
\end{align}    
\end{proposition}

\begin{proof}
It follows from \eqref{eq:near-Lamb-fixed-EL} that 
\begin{align*}
\omega_n-\omega^{L}_{1,W_{L}}
=&(\mathcal{G}\omega_n-W_nx_2-\gamma_n)_+
-(\mathcal{G}\omega^{L}_{1,W_{L}}-W_{L}x_2)_+ \\
=&\int_{0}^{1}\mathbf{1}_{(0,\infty)}(t(\mathcal{G}\omega_n-W_nx_2-\gamma_n)+(1-t)(\mathcal{G}\omega^{L}_{1,W_{L}}-W_{L}x_2))dt\\
&\cdot \left(\mathcal{G}\omega_n-W_nx_2-\gamma_n- (\mathcal{G}\omega^{L}_{1,W_{L}}-W_{L}x_2)\right).
\end{align*}
By the local uniform convergence of the stream function \eqref{eq:Psiconvergence}, $\omega_n$ converges to $\omega^{L}_{1,W_{L}}$ locally uniformly in $\overline{\mathbb{R}^{2}_{+}}$. Since $\omega_n$ is supported on $B_0(R)_{+}$ uniformly for all $n$ by Lemma \ref{lem:bddsupport}, $\omega_n$ converges to $\omega^{L}_{1,W_{L}}$ uniformly in $\overline{\mathbb{R}^{2}_{+}}$. The convergence \eqref{eq:nablapsiuniform} follows from \eqref{eq:sinfty}.
\end{proof}

\subsection{The distance estimate for the vortex core}\label{subsec:gap}
	
\begin{lemma}\label{lem:gap}
		There exists a constant $C>0$ such that
		\begin{align}
		C^{-1}\gamma_n\le\operatorname{dist}(\operatorname{spt}\omega_n,\partial\mathbb R^2_+)\le C\gamma_n.  \label{eq:gap}
		\end{align}
	\end{lemma}
    
\begin{proof}
Observe from \eqref{eq:near-Lamb-fixed-EL} that 
\begin{equation*}\
\{x\in\mathbb R^2_+:\omega_n>0\}=\left\{x\in\mathbb R^2_+:\frac{\mathcal{G}\omega_n}{x_2}-W_n>\frac{\gamma_n}{x_2}\right\}.
\end{equation*}
By the convergence \eqref{eq:nablapsiuniform}, $\mathcal{G}\omega_n/x_2$ is uniformly bounded in $\mathbb{R}^{2}_{+}$. Thus for $x\in \operatorname{spt} \omega_n$,
\begin{equation*}
\frac{\gamma_n}{x_2}\le C,
\end{equation*}
with some constant $C$ independent of $n$. This yields the lower bound in \eqref{eq:gap}.

By the explicit form of the stream function of the Lamb dipole \eqref{eq:Lambstream} for $r\leq j_{1,1}$ and $J'_1(0)=1/2$, we observe that  
\begin{align*}
\frac{\mathcal{G}\omega^{L}_{1,W_{L}}-W_{L}x_2}{x_2}=\frac{J_1(r)}{r}\left(-\frac{2W_{L}}{J_0(j_{1,1})}\right)\to \left(-\frac{W_{L}}{J_0(j_{1,1})}\right)=:c>0\quad \textrm{as}\ r\to 0.
\end{align*}
Since $\mathcal{G}\omega_n/x_2-W_n$ converges to $\mathcal{G}\omega^{L}_{1,W_{L}}/x_2-W_{L}$ uniformly in $\overline{\mathbb{R}^{2}_{+}}$ by \eqref{eq:nablapsiuniform}, there exists $R>0$ such that 
\begin{align*}
\frac{\mathcal{G}\omega_n-W_nx_2}{x_2}>\frac{c}{2},
\end{align*}
for all $x\in B_0(R)_{+}$ and $n$. The constant
\begin{align*}
b_n=\frac{2\gamma_n}{c},
\end{align*}
vanishes as $n\to\infty$ by \eqref{eq:convergenceLambg}. Thus $(0,b_n)\in B_0(R)_{+}$ for large $n$ and 
\begin{align*}
\frac{(\mathcal{G}\omega_n)(0,b_n)-W_nb_n}{b_n}>\frac{c}{2}.
\end{align*}
Namely, we have
\begin{align*}
(\mathcal{G}\omega_n)(0,b_n)-W_nb_n>\gamma_n.
\end{align*}
Hence, $(0,b_n)\in \operatorname{spt} \omega_n$. We conclude that the upper bound in \eqref{eq:gap} also holds.
\end{proof}

\subsection{Uniqueness of maximizers}
	We now prove the uniqueness of the maximizers in Theorem \ref{thm:unique-near-lamb} near the Lamb dipole regime.	

\begin{proposition}\label{p:identitydifference}
Let $\mu>\kappa_*$. For two maximizers $\omega^{i}\in S_{\mu}$ in Theorem \ref{thm:nearly Lamb} with constants $W^{i},\gamma^{i}>0$ for $i=1,2$, we have
\begin{equation}
\begin{aligned}
\omega^{1}-\omega^{2}
=& \left(\int_{0}^{1}\mathbf{1}_{(0,\infty)}\left(t\left(\mathcal{G}\omega^{1}-W^{1}x_2-\gamma^{1}\right)+(1-t)\left(\mathcal{G}\omega^{2}-W^{2}x_2-\gamma^{2}\right)\right)dt\right)\\
&\cdot\left(\mathcal{G}(\omega^1-\omega^{2})-\left(\frac{1}{\mu}x_2-1\right)\left(\omega^{1}-\omega^{2},\omega^{1}+\omega^{2}\right)_{L^{2}(\mathbb{R}^{2}_{+})}\right).
\end{aligned}
\label{eq:identitydifference}
\end{equation}
\end{proposition}

\begin{proof}
Since mass is saturated for $\mu>\kappa_*$, the identities \eqref{eq:pohozaev-penalized} and \eqref{eq:pohozaev-L2} imply 
\begin{align*}
&I_{\mu}=\frac{W^{i}}{2}\mu+\frac{\gamma^{i}}{2},\\
&||\omega^{i}||_{L^{2}(\mathbb{R}^{2}_{+})}^{2}=W^{i}\mu.
\end{align*}
Taking the difference, we have 
\begin{align*}
W^{1}-W^{2}&=\frac{1}{\mu}(\omega^{1}-\omega^{2},\omega^{1}+\omega^{2})_{L^{2}(\mathbb{R}^{2}_{+})},\\
\gamma^{1}-\gamma^{2}&=-(\omega^{1}-\omega^{2},\omega^{1}+\omega^{2})_{L^{2}(\mathbb{R}^{2}_{+})}.
\end{align*}
It follows that  
\begin{align*}
\omega^{1}-\omega^{2}
=&\left(\int_{0}^{1}\mathbf{1}_{(0,\infty)}\left(t\left(\mathcal{G}\omega^{1}-W^{1}x_2-\gamma^{1}\right)+(1-t)\left(\mathcal{G}\omega^{2}-W^{2}x_2-\gamma^{2}\right)\right)dt\right)\\
&\cdot \left(\mathcal{G}\left(\omega^{1}-\omega^{2}\right)-\left(W^{1}-W^{2}\right)x_2-\left(\gamma^{1}-\gamma^{2}\right)\right).
\end{align*}
Substituting the above expressions for $W^{1}-W^{2}$ and $\gamma^{1}-\gamma^{2}$ yields the desired expression.
\end{proof}

\begin{lemma}\label{uniq-maximizer}
		There exists $\delta_{*}>0$ such that for $\mu \in  (\kappa_*,\kappa_*+\delta_{*}]$, there exists a unique Steiner symmetric $\omega$ such that 
        \begin{align}
        S_{\mu}=\{\omega(\cdot +(a,0)): a\in \mathbb{R}\}.  \label{eq:uniqueness}
        \end{align}
		\end{lemma}
	
\begin{proof}
Suppose, for contradiction, that for any $n\geq 1$ there exists $\kappa_*< \mu_n\leq \kappa_*+1/n$ such that there exist two Steiner symmetric maximizers $\omega_{n}^{1}, \omega_{n}^{2}\in S_{\mu_n}$ with the constants $W^{i}_n, \gamma^{i}_{n}>0$. By properties of the maximizers, each $\omega^{i}_{n}$ satisfies 
\begin{align*}
I_{\mu_n}=E_{2}[\omega_n^{i}],\quad
\mu_n=\int_{\mathbb{R}^{2}_{+}}x_2\omega_{n}^{i}dx, \quad
1=\int_{\mathbb{R}^{2}_{+}}\omega_{n}^{i}dx, \quad
\omega^{i}_{n}=(\mathcal{G}\omega^{i}_{n}-W^{i}_nx_2-\gamma_n^{i})_+.
\end{align*}
By Lemmas \ref{lem:convergenceLamb} and \ref{lem:bddsupport}, along a subsequence, $\omega_n^{i}$ converges to the Lamb dipole  $\omega^{L}_{1,W_{L}}$ uniformly in $\overline{\mathbb{R}^{2}_{+}}$ and  
\begin{align*}
\operatorname{spt} \omega^{i}_{n}\subset B_+=B_0(R)_{+},
\end{align*}
for some $R>0$. We set 
\[
\zeta_n = \frac{\omega_{n}^{1}-\omega_{n}^{2}}{\|\omega_{n}^{1}-\omega_{n}^{2}\|_{L^2(B_+)}},
\]
so that 
\begin{align*}
||\zeta_n||_{L^{2}(B_+)}=1,\quad \int_{B_+}x_2 \zeta_n dx=0,\quad \int_{B_+}\zeta_n dx=0,\quad \zeta_n(x_1,x_2)=\zeta_n(-x_1,x_2).
\end{align*}
We apply the identity \eqref{eq:identitydifference} and observe that $\zeta_n$ satisfies 
\begin{equation*}
\begin{aligned}
\zeta_n
=& \left(\int_{0}^{1}\mathbf{1}_{(0,\infty)}\left(t\left(\mathcal{G}\omega^{1}_n-W^{1}_nx_2-\gamma^{1}_n\right)+(1-t)\left(\mathcal{G}\omega^{2}_n-W^{2}_nx_2-\gamma^{2}_n\right)\right)dt\right)\\
&\cdot\left(\mathcal{G}\zeta_n-\left(\frac{1}{\mu_n}x_2-1\right)\left(\zeta_n,\omega^{1}_n+\omega^{2}_n\right)_{L^{2}(B_+)}\right).
\end{aligned}
\end{equation*}
Passing to a subsequence, $\zeta_n$ converges to a limit $\zeta$ weakly in $L^{2}(B_+)$. Since $\zeta_n$ is bounded in $L^{2}(B_+)$, $\mathcal{G}\zeta_n$ is bounded in $L^{\infty}(\mathbb{R}^{2}_{+})$ by \eqref{eq:SFE}. By the above expression, $\zeta_n$ is uniformly bounded in $B_+$. Moreover, the Green function estimate \eqref{eq:GFE} implies that $\mathcal{G}\zeta_n(x)$ converges to $\mathcal{G}\zeta(x)$ for each $x\in B_+$. 

The convergence \eqref{eq:Psiconvergence} for $\omega^{i}_n$ implies that for each $0\leq t\leq 1$,
\begin{align*}
\mathbf{1}_{(0,\infty)}\left(t\left(\mathcal{G}\omega^{1}_n-W^{1}_nx_2-\gamma^{1}_n\right)+(1-t)\left(\mathcal{G}\omega^{2}_n-W^{2}_nx_2-\gamma^{2}_n\right)\right)
\to \mathbf{1}_{B_0(j_{1,1})_{+}}(x)\quad \textrm{a.e.}\ x\in \mathbb{R}^{2}_{+}. 
\end{align*}
Thus letting $n\to\infty$ implies 
\begin{align*}
\zeta_n\to \zeta= \mathbf{1}_{B_0(j_{1,1})_{+}}(x) \left(\mathcal{G}\zeta(x)-2\left(\frac{1}{\kappa_*}x_2-1\right)(\zeta,\omega^{L}_{1,W_{L}})_{L^{2}(B_+)}\right)\quad \textrm{a.e.}\ x\in \mathbb{R}^{2}_{+}.
\end{align*}
In particular, $\zeta_n\to \zeta$ in $L^{2}(B_+)$. Since $\zeta$ is supported on $\overline{B_{0}(j_{1,1})_{+}}$, this implies 
\begin{align*}
||\zeta||_{L^{2}(B_{0}(j_{1,1})_{+})}=1,\quad \int_{B_{0}(j_{1,1})_{+}}x_2 \zeta dx=0,\quad \int_{B_{0}(j_{1,1})_{+}}\zeta dx=0,\quad \zeta(x_1,x_2)=\zeta(-x_1,x_2).
\end{align*}
The limit $\zeta$ satisfies 
\begin{align*}
L_+\zeta=-2\left(\frac{1}{\kappa_*}x_2-1\right)(\zeta,\omega^{L}_{1,W_{L}})_{L^{2}(B_{0}(j_{1,1})_{+})},
\end{align*}
for the operator $L_{+}$ in \eqref{eq:L+}. In particular, using the impulse and mass conditions, we have 
\begin{align*}
(L_+\zeta,\zeta)_{L^{2}(B_{0}(j_{1,1})_{+})}=0.
\end{align*}
Since $\zeta$ is even in $x_1$ and supported on $B_{0}(j_{1,1})_+$, we apply the coercive estimate \eqref{eq:coercive+} and conclude that $\zeta=0$. This yields a contradiction. We proved \eqref{eq:uniqueness}. 
\end{proof}

\begin{proof}[Proof of Theorem \ref{thm:unique-near-lamb} for $\kappa_{*}<\kappa\leq \kappa_{*}+\delta_{*}$ and of Theorem \ref{thm:well separated dipole} (i)]
The uniqueness follows from Lemma \ref{uniq-maximizer}. The limiting profile follows from Lemmas \ref{lem:convergenceLamb}, \ref{lem:bddsupport}, \ref{lem:gap} and Proposition \ref{p:strongerconvergence0}.
\end{proof}

\section{Asymptotics at large impulse}\label{s:5}
	
In contrast to the near-Lamb dipole regime, in the large-impulse regime we first establish a uniform bound on the vortex cores of the maximizers as $\mu\to\infty$. This uniform localization allows us to show that the recentered maximizers form a maximizing sequence for the logarithmic free-energy maximization problem \eqref{eq:MFE}. Its compactness and uniqueness then yield convergence to the radially symmetric solution.

\subsection{Outline of the vortex core bound}\label{ss:5.1}

The main step of the proof of uniqueness (Theorem \ref{thm:unique-near-lamb}) is the uniform vortex core bound 
\begin{align}
\operatorname{spt}\hat{\omega}\subset B_0(R),\quad \hat{\omega}(x)=\omega(x +(0,\mu)), \label{eq:uniformvortexcorebound1}
\end{align}
with some $R>0$ for Steiner symmetric maximizers $\omega\in S_{\mu}$ to \eqref{eq: VP-2} and all large $\mu$. We prove this property in the following procedure.

\subsubsection{The lower bound for the flux constant}

We first show the lower bound for the constant $\gamma$:
\begin{align}
\gamma\geq \frac{1}{2\pi}\log(2\mu)-C,  \label{eq:gamma lower bound}
\end{align}
using the identities \eqref{eq:pohozaev-penalized} and
\eqref{eq:pohozaev-L2} together with the uniform estimate
\eqref{eq:uniformvortex}.

\subsubsection{Rough support bounds}

We then show the growth estimate of the vortex core 
\begin{equation}
D_1=\sup_{x\in \operatorname{spt}\omega} |x_1|\lesssim \mu^a,\quad 
D_2=\sup_{x\in \operatorname{spt}  \omega}x_2\lesssim \mu, \label{eq:rough uniform bounds}
\end{equation}
for any $a>8$. This shows that horizontal growth is at most polynomial. The vertical growth is optimal. These growth bounds follow from \eqref{eq:gamma lower bound} and the stream function estimates. 

\subsubsection{The local mass bound}

We first upgrade the horizontal bound for $a=0$. A key step of the proof is to estimate the local mass from below with the dimensionless constant
\begin{align}
Q=\max\left\{e,\sqrt{\left(\frac{D_1}{\mu}\right)^{2}+\left(\frac{D_2}{\mu}\right)^{2}}\right\}.  \label{eq:Q}
\end{align}
This constant is not a priori bounded for all $\mu$. We uniformly estimate a local mass from below by 
\begin{equation}\label{eq:large-local-mass}
\inf\left\{\int_{B_x({A Q^{1+\delta}})}\omega(y)\,dy:\ x\in\operatorname{spt}\omega,\ \mu\geq M\right\}\ge \frac{1}{2}\left(1-\frac{1}{1+\delta}\right),
\end{equation}
for every $\delta>0$ with some $A=A_{\delta}\geq 1$. We obtain this local mass bound by estimating $N*\omega$ from below by $-\log Q$ using \eqref{eq:gamma lower bound} and from above by $-\log{R}\int_{|x-y|\geq R}\omega dx$ and taking $R=AQ^{1+\delta}$.

\subsubsection{Packing argument}

We choose $N$ disjoint balls $\{B_{x_j}(R)\}_{j=1}^{N}$ with centers lying on a horizontal line intersecting
$\operatorname{spt}\omega$ and apply the local mass bound \eqref{eq:large-local-mass} to bound $D_1$ uniformly for all $\mu$; see Figure \ref{fig:horizontal-bound}. The uniform boundedness of $D_1$ implies a uniform bound for $Q$ and upgrades the radius of the local mass bound by a uniformly bounded constant $R_0$; namely we obtain 
\begin{equation}\label{eq:large-local-mass2}
\inf\left\{\int_{B_x(R_0)}\omega(y)\,dy:\ x\in\operatorname{spt}\omega,\ \mu\geq M\right\}\ge m_0.
\end{equation}

\begin{figure}[h]
    \centering

    \begin{subfigure}[b]{0.63\textwidth}
        \centering
        \includegraphics[width=\textwidth]{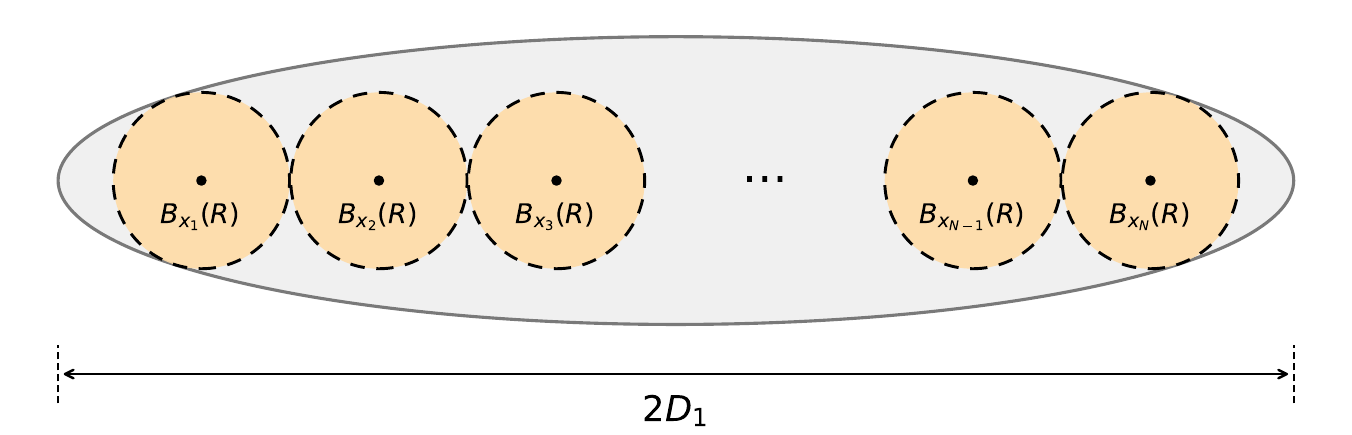}
        \caption{Horizontal bound}
        \label{fig:horizontal-bound}
    \end{subfigure}
    \hfill
    \begin{subfigure}[b]{0.35\textwidth}
        \centering
        \includegraphics[height=5.8cm]{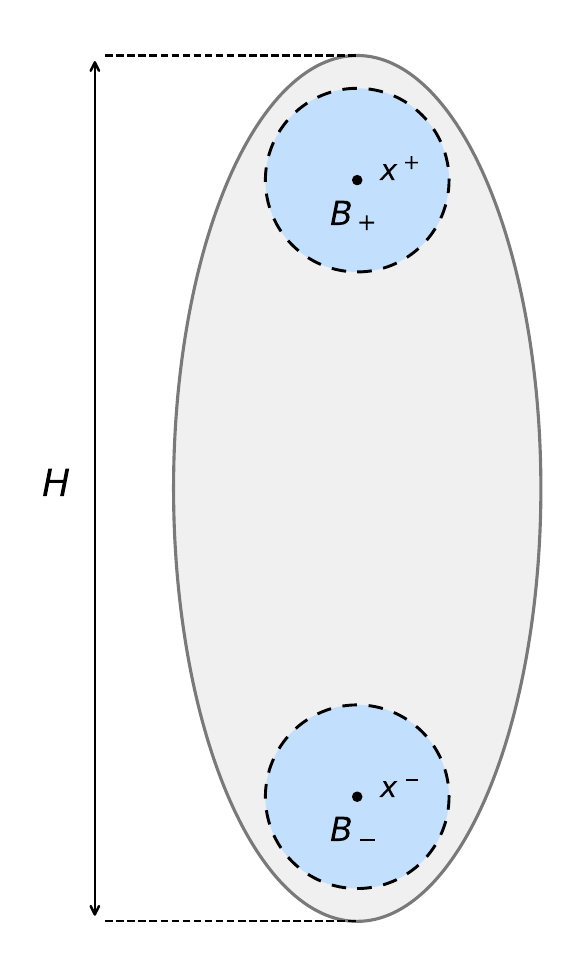}
        \caption{Vertical bound}
        \label{fig:vertical-bound}
    \end{subfigure}

    \caption{Schematic illustration of the vortex core localization argument}
    \label{fig:vortex-core-bound}
\end{figure}

\subsubsection{The lower bound for the free energy}

It remains to control the vertical extent of $\operatorname{spt}\hat{\omega}$. We show a lower bound for the free energy
\begin{align}
-C_1\leq \inf\{\mathcal{E}_2[\hat{\omega}]:\mu\geq M\}.\label{eq:lowerboundE2}
\end{align}
We estimate $I_{\mu}=E_{2}[\omega]$ from above by $\mathcal{E}_2[\hat{\omega}]+(2\pi)^{-1}\log(2\mu)+C$ and from below by $E_2[\check{\omega}^{R}_{1,\gamma_R}]$ using the maximality of $I_{\mu}$ and translation of the radially symmetric solution $\check{\omega}^{R}_{1,\gamma_R}$.

\subsubsection{The upper bound for the free energy}

The final step is the upper bound for the free energy 
\begin{align}
\mathcal{E}_2[\hat{\omega}]\leq C_2-\frac{m_0^{2}}{2\pi}\log{H},   \label{eq:upperboundE2}
\end{align}
using the vertical length of the vortex core
\begin{align}
H=\sup_{x\in \operatorname{spt}\hat{\omega}}x_2-\inf_{x\in \operatorname{spt}\hat{\omega}}x_2.
\end{align}
See Figure \ref{fig:vertical-bound}. We use the updated local mass bound \eqref{eq:large-local-mass2} to estimate the negative quantity in $N*\hat{\omega}$ from above by $-\log H$. The lower bound \eqref{eq:lowerboundE2} and the upper bound \eqref{eq:upperboundE2} imply $H$ is uniformly bounded for all $\mu$. This yields the desired bound  \eqref{eq:uniformvortexcorebound1} since $\operatorname{spt}\hat{\omega}$ intersects both sides of the $x_1$-axis by $\int_{\mathbb{R}^{2}}x_2\hat{\omega}dx=0$.

\subsection{Horizontal vortex core bound}

\begin{proposition}
    Let $\mu>\kappa_*$. Let $\omega\in S_{\mu}$ be a Steiner symmetric maximizer of \eqref{eq: VP-2} satisfying \eqref{eq:near-Lamb-fixed-EL} with $W,\gamma>0$ in Theorem \ref{thm:nearly Lamb}. There exist constants $M,C>0$ such that \eqref{eq:gamma lower bound} holds for $\mu\geq M$.
\end{proposition}

\begin{proof}
By identities \eqref{eq:pohozaev-penalized}-\eqref{eq:pohozaev-L2} and the uniform bound \eqref{eq:uniformvortex}, as $\mu\to\infty$, we have 
\begin{align*}
\gamma
=2I_{\mu}-W\mu=2I_{\mu}+O(1).
\end{align*}
We set $B=B_{(0,\mu)}(\pi^{-1/2})$ and $\hat{B}=B_{0}(\pi^{-1/2})$. Using the test function $ \mathbf{1}_{B} \in K_\mu$, 
\begin{align*}
I_{\mu}\geq E_{2}[\mathbf{1}_{B}]
=\mathcal{E}[\mathbf{1}_{\hat{B}}]
+\frac{1}{4\pi}\iint_{\hat{B}\times \hat{B}}\log{|x-\bar{y}+(0,2\mu)|}dxdy-\frac{1}{2}=\frac{1}{4\pi}\log{(2 \mu)}+O(1).
\end{align*}
We thus obtain \eqref{eq:gamma lower bound}.
\end{proof}

\begin{proposition}\label{lem: rough uniform bounds}
For $a>8$, the growth bounds \eqref{eq:rough uniform bounds} hold for $\mu\geq M$. 
\end{proposition}

\begin{proof}
We may assume that 
\begin{align*}
\frac{1}{2\pi}\log{(2M)}-C>0,
\end{align*}
for the constant $C$ in \eqref{eq:gamma lower bound}. We take $0<\theta<1$ and apply \eqref{eq:inequality sf 2} for $x_2\geq 1$ to estimate 
\begin{equation*}
\mathcal G\omega(x)
\lesssim \frac{\mu}{x_2}(1+\log x_2)
+\left(\frac\mu{x_2}\right)^\theta.
\end{equation*}
The right-hand side is decreasing for $x_2>1$. We take $K>0$ satisfying $K\geq 1+1/M$. For $x_2\geq \mu K\geq MK>1$, we have  
\begin{equation*}
\mathcal G\omega(x)
\leq \frac{C'}{K^{\theta}}\left(1+\log{(2\mu)}+\log{K}\right).
\end{equation*}
For $x\in \operatorname{spt} \omega$, we use \eqref{eq:gamma lower bound} to estimate
\begin{align*}
\frac{1}{2\pi}\log(2\mu)-C\leq \gamma\leq \mathcal{G}\omega(x).
\end{align*}
It follows that 
\begin{align*}
\left(\frac{1}{2\pi}-\frac{C'}{K^{\theta}}\right)\log{(2M)}-C\leq \frac{C'}{K^{\theta}}(1+\log{K}).
\end{align*}
The left-hand side converges to a positive constant, while the right-hand side vanishes as $K\to\infty$. Thus $\operatorname{spt}\omega\cap \{x_2\geq \mu K\}=\emptyset$ for large $K$. We obtain the height bound in \eqref{eq:rough uniform bounds}. 

We apply \eqref{eq:inequality sf 4} and \eqref{eq:gamma lower bound} for $x\in \operatorname{spt}\omega$ to estimate 
\begin{align*}
\frac{1}{2\pi}\log{(2M)}-C
\leq \mathcal{G}\omega(x)
\lesssim (\mu+1)x_2 \min\left\{1,\frac{1}{|x_1|^{\frac{1}{2r}}}\right\},
\end{align*}
for $2<r<\infty$. Using  $D_2\lesssim \mu$, we obtain
\begin{align*}
\left(\frac{1}{2\pi}\log{(2M)}-C\right)|x_1|^{\frac{1}{2r}}
\lesssim \mu(\mu+1),
\end{align*}
for $|x_1|>1$. Thus the horizontal bound in \eqref{eq:rough uniform bounds} holds for $a=4r>8$.
\end{proof}

\begin{proposition}
For every $\delta>0$, there exists a constant $A\geq 1$ such that \eqref{eq:large-local-mass} holds. 
\end{proposition}

\begin{proof}
We set  
		\[
		\mathcal G\omega(x)=\frac1{2\pi}\int_{\mathbb{R}^{2}_{+}}\log\frac1{|x-y|}\omega(y)\,dy+\frac1{2\pi}\int_{\mathbb{R}^{2}_{+}}\log|x-\bar y|\omega(y)\,dy=:\mathcal{G}_1\omega+\mathcal{G}_2\omega.
		\]
Observe that for $x=(x_1,x_2)$ and $y=(y_1,y_2)$ in $\operatorname{spt}\omega$,  
\begin{align*}
|x-\bar{y}|=\sqrt{|x_1-y_1|^{2}+|x_2+y_2|^{2}}\leq 2\mu\sqrt{\left(\frac{D_1}{\mu}\right)^{2}+\left(\frac{D_2}{\mu}\right)^{2}}\leq 2\mu Q.
\end{align*}
This yields 
\begin{align*}
\mathcal{G}_2\omega(x)\le \frac{1}{2\pi}\log(2\mu Q).
\end{align*}
Using \eqref{eq:gamma lower bound}, we obtain the lower bound 
\begin{equation*}
\mathcal{G}_1\omega(x)=\mathcal{G}\omega(x)-\mathcal{G}_2\omega(x)
>\gamma-\mathcal{G}_2\omega(x)
\ge-\frac1{2\pi}\log Q-C.
\end{equation*}
We take $R>1$ and set
\begin{align*}
\mathcal{G}_1\omega(x)=\frac1{2\pi}\int_{|x-y|<R}\log\frac1{|x-y|}\omega(y)\,dy+\frac1{2\pi}\int_{|x-y|\geq R}\log\frac1{|x-y|}\omega(y)\,dy.
\end{align*}
By H\"older's inequality and the uniform bound \eqref{eq:uniformvortex}, we estimate 
\begin{align*}
\frac1{2\pi}\int_{|x-y|<R}\log\frac1{|x-y|}\omega(y)\,dy
\leq \frac1{2\pi}\int_{|x-y|<1}\log\frac1{|x-y|}\omega(y)\,dy\leq C'.
\end{align*}
Using 
\[
\int_{|x-y|\geq R}\omega(y)\log\frac1{|x-y|}\,dy
\le-\log R\int_{|x-y|\geq R}\omega(y)\,dy,
\]
we obtain the upper bound
\begin{align*}
\mathcal{G}_1\omega(x)\leq C'-\frac{1}{2\pi}\log R\int_{|x-y|\geq R}\omega(y)\,dy.
\end{align*}
Combining the lower and upper bounds imply  
\begin{align*}
\int_{|x-y|\geq R}\omega(y)\,dy\leq \frac{\log Q+2\pi(C+C')}{\log R}.
\end{align*}
For $R=AQ^{1+\delta}$, we observe that  
\begin{align*}
\frac{\log Q+2\pi(C+C')}{\log R}=\frac{\log Q+2\pi(C+C')}{\log A+(1+\delta)\log{Q}}\leq \frac{1}{1+\delta}+\frac{2\pi(C+C')}{\log A}.
\end{align*} 
We take large $A \geq 1$ so that 
\begin{align*}
\frac{2\pi(C+C')}{\log A}\leq \frac{1}{2}\left(1-\frac{1}{1+\delta}\right).
\end{align*}
This bound yields 
\begin{align*}
\int_{B_x(R)}\omega(y) dy=1-\int_{|x-y|\geq R}\omega(y) dy\geq \frac{1}{2}\left(1-\frac{1}{1+\delta}\right).
\end{align*}
We obtain \eqref{eq:large-local-mass}.
\end{proof}

\begin{proposition}\label{p:D_1bdd}
The constant $D_1$ is uniformly bounded for $\mu\geq M$. Moreover, there exist constants $R_0>0$ and $0<m_0<1/2$ such that \eqref{eq:large-local-mass2} holds.     
\end{proposition}

\begin{proof}
For $R>0$, there exist points $\{x_j\}_{j=1}^{N}\subset \operatorname{spt} \omega$ lying on a horizontal line such that $\{B_{x_j}(R)\}_{j=1}^{N}$ are disjoint and $D_1\leq NR$ since $\omega$ is Steiner symmetric. Applying \eqref{eq:large-local-mass} for $R=AQ^{1+\delta}$ yields
\begin{align*}
1=\int_{\mathbb{R}^{2}_{+}}\omega dx\geq \sum_{j=1}^{N}\int_{B_{x_j}(AQ^{1+\delta})}\omega dx 
\geq \frac{N}{2}\left(1-\frac{1}{1+\delta}\right).
\end{align*}
We thus obtain $D_1\leq C_{\delta}Q^{1+\delta}$ with some constant $C_{\delta}$. Since $D_2/\mu$ is uniformly bounded for $\mu\geq M$, $Q\lesssim 1+D_1/\mu$ and 
\begin{align*}
D_1\leq C_{\delta}\left(1+\left(\frac{D_1}{\mu}\right)^{1+\delta}\right).
\end{align*}
Using $D_1\leq C\mu^{a}$ in \eqref{eq:rough uniform bounds}, we have
\begin{align*}
D_1\leq C_{\delta}'\left(1+\frac{1}{\mu^{1+\delta-a\delta}}D_1\right).
\end{align*}
We fix $\delta=\delta_0\in (0,1/(a-1))$ and absorb the second term into the left-hand side by taking $\mu$ sufficiently large. Thus $D_1$ and $Q$ are uniformly bounded for all $\mu\geq M$. We obtain the lower bound \eqref{eq:large-local-mass2} for   \begin{equation*}
    \begin{split} 
R_0=\sup\left\{A_{\delta_0}Q^{1+\delta_0}:\mu\geq M\right\},\quad 
m_0=\frac{1}{2}\left(1-\frac{1}{1+\delta_0}\right). \qedhere  
    \end{split}
\end{equation*}  
\end{proof}

\subsection{Vertical vortex core bound}

\begin{proposition}\label{p:lowerboundE2}
There exists a constant $C_1>0$ such that \eqref{eq:lowerboundE2} holds.
\end{proposition}

\begin{proof}
We estimate 
\begin{align*}
I_{\mu}=E_2[\omega]
=\mathcal{E}_2[\hat{\omega}]+\frac{1}{4\pi}\iint_{\mathbb{R}^{2}\times \mathbb{R}^{2}}\log|x-\overline{y}+(0,2\mu)|\hat{\omega}(x)\hat{\omega}(y)dxdy. 
\end{align*}
We apply Jensen's inequality
\begin{align*}
\int_{X}\log{f}d\sigma\leq \log\left(\int_{X}fd\sigma\right)
\end{align*}
for $X=\mathbb{R}^{4}$,  $d\sigma=\hat{\omega}(x)\hat{\omega}(y)dxdy$, and $f=|x-\overline{y}+(0,2\mu)|$ to estimate 
\begin{align*}
\iint_{\mathbb{R}^{2}\times \mathbb{R}^{2}}\log|x-\overline{y}+(0,2\mu)|\hat{\omega}(x)\hat{\omega}(y)dxdy\leq 
\log\left(\iint_{\mathbb{R}^{2}\times \mathbb{R}^{2}}|x-\overline{y}+(0,2\mu)|\hat{\omega}(x)\hat{\omega}(y)dxdy\right).
\end{align*}
Since $|x-\overline{y}+(0,2\mu)|\leq 2D_1+x_2+y_2+2\mu
$ and $\hat{\omega}$ has zero-$x_2$-moment and unit mass, the right-hand side is bounded by $\log2(D_1+\mu)$. We thus obtain the upper bound
\begin{align*}
I_{\mu}\leq \mathcal{E}_2[\hat{\omega}]+\frac{1}{4\pi}\log{(2D_1)}+\frac{1}{4\pi}\log{(2 \mu)}. 
\end{align*}
We estimate $I_{\mu}$ from below. For the radially symmetric $\omega^{R}_{1,\gamma_{R}}$ in \eqref{eq:radiallysymmetric}, we set  
\begin{align*}
\check{\omega}^{R}_{1,\gamma_{R}}(x)=\omega^{R}_{1,\gamma_{R}}(x-(0,\mu))
\end{align*}
so that 
\begin{align*}
\int_{\mathbb{R}^{2}_{+}}\check{\omega}^{R}_{1,\gamma_{R}}dx=1,\quad 
\int_{\mathbb{R}^{2}_{+}}x_2\check{\omega}^{R}_{1,\gamma_{R}}dx=\mu.
\end{align*}    
Namely, $\check{\omega}^{R}_{1,\gamma_{R}}\in K_{\mu}$. By the maximality of $I_{\mu}$, we have 
\begin{align*}
I_{\mu}
\geq E_2[\check{\omega}^{R}_{1,\gamma_{R}}]
&=\mathcal{I}+\frac{1}{4\pi}\iint_{\mathbb{R}^{2}\times \mathbb{R}^{2}}\log|x-\overline{y}+(0,2\mu)|\omega^{R}_{1,\gamma_{R}}(x)\omega^{R}_{1,\gamma_{R}}(y)dxdy \\
&=\mathcal{I}+\frac{1}{4\pi}\log(2\mu)+
\frac{1}{4\pi}\iint_{\mathbb{R}^{2}\times \mathbb{R}^{2}}\log\left|\frac{x-\overline{y}}{2\mu}+(0,1)\right|\omega^{R}_{1,\gamma_{R}}(x)\omega^{R}_{1,\gamma_{R}}(y)dxdy.
\end{align*}
Using $\log(t+1)\leq t$ and $\int x\omega^{R}_{1,\gamma_R}dx=0$, the last term is $O(\mu^{-2})$. Combining this with the upper bound, we obtain 
\begin{align*}
\mathcal{I}-\frac{1}{4\pi}\log2D_1+O\left(\mu^{-2}\right)\leq \mathcal{E}_2[\hat{\omega}].
\end{align*}
Since $D_1$ is bounded by Proposition \ref{p:D_1bdd}, we obtain \eqref{eq:lowerboundE2}.
\end{proof}

\begin{proposition}\label{p:upperboundE2}
There exists a constant $C_2>0$ such that \eqref{eq:upperboundE2} holds for $\mu\geq M$, provided that $H\geq 4(R_0+1)$. 
\end{proposition}

\begin{proof}
We take two points $x^{+}$, $x^{-}\in \operatorname{spt} \omega\cap \{x_1=0\}$ such that $|x^{+}-x^{-}|\geq H-1$. By Proposition \ref{p:D_1bdd}, we have 
\begin{align*}
\int_{B^{+}}\omega dx\geq m_0,\quad  \int_{B^{-}}\omega dx\geq m_0,
\end{align*}
for $B^{+}=B_{x^{+}}(R_0)$ and $B^{-}=B_{x^{-}}(R_0)$. For $x\in B^{+}$ and $y\in B^{-}$, we have $H-1\leq |x^{+}-x^{-}|\leq 2R_0+|x-y|$. Using the condition $H\geq 4(R_0+1)$,
we see that 
\begin{align*}
|x-y|\geq \frac{H}{2}.
\end{align*}
We set $\mathbb{R}^{4}=(B^{+}\times B^{-})\cup (B^{-}\times B^{+})\cup G$ and estimate 
\begin{align*}
\mathcal{E}_2[\hat{\omega}]
&\leq \frac{1}{4\pi}\iint_{\mathbb{R}^{2}\times \mathbb{R}^{2}}\log\frac{1}{|x-y|}\hat{\omega}(x)\hat{\omega}(y)dxdy\\
&=\frac{1}{4\pi}\iint_{G}\log\frac{1}{|x-y|}\hat{\omega}(x)\hat{\omega}(y)dxdy
+\frac{1}{2\pi}\iint_{B^{+}\times B^{-}}\log\frac{1}{|x-y|}\hat{\omega}(x)\hat{\omega}(y)dxdy.
\end{align*}
Observe that 
\begin{align*}
\iint_{G}\log\frac{1}{|x-y|}\hat{\omega}(x)\hat{\omega}(y)dxdy
&\leq \iint_{G}\left(\log\frac{1}{|x-y|}\right)_{+}\hat{\omega}(x)\hat{\omega}(y)dxdy \\
&=\iint_{|x-y|\leq 1}\left(\log\frac{1}{|x-y|}\right)_{+}\hat{\omega}(x)\hat{\omega}(y)dxdy.
\end{align*}
The right-hand side is uniformly bounded for $\mu\geq M$ by the uniform bound for $\hat{\omega}$ in $L^{\infty}\cap L^{1}$. Using $|x-y|\geq H/2$, we estimate
\begin{align*}
\iint_{B^{+}\times B^{-}}\log\frac{1}{|x-y|}\hat{\omega}(x)\hat{\omega}(y)dxdy\leq (\log2-\log{H})\left(\int_{B^{+}}\omega dx\right)\left(\int_{B^{-}}\omega dx\right)
\leq m_0^{2}(\log2-\log{H}).
\end{align*}
We thus obtain \eqref{eq:upperboundE2}.
\end{proof}

\begin{lemma}[Uniform vortex core bound]\label{lem:uniform bounds}
There exists a constant $R>0$ such that \eqref{eq:uniformvortexcorebound1} holds for $\mu\geq M$.
\end{lemma}

\begin{proof}
We show that $H$ is uniformly bounded for $\mu\geq M$. We may assume that $H\geq 4(R_0+1)$ and apply Propositions \ref{p:lowerboundE2} and \ref{p:upperboundE2} to estimate 
\begin{align*}
-C_1\leq C_2-\frac{m_0^{2}}{2\pi}\log{H}.
\end{align*}
Thus $H$ is uniformly bounded and \eqref{eq:uniformvortexcorebound1} holds for some constant $R>0$.
\end{proof}

\subsection{The logarithmic free-energy maximization}

We apply the existence and uniqueness results for the variational problem
\eqref{eq:MFE} proved in \cite[Theorems 2.1 and 4.8]{Car}. We also use compactness for maximizing sequences supported in a fixed compact set.

\begin{theorem}[Logarithmic free-energy maximization]\label{thm: lim VP}
The variational problem \eqref{eq:MFE} admits a unique maximizer,
given by the radially symmetric solution for $\lambda=1$. Moreover, for any nonnegative sequence $\{\omega_n\}\subset \mathcal{K}$ satisfying $\mathcal{E}_2[\omega_n]\to \mathcal{I}$ and $\operatorname{spt}\omega_n\subset B_0(R)$ for some $R>0$, there exists a subsequence such that $\omega_n\to \omega^{R}_{1,\gamma_{R}}$ in $L^{2}(B_0(R))$. 
\end{theorem}

\begin{proof}
For the sequence $\{\omega_n\}$ satisfying $\mathcal{E}_2[\omega_n]\to \mathcal{I}$ and $\operatorname{spt} \omega_n\subset B=B_0(R)$, we set 
\begin{align*}
\psi_n(x)&=\frac{1}{2\pi}\int_{B}\log\frac{1}{|x-y|}\omega_n(y)dy,\\
\mathcal{E}_2[\omega_n]
&=\frac{1}{2}(\psi_n,\omega_n)_{L^{2}(B)}-\frac{1}{2}||\omega_n||_{L^{2}(B)}^{2}.
\end{align*}
Using H\"older's inequality and $||\omega_n||_{L^{1}}=1$, we estimate $||\psi_n||_{L^{\infty}}\lesssim ||\omega_n||_{L^{2}}$ and 
\begin{align*}
\frac{1}{2}||\omega_n||_{L^{2}(B)}^{2}=\frac{1}{2}(\psi_n,\omega_n)_{L^{2}(B)}-\mathcal{E}_2[\omega_n]
\leq C||\omega_n||_{L^{2}(B)}-\mathcal{E}_2[\omega_n].
\end{align*}
Applying Young's inequality implies that $\{\omega_n\}$ is bounded in $L^{2}(B)$. We choose a subsequence such that $\omega_n\rightharpoonup \omega$ in $L^{2}(B)$. Since $1,x\in L^2(B)$, the weak convergence preserves the mass and
zero-moment conditions. Hence $\omega\in\mathcal K$. The stream function is bounded in $L^{\infty}(B)$ and converges to 
\begin{align*}
\psi(x)=\frac{1}{2\pi}\int_{B}\log\frac{1}{|x-y|}\omega(y)dy
\end{align*}
in $L^{2}(B)$. This yields the convergence 
\begin{align*}
\mathcal{I}=\lim_{n\to\infty}\mathcal{E}_2[\omega_n]
=\frac{1}{2}\lim_{n\to\infty}(\psi_n,\omega_n)_{L^{2}(B)}-\frac{1}{2}\liminf_{n\to\infty}||\omega_n||_{L^{2}(B)}^{2} 
\leq \frac{1}{2}(\psi,\omega)_{L^{2}(B)}-\frac{1}{2}||\omega||_{L^{2}(B)}^{2}\leq \mathcal{I}.
\end{align*}
Thus $\omega$ is a maximizer of \eqref{eq:MFE} and $\omega_n\to \omega$ in $L^{2}(B)$. By uniqueness, $\omega=\omega^{R}_{1,\gamma_R}$.
\end{proof}

\begin{remark}
The Euler--Lagrange equation of \eqref{eq:MFE} is 
\begin{align}
\omega=(\psi-\gamma)_{+},  \label{eq:ELR2}
\end{align}
for $\psi=N*\omega$ and some constant $\gamma$ \cite[Theorems 2.1, 3.1]{Car}. Multiplying this equation by \(\omega\) and integrating over \(\mathbb R^2\) yields the identity 
\begin{align}
\mathcal{I}=\frac{1}{2}\gamma.  \label{eq:P0}
\end{align}
The radially decreasing solution to \eqref{eq:ELR2} can be expressed as $\psi-\gamma=C_1J_0(r)$ for $0<r<R$ and $\psi=C_2\log r$ for $r>R$. Using the $C^{1}$ matching condition for $\psi$ at $r=R$ and the mass condition $\int_{\mathbb{R}^{2}} \omega dx=1$, we find that $R=j_{0,1}$ and $\psi$ is of the form
\begin{align}
\psi(r)=\begin{cases}
\displaystyle\frac{1}{2\pi j_{0,1}J_1(j_{0,1})}J_0(r)+\gamma_{j_{0,1}} & 0<r\leq j_{0,1},\\
-\displaystyle\frac{1}{2\pi}\log{r} & r>j_{0,1},
\end{cases}
\label{eq:psiexplicit}
\end{align}
and satisfies \eqref{eq:ELR2} for 
\begin{align}
\gamma=-\frac{1}{2\pi}\log{j_{0,1}}. \label{eq:gammaexplicit}
\end{align}
The constants $\gamma$ and $\mathcal{I}$ are negative, unlike the half-plane problem \eqref{eq: VP-2}.
\end{remark}

\begin{proposition}\label{p:convergencemax}
\begin{align}
I_{\mu}-\frac{1}{4\pi}\log{(2 \mu)}= \mathcal{I}+O\left(\mu^{-2}\right)\quad \textrm{as}\ \mu\to\infty. \label{eq:P5}
\end{align}
\end{proposition}

\begin{proof}
For the Steiner symmetric $\omega\in S_{\mu}$ and $\mu\geq M$, $\operatorname{spt}\hat{\omega}\subset B=B_0(R)$ by Lemma \ref{lem:uniform bounds}. Since $\omega$ is even for $x_1$, the recentered function \(\hat\omega\) satisfies
\begin{align*}
\int_{\mathbb{R}^{2}}\hat{\omega}dx=1,\quad \int_{\mathbb{R}^{2}}x\hat{\omega}dx=0.
\end{align*}
This means that $\hat{\omega}\in \mathcal{K}$. It follows from \eqref{eq:FE} that 
\begin{align*}
I_{\mu}=E_2[\omega]
&=\mathcal{E}_{2}[\hat{\omega}]+\frac{1}{4\pi}\iint_{B\times B}\log|x-\overline{y}+(0,2\mu)|\hat{\omega}(x)\hat{\omega}(y)dxdy \\
&\leq \mathcal{I}
+\frac{1}{4\pi}\log (2\mu)
+\frac{1}{8\pi}\iint_{B\times B}\log\left( \left|\frac{x_1-y_1}{2\mu}\right|^{2}+\left|\frac{x_2+y_2}{2\mu}+1\right|^{2} \right)\hat{\omega}(x)\hat{\omega}(y)dxdy.
\end{align*}
The last term is $O(\mu^{-2})$. Thus $I_{\mu}-(4\pi)^{-1}\log(2\mu)-\mathcal{I}\leq O(\mu^{-2})$. Conversely, for the radially symmetric $\omega^{R}_{1,\gamma_R}$ in \eqref{eq:radiallysymmetric}, the $x_2$-translation 
\begin{align*}
\check{\omega}^{R}_{1,\gamma_R}(x)=\omega^{R}_{1,\gamma_R}(x-(0,\mu))
\end{align*}
satisfies 
\begin{align*}
\int_{\mathbb{R}^{2}_{+}}\check{\omega}^{R}_{1,\gamma_R}dx=1,\quad 
\int_{\mathbb{R}^{2}_{+}}x_2\check{\omega}^{R}_{1,\gamma_R}dx=\mu.
\end{align*}    
As in the proof of Proposition \ref{p:lowerboundE2}, we obtain   
\begin{align*}
I_{\mu}
\geq E_2[\check{\omega}^{R}_{1,\gamma_R}]
=\mathcal{I}+\frac{1}{4\pi}\int_{\mathbb{R}^{2}}\int_{\mathbb{R}^{2}}\log|x-\overline{y}+(0,2\mu)|\omega^{R}_{1,\gamma_R}(x)\omega^{R}_{1,\gamma_R}(y)dxdy 
= \mathcal{I}+\frac{1}{4\pi}\log(2\mu)+O\left(\mu^{-2}\right).
\end{align*}
Thus $|I_{\mu}-(4\pi)^{-1}\log(2\mu)-\mathcal{I}|\lesssim \mu^{-2}$. We proved \eqref{eq:P5}.
\end{proof}

\subsection{Convergence toward the radially symmetric solution}
	
\begin{lemma}\label{l:converge}
Let $\{\mu_n\}$ be a sequence such that $\mu_n\to \infty$. Let $\omega_n\in S_{\mu_n}$ be a Steiner symmetric maximizer in Theorem \ref{thm:nearly Lamb} with the constants $W_n>0$ and $\gamma_n>0$. Then, there exists $R>0$ such that $\operatorname{spt} \hat{\omega}_n \subset B_0(R)$ and, passing to a subsequence, 
\begin{align}
\hat{\omega}_n&\to\omega^{R}_{1,\gamma_R}
		\quad\textrm{in}\ L^{2}(B_0(R)), \label{eq:convergencerecentred}  \\
W_n&=\frac1{4\pi\mu_n}+o\left(\mu^{-2}_{n}\right),   \label{eq:large-multiplier-asymptoticsW} \\
\gamma_n-\frac{1}{2\pi}\log(2\mu_n)+\frac{1}{4\pi}&=\gamma_R+o\left(\mu_n^{-1}\right), \label{eq:large-multiplier-asymptoticsg}
\end{align}
where $\gamma_R=\gamma_R(1)$ is the constant \eqref{eq:radialflux}.
\end{lemma}

\begin{proof}
The sequence $\{\omega_n\}$ is uniformly bounded by \eqref{eq:uniformvortex}. By translation,  
\[
		\int_{\mathbb R^2}\hat{\omega}_ndx=1,
		\qquad
		\int_{\mathbb R^2}x\hat{\omega}_n\,dx=0.
\]
By Lemma \ref{lem:uniform bounds}, $\operatorname{spt}\hat{\omega}_n\subset B=B_0(R)$ for some $R>0$. It follows that 
\begin{align*}
I_{\mu_n}-\frac{1}{4\pi}\log(2\mu_n)
&=\mathcal{E}_2[\hat{\omega}_n]+\frac{1}{4\pi}\iint_{B\times B}\log{|x-\overline{y}+(0,2\mu)|}\hat{\omega}_n(x)\hat{\omega}_n(y)dxdy-\frac{1}{4\pi}\log(2\mu_n)\\
&=\mathcal{E}_2[\hat{\omega}_n]+O\left(\mu_n^{-2}\right).
\end{align*}
Thus $\mathcal{E}_{2}[\hat{\omega}_n]\to \mathcal{I}$ by \eqref{eq:P5}. We apply Theorem \ref{thm: lim VP} and deduce the convergence \eqref{eq:convergencerecentred}. It follows from \eqref{eq:pohozaev-W} that 
\begin{align*}
W_n
&=\frac{1}{2\pi} \iint_{B\times B} \frac{x_2+y_2+2\mu_n}{|x-\bar y+(0,2\mu_n)|^2} \hat{\omega}_n(x)\hat{\omega}_n(y)\,dx\,dy\\
&=\frac1{4\pi\mu_n}+\frac{1}{2\pi} \iint_{B\times B} \left(\frac{x_2+y_2+2\mu_n}{|x-\bar y+(0,2\mu_n)|^2}-\frac{1}{2\mu_n}\right) \hat{\omega}_n(x)\hat{\omega}_n(y)\,dx\,dy
=\frac1{4\pi\mu_n}+o\left(\mu_n^{-2}\right).
\end{align*}
We thus obtain \eqref{eq:large-multiplier-asymptoticsW}. Using \eqref{eq:P1}, \eqref{eq:P5}, \eqref{eq:large-multiplier-asymptoticsW}, and \eqref{eq:P0}, we obtain
\begin{align*}
\gamma_n=2I_{\mu_n}-W_n\mu_n=\frac{1}{2\pi}\log{(2\mu_n)}+2\mathcal{I}-\frac{1}{4\pi}+o\left(\mu_n^{-1}\right)=\frac{1}{2\pi}\log{(2\mu_n)}+\gamma_R-\frac{1}{4\pi}+o\left(\mu_n^{-1}\right).
\end{align*}
Thus, \eqref{eq:large-multiplier-asymptoticsg} holds. 
\end{proof}

We use the symbol $f(x;\mu)=O_{L^{\infty}_{\textrm{loc}}}(\mu^{-1})$ if $||f(\cdot\ ;\mu)||_{L^{\infty}(B_0(R))}=O(\mu^{-1})$ as $\mu\to\infty$ for each $R>0$.

\begin{proposition}\label{p:strongerconvergence}
In Lemma \ref{l:converge}, set 
\begin{align}
\hat{\psi}_{n}(x)=\mathcal{G}\omega_n(x+(0,\mu_n)).
\end{align}
Then, $\hat{\psi}_{n}$ is odd with respect to the line $x_2=-\mu_n$ and satisfies 
\begin{align}
\hat{\psi}_{n}(x)&=N*\hat{\omega}_n+\frac{1}{2\pi}\log(2\mu_n)+O_{L^{\infty}_{\textrm{loc}}}(\mu_n^{-1}), \label{eq:psirecenter}\\
\hat{\omega}_n&=\left(N*\hat{\omega}_n-\gamma_R+O_{L^{\infty}_{\textrm{loc}}}(\mu_n^{-1})\right)_+,\quad x_2>-\mu_n. \label{eq:recentereq}
\end{align}
Moreover, 
\begin{align}
N*\hat{\omega}_n &\to N*\omega^{R}_{1,\gamma_R},
\qquad
\hat{\omega}_n \to \omega^{R}_{1,\gamma_R}\quad \textrm{uniformly in}\ \mathbb{R}^{2}, \label{eq:recenteruniformconvergence} \\
\mathbf{1}_{(0,\infty)}(N*\hat{\omega}_n-\gamma_R)(x)
&\to \mathbf{1}_{B_0(j_{0,1})}(x)\quad \textrm{a.e.}\ x\in \mathbb{R}^{2}.
\end{align}
\end{proposition}

\begin{proof}
We set 
\begin{align*}
\psi_n(x)=\mathcal{G}\omega_n(x)=\frac{1}{2\pi}\int_{\mathbb{R}^{2}_{+}}\log\frac{1}{|x-y|}\omega_n(y)dy+\frac{1}{2\pi}\int_{\mathbb{R}^{2}_{+}}\log|x-\bar{y}|\omega_n(y)dy.
\end{align*}
Changing the variable, we have 
\begin{align*}
\hat{\psi}_n(x)
&=N*\hat{\omega}_n(x)+\frac{1}{2\pi}\int_{\mathbb{R}^{2}}\log|x-\bar{y}+(0,2\mu_n)|\hat{\omega}_n(y)dy \\
&=N*\hat{\omega}_n(x)+\frac{1}{2\pi}\log(2\mu_n)+\frac{1}{2\pi}\int_{\mathbb{R}^{2}}\log\left|\frac{x-\bar{y}}{2\mu_n}+(0,1)\right|\hat{\omega}_n(y)dy.
\end{align*}
Using the pointwise estimate $\log(t+1)\lesssim t$ for $t>0$, we see that the last term is $O_{L^{\infty}_{\textrm{loc}}}(\mu_n^{-1})$. Thus \eqref{eq:psirecenter} holds.

Since $\omega_n=(\psi_n-W_nx_2-\gamma_n)_{+}$ for $x_2>0$, $\hat{\omega}_n$ satisfies 
\begin{align*}
\hat{\omega}_n(x)=(\hat{\psi}_n(x)-W_nx_2-W_n\mu_n-\gamma_n)_{+},\quad x_2>-\mu_n.
\end{align*}
Substituting \eqref{eq:psirecenter} into this, 
\begin{align*}
\hat{\omega}_n(x)=\left(N*\hat{\omega}_{n}+\frac{1}{2\pi}\log(2\mu_n)-W_n\mu_n-\gamma_n+O_{L^{\infty}_{\textrm{loc}} }\left(\mu_n^{-1}\right)\right)_{+}.
\end{align*}
Using \eqref{eq:large-multiplier-asymptoticsW} and \eqref{eq:large-multiplier-asymptoticsg}, we obtain \eqref{eq:recentereq}. Since $\hat{\omega}_n$ is compactly supported and converges to $\omega^{R}_{1,\gamma_{R}}$ in $L^{2}$ by \eqref{eq:convergencerecentred}, $N*\hat{\omega}_n$ converges uniformly to $N*\omega^{R}_{1,\gamma_{R}}$ in $\mathbb{R}^{2}$. By \eqref{eq:recentereq}, $\hat{\omega}_n$ converges to $\omega^{R}_{1,\gamma_R}$ locally uniformly in $\mathbb{R}^{2}$. Since $\hat{\omega}_n$ is uniformly supported in $B_0(R)$, its convergence is uniform in $\mathbb{R}^{2}$.
\end{proof}

\subsection{Uniqueness of maximizers}

\begin{proposition}\label{p:identitydifference2}
For two maximizers $\omega^{i}\in S_{\mu}$ in Theorem \ref{thm:nearly Lamb} for $i=1,2$, we have
\begin{equation}
\begin{aligned}
\hat{\omega}^{1}-\hat{\omega}^{2}
=& \left(\int_{0}^{1}\mathbf{1}_{(0,\infty)}\left(N*(t\hat{\omega}^{1}+(1-t)\hat{\omega}^{2})-\gamma_R+O_{L^{\infty}_{\textrm{loc}}}\left(\mu^{-1}\right)\right)dt\right)\\
&\cdot\left(N*(\hat{\omega}^1-\hat{\omega}^{2})
+||\hat{\omega}^{1}-\hat{\omega}^{2}||_{L^{2}(\mathbb{R}^{2})}O_{L^{\infty}_{\textrm{loc}}}\left(\mu^{-1}\right)
\right),
\end{aligned}
\label{eq:identitydifference2}
\end{equation}
as $\mu\to\infty$.
\end{proposition}

\begin{proof}
Using the Lagrange multipliers $W^{i},\gamma^{i}>0$ for $\omega^{i}$ and $\psi^{i}=\mathcal{G}\omega^{i}$, we have 
\begin{align*}
\hat{\psi}^{i}&=N*\hat{\omega}^{i}+\frac{1}{2\pi}\log(2\mu)+\frac{1}{2\pi}\int_{\mathbb{R}^{2}}\log\left|\frac{x-\bar{y}}{2\mu}+(0,1)\right|\hat{\omega}^{i}(y)dy,\\
\widehat{\left(\psi^{i}-W^{i}x_2-\gamma^{i}\right)}&=N*\hat{\omega}^{i}-\gamma_R+O_{L^{\infty}_{\textrm{loc}}}(\mu^{-1}).
\end{align*}
By Lemma \ref{lem:uniform bounds}, we have  $\operatorname{spt}\hat{\omega}^{i}\subset B=B_0(R)$ for some $R>0$. We set 
\begin{align*}
\zeta=\frac{\hat{\omega}^{1}-\hat{\omega}^{2}}{||\hat{\omega}^{1}-\hat{\omega}^{2}||_{L^{2}(B)}}
\end{align*}
so that $||\zeta||_{L^{2}}=1$. It follows that 
\begin{align*}
\hat{\psi}^{1}-\hat{\psi}^{2}=N*(\hat{\omega}^{1}-\hat{\omega}^{2})+\frac{||\hat{\omega}^{1}-\hat{\omega}^{2}||_{L^{2}(B)}}{2\pi}\int_{B}\log{\left|\frac{x-\bar{y}}{2\mu}+(0,1)\right|}\zeta(y)dy.
\end{align*}
Using $\log(t+1)\lesssim t$ for $t>0$ and $\operatorname{spt}\zeta\subset B$, the last integral is $O_{L^{\infty}_{\textrm{loc}}}(\mu^{-1})$. Substituting this into the translation of \eqref{eq:identitydifference}, we obtain \eqref{eq:identitydifference2}.
\end{proof}

\begin{lemma}\label{uniq-maximizer2}
		There exists $\kappa_{**}>0$ such that for $\mu \geq \kappa_{**}$, there exists a unique Steiner symmetric $\omega$ such that 
        \begin{align}
        S_{\mu}=\{\omega(\cdot +(a,0)): a\in \mathbb{R}\}.  \label{eq:uniqueness2}
        \end{align}
		\end{lemma}

\begin{proof}
Suppose, for contradiction, that for any $n\geq 1$ there exists $\mu_n\geq n$ such that there exist two Steiner symmetric maximizers $\omega_n^{i}\in S_{\mu_n}$ for $i=1,2$. The recentered $\hat{\omega}^{i}_{n}$ satisfy $\operatorname{spt} \hat{\omega}^{i}_{n}\subset B=B_0(R)$ for some $R>0$ and 
\begin{align*}
\int_{B}\hat{\omega}^{i}_ndx=1,\quad \int_{B}x\hat{\omega}^{i}_ndx=0. 
\end{align*}
By Proposition \ref{p:strongerconvergence}, $N*\hat{\omega}_n^{i}$ and $\hat{\omega}_n^{i}$ converge to $N*\omega^{R}_{1,\gamma_{R}}$ and $\omega^{R}_{1,\gamma_{R}}$ uniformly in $\mathbb{R}^{2}$.

We set 
\begin{align*}
\zeta_n=\frac{\hat{\omega}_n^{1}-\hat{\omega}_n^{2}}{||\hat{\omega}_n^{1}-\hat{\omega}_n^{2}||_{L^{2}(B)}}.
\end{align*}
Then, $\zeta_n$ satisfies 
\begin{align*}
||\zeta_n||_{L^{2}(B)}=1,\quad \int_{B}\zeta_n dx=0,\quad \int_{B}x\zeta_n dx=0. 
\end{align*}
Passing to a subsequence, we have $\zeta_n\rightharpoonup\zeta$ in $L^2(B)$. Applying \eqref{eq:identitydifference2}, $\zeta_n$ satisfies 
\begin{equation*}
\begin{aligned}
\zeta_n
= \left(\int_{0}^{1}\mathbf{1}_{(0,\infty)}\left(N*(t\hat{\omega}^{1}_n+(1-t)\hat{\omega}^{2}_n)-\gamma_R+O_{L^{\infty}_{\textrm{loc}}}\left(\mu^{-1}_n\right)\right)dt\right)
\cdot\left(N*\zeta_n
+O_{L^{\infty}_{\textrm{loc}}}\left(\mu^{-1}_n\right)
\right).
\end{aligned}
\end{equation*}
From this expression, $\zeta_n$ is uniformly bounded in $B$ and converges to the limit 
\begin{align*}
\zeta= \mathbf{1}_{B_0(j_{0,1})}N*\zeta,
\end{align*}
almost everywhere in $B$. In particular, $\zeta_n\to \zeta$ in $L^{2}(B)$ and $||\zeta||_{L^{2}(B)}=1$ and $\int_{B}x\zeta dx=0$. Since $\zeta$ belongs to the kernel of the operator $L=I-\mathbf{1}_{B_0(j_{0,1})}(-\Delta_{\mathbb{R}^{2}})^{-1}\mathbf{1}_{B_0(j_{0,1})}$, we apply the coercive estimate \eqref{eq:coercivej01} and conclude that $\zeta=0$. We obtained a contradiction, and the proof is now complete.
\end{proof}

\begin{proof}[Proof of Theorems \ref{thm:unique-near-lamb} and  \ref{thm:well separated dipole} (ii)]
The uniqueness follows from Lemma \ref{uniq-maximizer2}. The limiting profile \eqref{eq:uniformvortexcorebound}-\eqref{eq:convergenceWg} follows from Lemma \ref{l:converge} and \eqref{eq:recenteruniformconvergence}, \eqref{eq:large-multiplier-asymptoticsW}, and \eqref{eq:large-multiplier-asymptoticsg}.

It remains to prove the continuity of the unique Steiner symmetric
maximizer $\omega_\mu$. In the normalized case \(\lambda=\nu=1\), we have \(\kappa=\mu\). For $0<\mu\leq \kappa_{*}$, the continuity follows directly from the explicit formula \eqref{eq:Lamb}. It remains to consider the two mass-saturated regimes. Let $\mu_n\to\bar\mu$ in
\[
[\kappa_*,\kappa_*+\delta_{*}]\cup[\kappa_{**},\infty).
\]
Since $I_{\mu_n}\to I_{\bar\mu}$ by Lemma \ref{l:Lip}, the sequence
$\{\omega_{\mu_n}\}$ is a maximizing sequence for $I_{\bar\mu}$.
By Proposition~\ref{p:compactness-maximizing}, every subsequence admits a
further subsequence converging in $L^2\cap L^1_*$ to a maximizer of
$I_{\bar\mu}$. By uniqueness, the limit is necessarily
$\omega_{\bar\mu}$. Hence the whole sequence converges to
$\omega_{\bar\mu}$ in $L^2\cap L^1_*$. The uniform convergence follows
as in Proposition \ref{p:strongerconvergence0}. Therefore,
\[
\omega_{\mu_n}\to\omega_{\bar\mu}
\quad\text{in}\quad
L^2\cap L^1_*\cap L^{\infty}(\mathbb{R}^2_+).
\]
\end{proof}

\section{Stability}\label{s:6} 

It remains to show the stability (Theorem \ref{thm:stability-near-lamb}). We apply a stability result for $S_\mu$ under regular perturbations \cite[Theorem 1.4]{AC2019}.  

\begin{proposition}\label{p:stability}
Let $\mu>0$. The set $S_{\mu}$ is stable in the sense that for every $\varepsilon>0$, there exists $\delta>0$ such that, whenever $\zeta_0\in L^{\infty}\cap L^1  \cap L^1_*(\mathbb R^2_+)$ is nonnegative and satisfies $\| \zt_0\|_{L^1}\le 1$ and 
\begin{equation*}
\inf_{\omega\in S_{\mu}} \nrm{\zeta_0-\omega}_{L^2 (\bbR^2_+)} 
+\left|\int_{\mathbb R^2_+}x_2\zeta_0\,dx-\mu\right|
\leq\delta,
\end{equation*}
the unique global-in-time solution $\zeta(t)$ of \eqref{eq:2D-Euler} with initial data $\zeta_0$ satisfies
\begin{equation*}
\inf_{\omega\in S_{\mu}} \nrm{\zeta(t)-\omega}_{L^2 \cap L^1_*(\bbR^2_+)} 
\leq\varepsilon,
\qquad t \in \bbR. 
\end{equation*}
\end{proposition}

\begin{proof}[Proof of Theorem \ref{thm:stability-near-lamb}]
For $\kappa_*<\mu\leq \kappa_*+\delta_{*}$ or $\kappa_{**}\leq \mu$, there exists a maximizer $\omega$ such that  
\begin{align*}
S_{\mu}=\{\omega(\cdot+(a,0)):a\in \mathbb{R} \},
\end{align*}
by Theorem \ref{thm:unique-near-lamb}. Since 
\begin{align*}
\inf_{\omega\in S_{\mu}}||\zeta-\omega||_{L^{2}\cap L^{1}_{*}(\mathbb{R}^{2}_{+})}
=\inf_{a\in  \mathbb{R}}||\zeta-\omega(\cdot +(a,0))||_{L^{2}\cap L^{1}_{*}(\mathbb{R}^{2}_{+})},
\end{align*}
we deduce the stability result from Proposition \ref{p:stability}.
\end{proof}

\bigskip

\noindent\textbf{\large Statements and Declarations:}

	\medskip
    
\noindent\textbf{Acknowledgments}  KA is supported by the JSPS through the Grant-in-Aid for Scientific Research (C) 24K06800 and MEXT Promotion of Distinctive Joint Research Center Program JP MXP0619217849. IJ has been supported by the NRF grant from the Korean government (MSIT), No. 2022R1C1C1011051, RS-2024-00406821, the Asian Young Scientist Fellowship, and the KIAS Individual Grant at Korea Institute for Advanced Study. GQ is supported by National Key R\&D Program of China (Grant 2025YFA1018400) and NNSF of China (Grant 12471190). 
	
	\medskip
	
	\noindent\textbf{Conflict of interest statement} There is no conflict of interest.
	
	\medskip
	
	\noindent\textbf{Data availability statement} Data sharing not applicable to this article as no datasets were generated or analysed during the current study.
		
\begingroup
\small

	\bibliographystyle{alpha}
	\bibliography{dipoles-merged}

@unpublished{DMW26,
  author = {M. del Pino and M. Musso and J. Wei},
  title = {Nondegeneracy of the {Sadovskii} vortex},
  note = {\href{https://arxiv.org/abs/2609.13362}{arXiv:2609.13362}}
}

@article{ZPS24,
  author  = {X. Zhao and B. Protas and R. Shvydkoy},
  title   = {On the inviscid instability of the 2-D {Taylor--Green} vortex},
  journal = {J. Fluid Mech.},
  volume  = {999},
  pages   = {A64},
  year    = {2024},
  doi     = {10.1017/jfm.2024.946}
}

@unpublished{CQYZZ22,
  title  = {Existence, uniqueness and stability of steady vortex rings of small cross-section},
  author = {Cao, D. and Qin, G. and Yu, W. and Zhan, W. and Zou, C.},
  note   ={\href{https://arxiv.org/abs/2201.08232} {arXiv:2201.08232}},
}

@unpublished{CCDV26,
	archiveprefix = {arXiv},
	author = {G. Cao-Labora and M. Colombo and M. Dolce and P. Ventura},
	note = {\href{https://arxiv.org/abs/2601.23040}{arXiv:2601.23040}},
	title = {Instability of two-dimensional {Taylor--Green} vortices}}

@unpublished{EH26,
	author = {T. M. Elgindi and Y. Huang},
	note = {\href{https://arxiv.org/abs/2604.12962}{arXiv:2604.12962}},
	title = {On the flexibility of 2D {Euler} steady states}}

@article{DDMP24,
	author = {J. D{\'a}vila and M. del Pino and M. Musso and S. Parmeshwar},
	doi = {10.1016/j.jde.2024.06.023},
	journal = {J. Differential Equations},
	pages = {33--63},
	title = {Asymptotic properties of vortex-pair solutions for incompressible {Euler} equations in $\mathbb{R}^2$},
	volume = {408},
	year = {2024}}

@article{DDMP26,
	archiveprefix = {arXiv},
	author = {J. D{\'a}vila and M. del Pino and M. Musso and S. Parmeshwar},
	doi = {10.4171/JEMS/1776},
	eprint = {2310.07238},
	journal = {J. Eur. Math. Soc.},
	note = {Online first},
	title = {Global-in-time vortex configurations for the 2D {Euler} equations},
	year = {2026}}

@book{WMZ15,
	address = {Berlin, Heidelberg},
	author = {J.-Z. Wu and H.-Y. Ma and M.-D. Zhou},
	doi = {10.1007/978-3-662-47061-9},
	publisher = {Springer},
	title = {Vortical Flows},
	year = {2015}}

@article{Pro24,
	author = {B. Protas},
	doi = {10.1017/jfm.2024.160},
	journal = {J. Fluid Mech.},
	pages = {A7},
	title = {On the linear stability of the {Lamb--Chaplygin} dipole},
	volume = {984},
	year = {2024}}

@article{BCK26,
	archiveprefix = {arXiv},
	author = {E. Bru{\`e} and M. Colombo and A. Kumar},
	doi = {10.1215/00127094-2025-0056},
	eprint = {2408.07934},
	journal = {Duke Math. J.},
	number = {9},
	pages = {1593--1647},
	title = {Flexibility of two-dimensional {Euler} flows with integrable vorticity},
	url = {https://arxiv.org/abs/2408.07934},
	volume = {175},
	year = {2026}}

@article{CDG21,
	author = {P. Constantin and T. D. Drivas and D. Ginsberg},
	journal = {Comm. Math. Phys.},
	number = {1},
	pages = {521--563},
	title = {Flexibility and rigidity in steady fluid motion},
	volume = {385},
	year = {2021}}

@article{GPSY21,
	author = {J. G{\'o}mez-Serrano and J. Park and J. Shi and Y. Yao},
	journal = {Duke Math. J.},
	number = {13},
	pages = {2957--3038},
	title = {Symmetry in stationary and uniformly rotating solutions of active scalar equations},
	volume = {170},
	year = {2021}}

@article{Ruiz23,
	author = {D. Ruiz},
	journal = {Arch. Ration. Mech. Anal.},
	number = {3},
	pages = {Paper No. 40},
	title = {Symmetry results for compactly supported steady solutions of the 2D {Euler} equations},
	volume = {247},
	year = {2023}}

@article{DG24,
	author = {T. D. Drivas and D. Ginsberg},
	journal = {Proc. Amer. Math. Soc.},
	number = {11},
	pages = {4855--4863},
	title = {Islands in stable fluid equilibria},
	volume = {152},
	year = {2024}}

@article{EFR24,
	author = {A. Enciso and A. J. Fern{\'a}ndez and D. Ruiz},
	journal = {J. Eur. Math. Soc.},
	note = {\href{https://arxiv.org/abs/2406.04414}{arXiv:2406.04414}},
	title = {Smooth nonradial stationary {Euler} flows on the plane with compact support},
	year = {to appear}}

@article{EHSX26,
	author = {T. M. Elgindi and Y. Huang and A. R. Said and C. Xie},
	journal = {Duke Math. J.},
	note = {\href{https://arxiv.org/abs/2408.14662}{arXiv:2408.14662}},
	title = {A classification theorem for steady {Euler} flows},
	year = {to appear}}

@article{FT81,
	author = {Friedman, A. and Turkington, B.},
	doi = {10.1090/S0002-9947-1981-0628444-6},
	journal = {Trans. Amer. Math. Soc.},
	number = {1},
	pages = {1--37},
	title = {Vortex rings: Existence and asymptotic estimates},
	volume = {268},
	year = {1981}}

@article{Fra70,
	author = {Fraenkel, L. E.},
	doi = {10.1098/rspa.1970.0065},
	journal = {Proc. Roy. Soc. London Ser. A},
	pages = {29--62},
	title = {On steady vortex rings of small cross-section in an ideal fluid},
	volume = {316},
	year = {1970}}

@article{No73,
	author = {Norbury, J.},
	doi = {10.1017/S0022112073001266},
	journal = {J. Fluid Mech.},
	number = {3},
	pages = {417--431},
	title = {A family of steady vortex rings},
	volume = {57},
	year = {1973}}

@article{CLQZZ26,
	author = {Cao, D. and Lai, S. and Qin, G. and Zhan, W. and Zou, C.},
	doi = {10.1007/s00208-026-03352-5},
	journal = {Math. Ann.},
	pages = {9},
	title = {Uniqueness and stability of steady vortex rings for {3D} incompressible {Euler} equation},
	volume = {394},
	year = {2026}}

@book{LL01,
	address = {Providence, RI},
	author = {Lieb, E. H. and Loss, M.},
	edition = {2},
	publisher = {American Mathematical Society},
	series = {Graduate Studies in Mathematics},
	title = {Analysis},
	volume = {14},
	year = {2001}}

@article{CS12,
	author = {Choffrut, A. and {\v{S}}ver{\'a}k, V.},
	doi = {10.1007/s00039-012-0149-6},
	journal = {Geom. Funct. Anal.},
	number = {1},
	pages = {136--201},
	title = {Local Structure of the Set of Steady-State Solutions to the 2D Incompressible Euler Equations},
	volume = {22},
	year = {2012}}

@incollection{BCG26,
	address = {Singapore},
	author = {Bailo, R. and Carrillo, J. A. and G{\'o}mez-Castro, D.},
	booktitle = {Recent Developments in Industrial and Applied Mathematics},
	doi = {10.1007/978-981-95-1446-5_9},
	pages = {177--200},
	publisher = {Springer},
	title = {Aggregation-Diffusion Equations for Collective Behaviour in the Sciences},
	volume = {1},
	year = {2026}}

@article{CHVY19,
	author = {Carrillo, J. A. and Hittmeir, S. and Volzone, B. and Yao, Y.},
	doi = {10.1007/s00222-019-00898-x},
	journal = {Invent. Math.},
	number = {3},
	pages = {889--977},
	title = {Nonlinear aggregation-diffusion equations: radial symmetry and long time asymptotics},
	volume = {218},
	year = {2019}}

@article{BNL13,
	author = {Burton, G. R. and Nussenzveig Lopes, H. J. and Lopes Filho, M. C.},
	doi = {10.1007/s00220-013-1806-y},
	fjournal = {Communications in Mathematical Physics},
	issn = {0010-3616,1432-0916},
	journal = {Comm. Math. Phys.},
	mrclass = {35Q30 (35B35 49K10)},
	mrnumber = {3117517},
	mrreviewer = {John\ Albert},
	number = {2},
	pages = {445--463},
	title = {Nonlinear stability for steady vortex pairs},
	url = {https://doi.org/10.1007/s00220-013-1806-y},
	volume = {324},
	year = {2013}}

@article{Burton21,
	author = {Burton, G. R.},
	doi = {10.1016/j.jde.2020.08.009},
	fjournal = {Journal of Differential Equations},
	issn = {0022-0396,1090-2732},
	journal = {J. Differential Equations},
	mrclass = {76B47 (35Q35 76E07)},
	mrnumber = {4150383},
	mrreviewer = {Michael\ J.\ Carley},
	pages = {547--572},
	title = {Compactness and stability for planar vortex-pairs with prescribed impulse},
	url = {https://doi.org/10.1016/j.jde.2020.08.009},
	volume = {270},
	year = {2021}}

@article{Burton96,
	author = {Burton, G. R.},
	journal = {Proc. Roy. Soc. London Ser. A},
	pages = {2343--2350},
	title = {Uniqueness for the circular vortex-pair in a uniform flow},
	volume = {452},
	year = {1996}}

@article{Burton05b,
	author = {Burton, G. R.},
	doi = {10.1007/s00021-004-0126-6},
	journal = {J. Math. Fluid Mech.},
	pages = {S68--S80},
	title = {Isoperimetric properties of {L}amb's circular vortex-pair},
	volume = {7},
	year = {2005}}

@article{Burton88,
	author = {Burton, G. R.},
	journal = {Proc. Roy. Soc. Edinburgh Sect. A},
	pages = {269--290},
	title = {Steady symmetric vortex pairs and rearrangements},
	volume = {108},
	year = {1988}}

@article{Norbury75,
	author = {Norbury, J.},
	doi = {10.1002/cpa.3160280602},
	journal = {Comm. Pure Appl. Math.},
	pages = {679--700},
	title = {Steady planar vortex pairs in an ideal fluid},
	volume = {28},
	year = {1975}}

@unpublished{CJY,
	archiveprefix = {arXiv},
	author = {K. Choi and I.-J. Jeong and Y. Yao},
	eprint = {2409.19822},
	note = {\href{https://arxiv.org/abs/2409.19822}{arXiv:2409.19822}},
	primaryclass = {math.AP},
	title = {Stability of vortex quadrupoles with odd-odd symmetry},
	url = {https://arxiv.org/abs/2409.19822}}

@unpublished{ACJ,
	archiveprefix = {arXiv},
	author = {K. Abe and K. Choi and I.-J. Jeong},
	eprint = {2510.00539},
	note = {\href{https://arxiv.org/abs/2510.00539}{arXiv:2510.00539}},
	primaryclass = {math.AP},
	title = {Stability of {L}amb dipoles for odd-symmetric and non-negative initial disturbances without the finite mass condition},
	url = {https://arxiv.org/abs/2510.00539}}

@unpublished{JYZ,
	archiveprefix = {arXiv},
	author = {I.-J. Jeong and Y. Yao and T. Zhou},
	eprint = {2507.15739},
	note = {\href{https://arxiv.org/abs/2507.15739}{arXiv:2507.15739}},
	primaryclass = {math.AP},
	title = {Superlinear gradient growth for ${2D}$ {E}uler equation without boundary},
	url = {https://arxiv.org/abs/2507.15739}}

@unpublished{CSW,
	archiveprefix = {arXiv},
	author = {K. Choi and Y.-J. Sim and K. Woo},
	eprint = {2507.00910},
	note = {\href{https://arxiv.org/abs/2507.00910v1}{arXiv:2507.00910v1}},
	primaryclass = {math.AP},
	title = {Existence and stability of {Sadovskii} vortices: from vortex patches to regular vorticity}}

@unpublished{ACJSW,
	archiveprefix = {arXiv},
	author = {K. Abe and K. Choi and I.-J. Jeong and Y.-J. Sim and K. Woo},
	eprint = {2507.00910},
	note = {\href{https://arxiv.org/abs/2507.00910v2}{arXiv:2507.00910v2}},
	primaryclass = {math.AP},
	title = {Existence and stability of {Sadovskii} vortices: from patch to smooth vortices},
	url = {https://arxiv.org/abs/2507.00910v2}}

@article{VV98,
	author = {van Geffen, J. H. G. M. and van Heijst, G. J. F.},
	doi = {10.1016/S0169-5983(97)00033-6},
	journal = {Fluid Dynamics Research},
	number = {4},
	pages = {191--213},
	title = {Viscous evolution of {2D} dipolar vortices},
	volume = {22},
	year = {1998}}

@article{FV94,
	author = {Fl{\'o}r, J. B. and van Heijst, G. J. F.},
	journal = {J. Fluid Mech.},
	pages = {101--133},
	title = {An experimental study of dipolar vortex structures in a stratified fluid},
	volume = {279},
	year = {1994}}

@article{VF89,
	author = {van Heijst, G. J. F. and Fl{\'o}r, J. B.},
	journal = {Nature},
	pages = {212--215},
	title = {Dipole formation and collisions in a stratified fluid},
	volume = {340},
	year = {1989}}

@article{Afan,
	author = {Afanasyev, Y. D.},
	doi = {10.1063/1.2182006},
	eprint = {https://pubs.aip.org/aip/pof/article-pdf/doi/10.1063/1.2182006/14830962/037103_1_online.pdf},
	issn = {1070-6631},
	journal = {Physics of Fluids},
	month = {03},
	number = {3},
	pages = {037103},
	title = {Formation of vortex dipoles},
	url = {https://doi.org/10.1063/1.2182006},
	volume = {18},
	year = {2006}}

@article{VAF,
	author = {Voropayev, S. I. and Afanasyev, Y. D. and Filippov, I. A.},
	journal = {J. Fluid Mech.},
	pages = {543--566},
	title = {Horizontal jets and vortex dipoles in a stratified fluid},
	volume = {227},
	year = {1991}}

@book{Lamb,
	author = {Lamb, H.},
	edition = {Third},
	publisher = {Cambridge Univ. Press.},
	title = {Hydrodynamics},
	year = {1906}}

@article{Chap1903,
	author = {Chaplygin, S. A.},
	journal = {Trudy Otd. Fiz. Nauk Imper. Mosk. Obshch. Lyub. Estest.},
	number = {2},
	pages = {11--14},
	title = {One Case of Vortex Motion in Fluid},
	volume = {11},
	year = {1903}}

@article{MV94,
	author = {Meleshko, V. V. and van Heijst, G. J. F.},
	journal = {J. Fluid Mech.},
	pages = {157--182},
	title = {On {C}haplygin's investigations of two-dimensional vortex structures in an inviscid fluid},
	volume = {272},
	year = {1994}}

@article{CLZ21,
	author = {Cao, D. and Lai, S. and Zhan, W.},
	doi = {10.1007/s00526-021-02068-5},
	fjournal = {Calculus of Variations and Partial Differential Equations},
	issn = {0944-2669},
	journal = {Calc. Var. Partial Differential Equations},
	mrclass = {35J60 (35Q31 76B47)},
	mrnumber = {4295232},
	number = {5},
	pages = {Paper No. 190, 16},
	title = {Traveling vortex pairs for 2{D} incompressible {E}uler equations},
	url = {https://doi.org/10.1007/s00526-021-02068-5},
	volume = {60},
	year = {2021}}

@article{CQZZ23,
	author = {Cao, D. and Qin, G. and Zhan, W. and Zou, C.},
	doi = {10.1090/tran/8888},
	fjournal = {Transactions of the American Mathematical Society},
	issn = {0002-9947,1088-6850},
	journal = {Trans. Amer. Math. Soc.},
	mrclass = {76B47 (35A15 35J20)},
	mrnumber = {4577334},
	mrreviewer = {Tomasz\ Cie\'slak},
	number = {5},
	pages = {3377--3395},
	title = {Remarks on orbital stability of steady vortex rings},
	url = {https://doi.org/10.1090/tran/8888},
	volume = {376},
	year = {2023}}

@article{CQZZ25,
	author = {Cao, D. and Qin, G. and Zhan, W. and Zou, C.},
	doi = {10.1007/s40818-024-00191-y},
	fjournal = {Annals of PDE. Journal Dedicated to the Analysis of Problems from Physical Sciences},
	issn = {2524-5317},
	journal = {Ann. PDE},
	mrclass = {35J65 (35J20 35Q31 76B03 76B47)},
	mrnumber = {4842908},
	number = {1},
	pages = {Paper No. 1, 55},
	title = {Uniqueness and stability of traveling vortex pairs for the incompressible {E}uler equation},
	url = {https://doi.org/10.1007/s40818-024-00191-y},
	volume = {11},
	year = {2025}}

@article{Tu83a,
	author = {Turkington, B.},
	doi = {10.1080/03605308308820293},
	journal = {Comm. Partial Differential Equations},
	number = {9},
	pages = {999--1030},
	title = {On steady vortex flow in two dimensions. {I}},
	volume = {8},
	year = {1983}}

@article{Tu83b,
	author = {Turkington, B.},
	doi = {10.1080/03605308308820294},
	journal = {Comm. Partial Differential Equations},
	number = {9},
	pages = {1031--1071},
	title = {On steady vortex flow in two dimensions. {II}},
	volume = {8},
	year = {1983}}

@article{FB74,
	author = {Fraenkel, L. E. and Berger, M. S.},
	journal = {Acta Math.},
	pages = {13--51},
	title = {A global theory of steady vortex rings in an ideal fluid},
	volume = {132},
	year = {1974}}

@article{Wang.2024,
	author = {Wang, G.},
	doi = {10.1090/tran/9105},
	fjournal = {Transactions of the American Mathematical Society},
	issn = {0002-9947,1088-6850},
	journal = {Trans. Amer. Math. Soc.},
	mrclass = {35Q35 (35C07 76B47 76E30)},
	mrnumber = {4744767},
	number = {4},
	pages = {2635--2661},
	title = {On concentrated traveling vortex pairs with prescribed impulse},
	url = {https://doi.org/10.1090/tran/9105},
	volume = {377},
	year = {2024}}

@article{Car,
	author = {Carrillo, J. A. and Castorina, D. and Volzone, B.},
	doi = {10.1137/140951588},
	fjournal = {SIAM Journal on Mathematical Analysis},
	issn = {0036-1410,1095-7154},
	journal = {SIAM J. Math. Anal.},
	mrclass = {35K59 (35A15 35B40 82B05 92C17)},
	mrnumber = {3296600},
	mrreviewer = {Daniel\ Matthes},
	number = {1},
	pages = {1--25},
	title = {Ground states for diffusion dominated free energies with logarithmic interaction},
	url = {https://doi.org/10.1137/140951588},
	volume = {47},
	year = {2015}}

@article{AF86,
	author = {Amick, C. J. and Fraenkel, L. E.},
	journal = {Arch. Rational Mech. Anal.},
	pages = {91--119},
	title = {The uniqueness of {H}ill's spherical vortex},
	volume = {92},
	year = {1986}}

@article{AF88,
	author = {Amick, C. J. and Fraenkel, L. E.},
	journal = {Arch. Rational Mech. Anal.},
	pages = {207--241},
	title = {The uniqueness of a family of steady vortex rings},
	volume = {100},
	year = {1988}}

@article{DG26,
	author = {Dolce, M. and Gallay, T.},
	doi = {10.1007/s00205-026-02169-5},
	fjournal = {Archive for Rational Mechanics and Analysis},
	issn = {0003-9527,1432-0673},
	journal = {Arch. Ration. Mech. Anal.},
	mrclass = {76D17 (35B35 35Q30)},
	mrnumber = {5048483},
	number = {2},
	pages = {Paper No. 18, 63},
	title = {The long way of a viscous vortex dipole},
	url = {https://doi.org/10.1007/s00205-026-02169-5},
	volume = {250},
	year = {2026}}

@unpublished{LSZ26,
	archiveprefix = {arXiv},
	author = {Z. Li and P. Song and T. Zhou},
	eprint = {2605.01491},
	note = {\href{https://arxiv.org/abs/2605.01491}{arXiv:2605.01491}},
	primaryclass = {math.AP},
	title = {On the stability of the {Lamb--Chaplygin} dipole for the {2D Euler} equation},
	url = {https://arxiv.org/abs/2605.01491}}

@unpublished{PV,
	archiveprefix = {arXiv},
	author = {F. Pio Numero and P. Ventura},
	eprint = {2606.09775},
	note = {\href{https://arxiv.org/abs/2606.09775}{arXiv:2606.09775}},
	primaryclass = {math.AP},
	title = {Linear Stability of the {L}amb-{C}haplygin Dipole}}

@article{Kelvin1875,
	author = {Thomson, W.},
	doi = {10.1017/S0370164600031679},
	journal = {Proc. Roy. Soc. Edinburgh},
	pages = {59--73},
	title = {Vortex statics},
	volume = {9},
	year = {1878}}

@article{Arnold66,
	author = {Arnold, V. I.},
	journal = {Ann. Inst. Fourier (Grenoble)},
	pages = {319--361},
	title = {Sur la g\'{e}om\'{e}trie diff\'{e}rentielle des groupes de {L}ie de dimension infinie et ses applications \`a l'hydrodynamique des fluides parfaits},
	volume = {16},
	year = {1966}}

@article{HaHm21,
	author = {Hassainia, Z. and Hmidi, T.},
	doi = {10.3934/dcds.2020348},
	fjournal = {Discrete and Continuous Dynamical Systems. Series A},
	issn = {1078-0947},
	journal = {Discrete Contin. Dyn. Syst.},
	mrclass = {35Q35 (35Q31 76B47)},
	mrnumber = {4211209},
	number = {4},
	pages = {1939--1969},
	title = {Steady asymmetric vortex pairs for {E}uler equations},
	url = {https://doi.org/10.3934/dcds.2020348},
	volume = {41},
	year = {2021}}

@article{CLW14,
	author = {Cao, D. and Liu, Z. and Wei, J.},
	doi = {10.1007/s00205-013-0692-y},
	fjournal = {Archive for Rational Mechanics and Analysis},
	issn = {0003-9527},
	journal = {Arch. Ration. Mech. Anal.},
	mrclass = {35Q31 (35A09 35B25 35B65 35J25 76B03 76B47)},
	mrnumber = {3162476},
	mrreviewer = {Francesco Fanelli},
	number = {1},
	pages = {179--217},
	title = {Regularization of point vortices pairs for the {E}uler equation in dimension two},
	url = {https://doi.org/10.1007/s00205-013-0692-y},
	volume = {212},
	year = {2014}}

@article{SV10,
	author = {Smets, D. and Van Schaftingen, J.},
	doi = {10.1007/s00205-010-0293-y},
	fjournal = {Archive for Rational Mechanics and Analysis},
	issn = {0003-9527,1432-0673},
	journal = {Arch. Ration. Mech. Anal.},
	mrclass = {35Q31 (35B25 35J91 76B03 76B47)},
	mrnumber = {2729322},
	mrreviewer = {Stefano\ Bianchini},
	number = {3},
	pages = {869--925},
	title = {Desingularization of vortices for the {E}uler equation},
	url = {https://doi.org/10.1007/s00205-010-0293-y},
	volume = {198},
	year = {2010}}

@article{HT25,
	author = {Huang, D. and Tong, J.},
	doi = {10.1007/s00205-025-02113-z},
	fjournal = {Archive for Rational Mechanics and Analysis},
	issn = {0003-9527,1432-0673},
	journal = {Arch. Ration. Mech. Anal.},
	mrclass = {76B47},
	mrnumber = {4933909},
	number = {4},
	pages = {Paper No. 46, 52},
	title = {Steady contiguous vortex-patch dipole solutions of the 2{D} incompressible {E}uler equation},
	url = {https://doi.org/10.1007/s00205-025-02113-z},
	volume = {249},
	year = {2025}}

@incollection{Benjamin76,
	address = {Berlin},
	author = {Benjamin, T. B.},
	booktitle = {Applications of Methods of Functional Analysis to Problems in Mechanics},
	doi = {10.1007/BFb0088744},
	pages = {8--29},
	publisher = {Springer},
	series = {Lecture Notes in Mathematics},
	title = {The alliance of practical and analytical insights into the nonlinear problems of fluid mechanics},
	url = {https://doi.org/10.1007/BFb0088744},
	volume = {503},
	year = {1976}}

@article{Pierrehumbert80,
	author = {Pierrehumbert, R. T.},
	journal = {J. Fluid Mech.},
	number = {1},
	pages = {129--144},
	title = {A family of steady, translating vortex pairs with distributed vorticity},
	volume = {99},
	year = {1980}}

@article{Leweke16,
	author = {Leweke, T. and Le Diz\`es, S. and Williamson, C. H. K.},
	journal = {Annu. Rev. Fluid Mech.},
	pages = {507--541},
	title = {Dynamics and instabilities of vortex pairs},
	volume = {48},
	year = {2016}}

@article{GarciaHaziot23,
	author = {Garc\'ia, C. and Haziot, S. V.},
	doi = {10.1007/s00220-023-04741-6},
	journal = {Comm. Math. Phys.},
	number = {2},
	pages = {1167--1204},
	title = {Global bifurcation for corotating and counter-rotating vortex pairs},
	url = {https://doi.org/10.1007/s00220-023-04741-6},
	volume = {402},
	year = {2023}}

@article{JangSeok22,
	author = {Jang, J. and Seok, J.},
	doi = {10.1007/s00205-022-01766-4},
	journal = {Arch. Ration. Mech. Anal.},
	number = {2},
	pages = {443--499},
	title = {On uniformly rotating binary stars and galaxies},
	url = {https://doi.org/10.1007/s00205-022-01766-4},
	volume = {244},
	year = {2022}}

@book{Evans2010,
	address = {Providence, RI},
	author = {Evans, L. C.},
	edition = {2},
	publisher = {American Mathematical Society},
	series = {Graduate Studies in Mathematics},
	title = {Partial Differential Equations},
	volume = {19},
	year = {2010}}

@article{Tang,
	author = {Tang, Y.},
	doi = {10.2307/2000733},
	fjournal = {Transactions of the American Mathematical Society},
	issn = {0002-9947,1088-6850},
	journal = {Trans. Amer. Math. Soc.},
	mrclass = {76C05 (35B35 35Q10)},
	mrnumber = {911087},
	mrreviewer = {John\ Adam},
	number = {2},
	pages = {617--638},
	title = {Nonlinear stability of vortex patches},
	url = {https://doi.org/10.2307/2000733},
	volume = {304},
	year = {1987}}

@article{WP,
	author = {Wan, Y. H. and Pulvirenti, M.},
	fjournal = {Communications in Mathematical Physics},
	issn = {0010-3616,1432-0916},
	journal = {Comm. Math. Phys.},
	mrclass = {76E30 (35Q10 58E07 76C05)},
	mrnumber = {795112},
	mrreviewer = {Jacob\ Burbea},
	number = {3},
	pages = {435--450},
	title = {Nonlinear stability of circular vortex patches},
	url = {http://projecteuclid.org/euclid.cmp/1103942770},
	volume = {99},
	year = {1985}}

@article{AJY,
	archiveprefix = {arXiv},
	author = {K. Abe and I.-J. Jeong and Y. Yao},
	eprint = {2507.16474},
	journal = {Courant J. Pure Appl. Math.},
	note = {\href{https://arxiv.org/abs/2507.16474}{arXiv:2507.16474}},
	primaryclass = {math.AP},
	title = {Stability for multiple {L}amb dipoles},
	url = {https://arxiv.org/abs/2507.16474},
	year = {to appear}}

@article{AC2019,
	author = {Abe, K. and Choi, K.},
	doi = {10.1007/s00205-022-01782-4},
	fjournal = {Archive for Rational Mechanics and Analysis},
	issn = {0003-9527,1432-0673},
	journal = {Arch. Ration. Mech. Anal.},
	mrclass = {35Q31 (76B47)},
	mrnumber = {4419609},
	number = {3},
	pages = {877--917},
	title = {Stability of {L}amb dipoles},
	url = {https://doi.org/10.1007/s00205-022-01782-4},
	volume = {244},
	year = {2022}}

@article{Choi24,
	author = {Choi, K.},
	fjournal = {Communications on Pure and Applied Mathematics},
	issn = {0010-3640,1097-0312},
	journal = {Comm. Pure Appl. Math.},
	mrclass = {76B47},
	mrnumber = {4666623},
	number = {1},
	pages = {52--138},
	title = {Stability of {H}ill's spherical vortex},
	volume = {77},
	year = {2024}}

@article{CJ-Lamb,
	author = {Choi, K. and Jeong, I.-J.},
	doi = {10.1016/j.nonrwa.2021.103470},
	fjournal = {Nonlinear Analysis. Real World Applications. An International Multidisciplinary Journal},
	issn = {1468-1218,1878-5719},
	journal = {Nonlinear Anal. Real World Appl.},
	mrclass = {35Q31 (76B47)},
	mrnumber = {4350517},
	pages = {Paper No. 103470, 20},
	title = {Infinite growth in vorticity gradient of compactly supported planar vorticity near {L}amb dipole},
	url = {https://doi.org/10.1016/j.nonrwa.2021.103470},
	volume = {65},
	year = {2022}}

@article{CJS,
	author = {Choi, K. and Jeong, I.-J. and Sim, Y.-J.},
	doi = {10.1007/s40818-025-00212-4},
	fjournal = {Annals of PDE. Journal Dedicated to the Analysis of Problems from Physical Sciences},
	issn = {2524-5317,2199-2576},
	journal = {Ann. PDE},
	mrclass = {35C07 (35A01 35Q31)},
	mrnumber = {4927140},
	mrreviewer = {Da-Wen\ Deng},
	number = {2},
	pages = {Paper No. 18, 68},
	title = {On existence of {S}adovskii vortex patch: a touching pair of symmetric counter-rotating uniform vortices},
	url = {https://doi.org/10.1007/s40818-025-00212-4},
	volume = {11},
	year = {2025}}

\endgroup 

\end{document}